\documentclass[reqno,a4paper,10pt]{amsart}

\usepackage[utf8]{inputenc}
\usepackage{graphicx}
\usepackage{subfigure}
\usepackage{amsmath, amssymb, amsthm}
\usepackage{enumitem}

\usepackage{graphicx}
\usepackage{array}
\usepackage{tikz}
\usepackage{mathrsfs}
\usetikzlibrary{arrows}
\usetikzlibrary{shapes}
\usepackage{cleveref}

\renewcommand{\theenumi}{(\roman{enumi})}

\newcommand{\ol}[1]{\overline{#1}}
\newcommand{\br}[1]{\left(#1\right)}
\newcommand{\torus}[1]{\mathbb{S}_{#1}}
\newcommand{\Sg}{\mathbb{S}_g}
\newcommand{\Z}{\mathbb{Z}}
\newcommand{\N}{\mathbb{N}}
\newcommand{\R}{\mathbb{R}}
\newcommand{\C}{\mathbb{C}}
\newcommand{\opa}[1]{\delta_#1}
\newcommand{\opb}[2]{\left(\delta_#1+\delta_#2\right)}
\DeclareMathOperator{\fw}{fw}
\DeclareMathOperator{\ew}{ew}
\newcommand{\class}[1]{\mathcal{#1}}
\newcommand{\vertex}[1]{#1}
\newcommand{\deriv}[1]{\frac{\partial}{\partial #1}}
\newcommand{\derivf}[2]{\frac{\partial #1}{\partial #2}}
\newcommand{\derivn}[2]{\frac{\partial^{#2}}{\partial #1^{#2}}}

\newcommand{\triangt}{\hat{S}}
\def\triangs{S}
\def\triangr{R}
\def\triangn{N}
\def\triangm{M}

\def\three{D}
\def\two{B}
\def\one{C}
\def\general{G}
\def\network{N}
\def\quasis{P}
\def\main{M}
\def\error{E}

\def\dd#1{\,\mathrm{d}{#1}}

\def\oy{\overline{y}}
\def\oz{\overline{z}}
\def\Gimenez{Gim{\'e}nez\;}

\definecolor{gray}{gray}{0.9}

\newtheorem{thm}{Theorem}[section]
\newtheorem{prop}[thm]{Proposition}
\newtheorem{coro}[thm]{Corollary}
\newtheorem{lem}[thm]{Lemma}
\newtheorem{definition}[thm]{Definition}

\crefname{thm}{theorem}{theorems}
\crefname{prop}{proposition}{propositions}
\crefname{coro}{corollary}{corollaries}
\crefname{lem}{lemma}{lemmas}
\crefname{definition}{definition}{definitions}

\newtheoremstyle{claim}
{}
{}
{\itshape}
{}
{\bf}
{.}
{.5em}
{}
\theoremstyle{claim}
\newtheorem{claim}{Claim}

\crefname{claim}{claim}{claims}
\begin{document}
\title[Cubic graphs and related triangulations on orientable surfaces]{Cubic graphs and related triangulations on orientable surfaces\textsuperscript{1}}
\thanks{\textsuperscript{1}An extended abstract of this paper has been published in  the Proceedings of the European Conference on Combinatorics, Graph Theory and Applications (EuroComb15), Electronic Notes in Discrete Mathematics (2015), 603--610.}
\author[W. Fang]{W. Fang\textsuperscript{2}}
\thanks{\textsuperscript{2}LIAFA, UMR CNRS 7089, Universit{\'e} Paris-Diderot, Paris Cedex 13, France. email: {\tt wfang@liafa.univ-paris-diderot.fr}. Partially funded by \textit{Agence Nationale de la Recherche}, grant ANR 12-JS02-001-01 "Cartaplus"}
\author{M. Kang\textsuperscript{3}}
\author{M. Mo{\ss}hammer\textsuperscript{3}}
\author{P. Spr{\" u}ssel\textsuperscript{3}}
\thanks{\textsuperscript{3}Graz University of Technology, Institute of Discrete Mathematics,
8010 Graz, Austria. email: {\tt \{kang,mosshammer,spruessel\}@math.tugraz.at}. Supported by Austrian Science Fund (FWF):  P27290 and W1230 II}

\begin{abstract}
	Let $\Sg$ be the orientable surface of genus $g$.
	We show that the number of vertex-labelled cubic multigraphs embeddable on $\Sg$ with $2n$ vertices is 
	asymptotically $c_g n^{5(g-1)/2-1}\gamma^{2n}(2n)!$, where $\gamma$ is an algebraic constant and $c_g$ 
	is a constant depending only on the genus $g$. We also derive an analogous result for simple cubic graphs and 
	weighted cubic multigraphs. Additionally we prove that a typical cubic multigraph embeddable 
	on $\Sg$, $g\ge 1$, has exactly one non-planar component.
\end{abstract}
\maketitle

	\section{Introduction}

	Determining the numbers of maps and graphs \emph{embeddable} on surfaces have been one of the main  
	objectives of enumerative combinatorics for the last 50 years.
	Starting from the enumeration of \emph{planar} maps by Tutte \cite{tutte1963-planar-maps} various types
	of maps on the sphere were counted, e.g.\ planar \emph{cubic} maps by Gao and Wormald \cite{MR1980342}.
	Furthermore, Tutte's methods were generalised to enumerate maps on surfaces of higher genus 
	\cite{BENDER1986244,Bender1993-maps,Bender1988370}.
	
	An important subclass of maps are \emph{triangulations}. 
	Brown \cite{triang-disk} determined the number of triangulations of a disc and Tutte enumerated planar 
	triangulations \cite{tutte1962-planar-triangulations}. Triangulations on other surfaces have 
	been considered as well: Gao enumerated $2$-connected triangulations on 
	the projective plane \cite{Gao1991-2-connected-projective} as well as $2$-connected and connected triangulations 
	on surfaces of arbitrary genus \cite{Gao1992-2-connected-surface,Gao1993-pattern}.
	
	Frieze \cite{seminaronrg} was the first to ask about properties of random planar \emph{graphs}.
	McDiarmid, Steger, and Welsh \cite{mcdiarmid2005} showed the existence of an exponential growth 
	constant for the number of labelled planar graphs on $n$ vertices. This growth constant and the asymptotic 
	number were determined by \Gimenez and Noy \cite{gimenez2009}, while the corresponding results for the 
	higher genus case were derived by Chapuy et al.\ \cite{Chapuy2011-enumeration-graphs-on-surfaces}.
	Since then various other classes of planar graphs were 
	counted \cite{bender,MR2387555,bodirsky2007-cubic-graphs,rgraphs,Kang2012,ks,osthus}.

	An interesting subclass of planar graphs are \emph{cubic} planar graphs, which have been enumerated by
	Bodirsky et al.\ \cite{bodirsky2007-cubic-graphs}. Cubic planar graphs occur as substructures of sparse 
	planar graphs and have thus been one of the essential ingredients in the study of sparse random 
	planar graphs \cite{Kang2012}. For surfaces of higher genus, the number of embeddable 
	cubic graphs has not been studied, and is the main subject of this paper.

\subsection{Main results}
 
The main contributions of this paper are fourfold.
We determine the asymptotic number of cubic multigraphs embeddable on $\Sg$, the orientable 
surface of genus $g$. Using the same method we also determine the asymptotic number of weighted cubic multigraphs 
and cubic simple graphs embeddable on $\Sg$. Finally we prove that almost all (multi)graphs from either
of the three classes have exactly one non-planar component.

The first main result is about the exact asymptotic expression of the number of cubic multigraphs embeddable on $\Sg$.
\begin{thm}\label{main1}
	The number $m_g(n)$ of vertex-labelled cubic multigraphs embeddable on $\Sg$ with $2n$ vertices is given by
	\begin{align*}
		m_g(n)&=\br{1+O\br{n^{-1/4}}}d_gn^{5g/2-7/2}\gamma_1^{2n}(2n)!\, ,
	\end{align*} where $\gamma_1$ 
	is an algebraic constant independent of the genug $g$ and $d_g$ is a constant 
	depending only on $g$. The first digits of $\gamma_1$ are $3.973$.
\end{thm}

Our next main result concerns multigraphs weighted with the so-called \emph{compensation factor} introduced by
Janson et al.\ \cite{jansonea}. 
This factor is defined as the number of ways to orient and order all edges of the multigraph divided 
by $2^mm!$, which is equal to the number of such oriented orderings if all edges were distinguishable. For example,
a double edge results in a factor $\frac{1}{2}$ and simple graphs are the only multigraphs with compensation factor one. 

\begin{thm}\label{main2}
	The number $w_g(n)$ of vertex-labelled cubic multigraphs embeddable 
	on $\Sg$ with $2n$ vertices weighted by their compensation factor is given by
	\begin{align*}
		w_g(n)&=\br{1+O\br{n^{-1/4}}}e_g n^{5g/2-7/2}\gamma_2^{2n}(2n)!\,,
	\end{align*}
	 where $\gamma_2=\frac{79^{3/4}}{54^{1/2}}\approx 3.606$ 
	 and $e_g$ is a constant depending only on the genus $g$.
\end{thm}

\Cref{main2} can be used to derive the asymptotic number and structural properties of graphs embeddable on $\Sg$ \cite{prep}.
\emph{Planar} cubic multigraphs weighted by the compensation factor were counted by 
Kang and \L{}uczak \cite{Kang2012}. The 
discrepancy to their exponential growth constant $\gamma\approx 3.38$ is due to incorrect initial conditions, 
as pointed out by Noy et al.\ \cite{nrr}. While the explicit value of the correct 
exponential growth constant $\gamma$ was not determined in \cite{nrr}, the implicit 
equations given there yield the same exponential growth constant $\gamma_2$ as in \Cref{main2}.

Our methods also allow us to count cubic \emph{simple} graphs (graphs without loops and multiedges) embeddable on $\Sg$.

\begin{thm}\label{main3}
	The number $s_g(n)$ of vertex-labelled cubic simple graphs embeddable on $\Sg$ with $2n$ vertices is given by
	\begin{align*}
		s_g(n)&=\br{1+O\br{n^{-1/4}}}f_g n^{5g/2-7/2}\gamma_3^{2n}(2n)!\,,
	\end{align*}
  where $\gamma_3$ 
	is an algebraic constant independent of the genus $g$ and $f_g$ is a constant 
	depending only on $g$. The first digits of $\gamma_3$ are $3.133$.
\end{thm}

The exponential growth constant $\gamma_3$ coincides with the growth constant calculated for labelled cubic planar graphs by 
Bodirsky et al.\ \cite{bodirsky2007-cubic-graphs}.

The final result describes the structure of cubic multigraphs embeddable on $\Sg$.

\begin{thm}\label{main4}
	Let $g\ge 1$ and let $G$ be a graph chosen uniformly at random from the class of vertex-labelled cubic multigraphs, 
	cubic weighted 
	multigraphs, or cubic simple graphs embeddable on $\Sg$ with $n$ vertices, respectively. Then with 
	probability $1-O(n^{-2})$, $G$  has one 
	component that is embeddable on $\Sg$ but not on $\torus{g-1}$ and all its other components are planar.
\end{thm}

\subsection{Proof techniques}

To derive our results we will use topological manipulations of surfaces (described in \Cref{section:surgeries}), constructive decomposition 
of graphs along connectivity, and singularity analysis of generating functions.

More precisely, in order to enumerate cubic multigraphs we use constructive decompositions along connectivity. 
The basic building blocks in the decomposition are $3$-connected cubic graphs which we will
then relate to their corresponding cubic \emph{maps}.
Note that due to Whitney's Theorem \cite{whitney} $3$-connected planar graphs have a
unique embedding on the sphere. Therefore this step causes no problem for the planar case.
For surfaces of positive genus, however, embeddings of $3$-connected graphs are not unique. Following an
idea from \cite{Chapuy2011-enumeration-graphs-on-surfaces}, we circumvent this problem by making use of the
concept of the \emph{facewidth} of a graph and applying results by Robertson and 
Vitray \cite{robertson1990-representativity} which relate $3$-connected graphs and maps. 

Counting $3$-connected cubic maps on $\Sg$ is a challenging task. To circumvent this challenge, we shall use their duals, 
triangulations. In fact, Gao \cite{Gao1991-2-connected-projective, Gao1992-2-connected-surface, Gao1993-pattern} enumerated 
triangulations on 
$\Sg$ with various restrictions on the containment of loops and multiedges. However, it turns out that the duals of 
$3$-connected cubic maps on $\Sg$ have very specific constraints that have not been considered by Gao. 
In this paper we investigate therefore such triangulations by relating them to simple triangulations counted by 
Gao \cite{Gao1992-2-connected-surface} (see \Cref{simplerefined,prop:M-asymptotic}). 
We strengthen Gao's results (see \Cref{simplerefined,looplessrefined}) and derive very precise singular expansions of generating functions.
This enables us to apply singularity analysis 
to the generating functions of these triangulations, as well as to the generating functions of all other classes of 
maps and graphs considered in this paper.

The rest of the paper is organised as follows. In the next section we introduce some basic notions and notations. 
In \Cref{maps} we enumerate the triangulations that occur as duals of $3$-connected cubic maps and 
in \Cref{graphs} we prove the main results (\Cref{main1,main2,main3,main4}) after giving a constructive 
decomposition along connectivity for cubic graphs embeddable on an orientable surface. Our 
strengthening of Gao's results and proofs is provided in the appendix.

	\section{Preliminaries}

A graph $G$ is called \emph{simple} if it does not contain loops or multiedges.
If in a multigraph there are more than two edges connecting the same pair of vertices, we call
each pair of those edges a \emph{double edge}. Therefore, every \emph{multiedge} consisting of $n$ edges between the same two vertices contains $\binom{n}{2}$ double edges. If $e$ is a loop, we denote by the \emph{base} of $e$ the unique vertex
$e$ is incident with. Similarly, we say that $e$ is \emph{based at} its base. An edge that is neither a loop nor part of a double edge 
is a \emph{single edge}. An edge $e$ of a connected multigraph $G$ is called a \emph{bridge} if deleting $e$ disconnects $G$.

A multigraph is called \emph{cubic} if each vertex has degree three. By $\Phi$ we denote the cubic multigraph with two
vertices $u,v$ and three edges between $u$ and $v$ (i.e.\ a triple edge). Given a connected cubic multigraph $G$, let $k$ and $l$ denote
the number of double edges and loops of $G$, respectively. We define the \emph{weight} of $G$ to be
\begin{equation*}
  W(G) =
  \begin{cases}
    \frac16 & \text{if } G = \Phi,\\
    2^{-(k+l)} & \text{otherwise}.
  \end{cases}
\end{equation*}
If $G$ is not connected, we define $W(G)$ as the product of weights of its components. For cubic multigraphs,
this weight coincides with the compensation factor introduced by Janson et al.\ \cite{jansonea}. Throughout this
paper, when we refer to a \emph{weighted cubic multigraph} $G$, the weight in consideration will be $W(G)$.

Throughout this paper let $g\ge0$ be fixed and let $\Sg$ be the orientable surface of genus $g$ with a fixed orientation.
An \emph{embedding} of a multigraph $G$ on $\Sg$ is a drawing of $G$ on $\Sg$ without crossing edges.
An embedding where additionally all faces are homeomorphic to open discs is called a \emph{$2$-cell embedding}.
Multigraphs that have an embedding are called \emph{embeddable} on $\Sg$ and multigraphs that have a $2$-cell embedding are called
\emph{strongly embeddable}. A $2$-cell embedding of a strongly embeddable multigraph is also called \emph{map}.
A \emph{triangulation} is a map where each face is bounded by a triangle. These triangles 
might be degenerate, meaning that three loops with the same base, a double edge and a loop based at one of the end vertices 
of the double edge, or a loop and an edge from the base of the loop to a vertex of degree one are considered to be 
triangles as well. If $S$ is the disjoint union of $\torus{g_1},\ldots,\torus{g_n}$ and for each $i=1,\ldots,n$, $M_i$ is a 
$2$-cell embedding of a graph $G_i$ on $\torus{g_i}$, then the induced function $N:(G_1\cup\cdots\cup G_n)\to S$ is 
called a \emph{map} on $S$. Triangulations on $S$ are defined analogously.
We denote by $V(M)$, $E(M)$, and $F(M)$ the set of all vertices, edges, and faces 
of an embedding $M$, respectively.

We obtain the \emph{dual map} of a map by taking one vertex for each face and connecting two of these vertices if the 
corresponding faces have a common edge on their boundaries. Note that the dual map has multiedges if two faces 
of the original (\emph{primal}) map have more than one edge in common. It is well known that the dual of a map is again a 
map.

For each vertex $v\in V(M)$ of a map $M$, the edges and faces incident to $v$ have a canonical cyclic order 
$e_0,f_0,e_1,f_1,\ldots,e_{d-1},f_{d-1}$ by the way they are arranged around $v$ (in counterclockwise direction). 
Note that faces can appear multiple times and  that a loop based at $v$ 
will appear twice in this sequence. To avoid ambiguities, we make the two ends of the loop distinguishable in 
this sequence (e.g. by using half-edges or by orienting each loop). A triple $(v,e_i,e_{(i+1)\bmod d})$
of a vertex and two consecutive edges in the cyclic sequence is called a \emph{corner} (at $v$). We also say that
$(v,e_i,e_{(i+1)\bmod d})$ is a corner \emph{of the face $f_i$}.
When we enumerate maps, we always work with maps with one distinguished corner, called the \emph{root} of the map. 
If $(v,e_i,e_{i+1})$ is the root corner, we will call $v$ the \emph{root vertex}, $e_i$ the \emph{root edge}, 
and $f_i$ the \emph{root face}.

\subsection{Topological operations: surgeries}\label{section:surgeries}

When dealing with maps on $\Sg$ we will perform operations on the surfaces that are commonly known as \emph{cutting}
and \emph{glueing}. In the course of these operations we will encounter \emph{surfaces with holes}. A
\emph{surface with $k$ holes} is a surface $\Sg$ 
from which the disjoint interiors $D_1,\ldots,D_k$ of $k$ closed discs have been deleted. Each $D_i$ is called 
a \emph{hole}. Let $S$ be the disjoint union of finitely many orientable surfaces, at least one of them
with holes, and suppose that $X$ and $Y$ are homeomorphic subsets of the boundary of $S$. By \emph{glueing $S$
along $X$ and $Y$} we mean the operation of identifying every point $x\in X$ with $f(x)$ for any fixed homeomorphism
$f\colon X\to Y$. The identification of $X$ and $Y$ induces a surjection $\sigma$ from $S$ onto the resulting space
$\tilde S$. We write $\tilde X$ for the subset $\sigma(X)=\sigma(Y)$ of $\tilde S$.

We will glue along subsets in two particular situations: when $X$ and $Y$ are
\begin{enumerate}
\item\label{glue:twoholes} disjoint boundaries of holes of $S$ or
\item\label{glue:zipp} sub-arcs of the boundary of the same hole that meet precisely in their endpoints.
\end{enumerate}
For \ref{glue:zipp}, we will additionally assume that the homeomorphism $f\colon X\to Y$ induces the identity on 
$X\cap Y$.
In either case, the space $\tilde S$ resulting from glueing along $X$ and $Y$ is again the disjoint union of finitely
many orientable surfaces with holes, with the number of components being either the same or one less than that for $S$. The 
subset $\tilde X$ of $\tilde S$ is a circle in Case~\ref{glue:twoholes} and homeomorphic to the closed unit interval in
Case~\ref{glue:zipp}. If $S$ has $k$ holes, then $\tilde S$ has $k-2$ holes in Case~\ref{glue:twoholes} and $k-1$ holes
in Case~\ref{glue:zipp}. A special case of~\ref{glue:twoholes} is when one of the components of $S$ is a disc (i.e.\ a
sphere with one hole) and $Y$ is its boundary. In this case, we say that we \emph{close the hole bounded by $X$ by 
inserting a disc}.

If in addition we are given a map $M$ on $S$, then we will glue along $X$ and $Y$ only if either both are contained in a
face or both are unions of the same number of vertices and the same number of edges of $M$. We also assume the
homeomorphism $f\colon X\to Y$ maps vertices to vertices and edges to edges. Under these conditions, we obtain a
map $\tilde M$ on $\tilde S$. The subset $\tilde X$ of $\tilde S$ is then either a subgraph of $\tilde M$ or
a subset of a face of $\tilde M$. Observe that the surjection $\sigma\colon S\to\tilde S$ induces a bijection
between the sets of corners of $M$ and of corners of $\tilde M$. We will refer to this bijection by saying that
every corner of $M$ \emph{corresponds} to a corner of $\tilde M$.

If $\tilde S$ is obtained from $S$ by glueing along $X$ and $Y,$ we also say that $S$ is obtained from $\tilde S$ by
\emph{cutting along $\tilde X$}. The operation of cutting along a circle or interval is well defined in the sense
that if $\tilde S$ and $\tilde X$ are given, then $S$, $X$, and $Y$ are unique up to homeomorphism. If $S$ has more components than $\tilde S$, we call $\tilde X$
\emph{separating}. Cutting along a separating circle on $\Sg$ and closing the resulting holes by inserting discs will
yield two surfaces $\torus{g_1},\torus{g_2}$ with $g_1+g_2=g$. Cutting along a non-separating circle and closing the
holes by inserting discs reduces the genus by one.

A combination of cutting and glueing surfaces along some subsets of their boundaries is called a \emph{surgery}.
Again, if a map $\tilde M$ results from performing surgeries on a map $M$, then every corner of $\tilde M$
corresponds to a corner of $M$. The following result on embeddability can be obtained by performing simple surgeries.

\begin{prop}\label{strongemb}
  Let $G$ be a multigraph.
\begin{enumerate}
	\item\label{strong:face} If $M$ is an embedding of $G$ on $\Sg$ and $f$ is a face of $M$ that is not a disc, then $f$ contains
		a circle $C$ such that cutting along $C$ results in one or two surfaces such that either two of them 
		contain a vertex of $M$ or one of them contains all vertices and edges of $M$ and has genus less than $g$.
	\item\label{strong:minimal} If $G$ is connected and $g$ is minimal such that $G$ is embeddable on $\Sg$, 
		then every embedding of $G$ on $\Sg$ is a $2$-cell embedding. In particular, $G$ is strongly embeddable 
		on $\Sg$.

	\item\label{strong:emb} $G$ is embeddable on $\Sg$ if and only if each connected component $C_i$ of $G$ is strongly 
		embeddable on a surface $\torus{g_i}$ so that $\sum_i g_i\le g$.
\end{enumerate}

\end{prop}

\begin{proof}
	For \ref{strong:face}, let $\Gamma$ be the union of all curves in $f$ that run along the boundary at a small constant
	distance $\varepsilon>0$. Suppose first that $\Gamma$ has only one component. Then this component is a circle
	$C$. Cutting along $C$ results in one surface $S_1$ that contains all vertices and edges of $M$ and one
	component $S_2$ that contains all points of $f$ with distance greater than $\varepsilon$ to the boundary. Note that $S_2$ cannot be a sphere with one hole as otherwise $f$
	would have been a disc. Thus, the genus of $S_1$ is smaller than $g$ and $C$ is the desired circle.
	
	Now suppose that $\Gamma$ has more than one component. Each component is a circle and cutting along all these
	circles results in one surface $S$ that contains all points of $f$ with distance greater than $\varepsilon$ to the boundary and surfaces $S_1,\dotsc,S_k$ which all contain
	vertices from $M$. If $k\ge2$, then there is a circle $C$ in $f$ that separates the vertices in at least two
	of these components. Otherwise, as $S$ has at least two holes, let $C$ be one of the components of $\Gamma$;
	this is a non-contractible circle on $\Sg$ that is not separating. Thus, cutting along $C$ reduces the genus
	of the surface.
	
	Claim \ref{strong:minimal} follows immediately from \ref{strong:face}. Finally, \ref{strong:emb} follows 
	by recursively cutting along circles provided by \ref{strong:face}.
\end{proof}

\subsection{Generating functions and singularity analysis}\label{section:singularity}

We will use generating functions to enumerate the various classes of maps, graphs and multigraphs we consider.
Unless stated otherwise the formal variables $x$ and $y$ will always mark vertices and edges, respectively. 
Generating functions for classes of \emph{maps} will be \emph{ordinary} unless stated otherwise. 
Generating functions for \emph{multigraphs} will be \emph{exponential} in $x$, because we consider \emph{vertex-labelled} multigraphs.  
If $\class{A}$ is a class of maps, we write $\class{A}(m)$ of $\class{A}$ for the subclass containing all maps
with exactly $m$ edges. The generating function $\sum_m |\class{A}(m)|y^m$ will be denoted by $A(y)$.
If $\class{B}$ is a class of multigraphs, we write $\class{B}(n)$ for the subclass of $\class{B}$ containing all multigraphs
with exactly $n$ vertices. The generating function $\sum_n \frac{|\class{B}(n)|}{n!}x^n$ will be denoted by $B(x)$.
For an ordinary generating function $F(z)=\sum_n f_nz^n$ we use the notation $[z^n]F(z):=f_n$.
For an exponential generating function $G(z)=\sum\frac{g_n}{n!}z^n$ we write $[z^n]G(z):=\frac{g_n}{n!}$.

If two generating functions $F(z), G(z)$ satisfy $0 \leq [z^n]F(z) \leq [z^n] G(z)$ for all $n$, we 
say that $F$ is \emph{coefficient-wise smaller} than $G$, denoted by $F \preceq G$. 
The singularities of $F(z)$ with the smallest modulus are called \emph{dominant singularities} of $F(z)$.
Because every generating function we consider in this paper always has non-negative coefficients $[z^n]F(z)$, there is a dominant singularity
located on the positive real axis by Pringsheim's Theorem \cite[pp. 214]{titchmarsh1939theory}. We denote this dominant singularity by $\rho_F$.
If an arbitrary function $F$ has a unique singularity with smallest modulus and this singularity lies on the
positive real axis, then we also denote it by $\rho_F$.
The function $F:\C\to\C$ converges on the open disc of radius $\rho_F$ and thus corresponds to a holomorphic function
on this disc. In many cases, this function can be holomorphically extended to a larger domain. Given $\rho,R\in\R$
with $0<\rho<R$ and $\theta\in(0,\pi/2)$,
\begin{equation*}
  \Delta(\rho,R,\theta) := \{ z\in\C \mid |z| < R \,\land\, |\arg(z-\rho)| > \theta \}
\end{equation*}
is called a \emph{$\Delta$-domain}. Here, $\arg(z)$ denotes the \emph{argument} of a complex number: $\arg(0) := 0$
and $\arg(re^{it}) := t$ for $r>0$ and $t\in(-\pi/2,\pi/2]$. We say that $F$ is \emph{$\Delta$-analytic}
if it is holomorphically extendable to some $\Delta$-domain $\Delta(\rho_F,R,\theta)$.

A function $F$ is \emph{subdominant} to a function $G$ if either $\rho_F>\rho_G$ or $\rho_F=\rho_G$ and 
$\lim_{z\to \rho_G}\frac{F(z)}{G(z)}=0$. In the latter case, if both $F$ and $G$ are $\Delta$-analytic,
then the above limit is required to be zero for $z$ from some fixed $\Delta$-domain to which both $F$ and
$G$ are holomorphically extendable. If $F$ is subdominant to $G$, we also write $F(z)=o(G(z))$. Analogously we write $F(z)=O(G(z))$
if either $\rho_F>\rho_G$ or $\rho_F=\rho_G$ and $\limsup_{z\to \rho_G}\frac{|F(z)|}{|G(z)|}<\infty$.

Given a function $F(z)$ with a dominant singularity $\rho_F$, we say that a function $G(z)=c\br{1-\rho_F^{-1}z}^{-\alpha}$ with
$\alpha\in\R\setminus\Z_{\le 0},c\in\R\setminus\{0\}$ or $G(z)=c\log\br{1-\rho_F^{-1}z}$ is the \emph{dominant term} of
$F$ if there is a decomposition
\begin{equation*}
  F(z) = P(z) + G(z) + o(G(z)),
\end{equation*}
where $P(z)$ is a polynomial. The dominant term, if it exists, is uniquely defined and $\Delta$-analytic. If 
$G(z)=c(1-\rho_F^{-1}z)^{-\alpha}$, the exponent
$-\alpha$ is called the \emph{dominant exponent} of $F$. If $G(z)=c\log\br{1-\rho_F^{-1}z}$,
we say that $F$ has dominant exponent $0$.

The number of edges in cubic multigraphs and triangulations is always a multiple of three. In terms of generating
functions, this is reflected by the existence of three different dominant singularities, all of which only differ
by a third root of unity. The corresponding dominant terms will also be the same up to a third root of unity.
Analogously, the number of vertices in cubic multigraphs is always even, resulting in two dominant singularities
$\rho_F$ and $-\rho_F$. Again, the dominant terms only differ by a factor of $-1$. In either case, the terms
for the coefficients coming from the different dominant singularities will also only differ by the corresponding root of unity.
Therefore, we will state our results only for the singularity $\rho_F$. With a slight abuse of notation, we will also refer
to $\rho_F$ as \emph{the} dominant singularity.

Singularity analysis allows us to derive an asymptotic expression for the coefficients of a generating function  $F(z)$  
with help of the dominant singularity and the dominant term of $F(z)$. We state the well-known `transfer theorem' for the
specific cases we will need.

\begin{thm}[\cite{flajolet1990-transfer}]\label{tt}
	Let $F(z)$ be a $\Delta$-analytic generating function.
  \begin{enumerate}
  \item\label{tt:polynomial}
    If
    \begin{equation*}
      F(z) = P(z) + c\br{1-\rho_F^{-1}z}^{-\alpha} + O\br{\br{1-\rho_F^{-1}z}^{1/4-\alpha}}
    \end{equation*}
    with a polynomial $P(z)$ and constants $c\not=0$, $\alpha\in\R\setminus\Z_{\le 0}$, then
    \[
      [z^n]F(z)=\br{1+O\br{n^{-1/4}}}\frac{c}{\Gamma(\alpha)} n^{\alpha-1}\rho_F^{-n}.
    \]
    Here, $\Gamma(\alpha):=\int_0^\infty z^{\alpha-1}e^{-z}\dd z$ is the gamma function.
  \item\label{tt:log}
    If
    \begin{equation*}
      F(z) = P(z) + c\cdot\log\br{1-\rho_F^{-1}z} + O\br{\br{1-\rho_F^{-1}z}^{1/4}},
    \end{equation*}
    then
    \[
      [z^n]F(z)=\br{1+O\br{n^{-1/4}}}(-c) n^{-1}\rho_F^{-n}.
    \]
  \end{enumerate}
\end{thm}

If we are counting rooted maps or multigraphs, the roots will be counted in the 
generating function unless stated otherwise. We will often mark vertices or edges of multigraphs or maps. 
This corresponds to applying the differential operator $z\frac{d}{dz}$ to the generating functions (with $z=x$ if vertices 
are marked and $z=y$ if edges are marked). To simplify notation 
we write $\opa{z}$ for $z\frac{d}{dz}$ and $\opa{z}^n$ for recursively applying $n$ times the operator $z\frac{d}{dz}$,
which corresponds to marking $n$ vertices or edges, \emph{allowing multiple marks on vertices or edges}. 
Thus we will have to compare dominant terms of functions before and after differentiation and (for a slightly different
reason) of functions before and after integration.

\begin{lem}\label{integrating}
	Let $F(z)$ be a $\Delta$-analytic generating function with the dominant term $F_d(z)$.
  Suppose there exists $\beta\in\R$ with
  \begin{equation*}
    F(z) = P(z) + F_d(z) +
      O\br{\br{1-\rho_F^{-1}z}^{-\beta}},
  \end{equation*}
  where $P(z)$ is a polynomial and $\br{1-\rho_F^{-1}z}^{-\beta} = o(F_d(z))$. Then
  \begin{enumerate}
  \item\label{integrating:diff}
    we have
    \begin{equation*}
      F'(z) = P'(z) + F'_d(z) +
      O\br{\br{1-\rho_F^{-1}z}^{-\beta-1}}.
    \end{equation*}
  \item\label{integrating:int}
    If in addition $F_d(z)=c\br{1-\rho_F^{-1}z}^{-\alpha}$, then for any primitive $G(z)$ of $F(z)$ there exists a
    primitive $Q(z)$ of $P(z)$ such that
    \begin{equation*}
      G(z) = Q(z) + G_d(z) + O(R(z))
    \end{equation*}
    with
    \begin{equation*}
      G_d(z) =
      \begin{cases}
        \frac{c\rho_F}{\alpha-1}\br{1-\rho_F^{-1}z}^{-\alpha+1} & \text{if }\alpha\not=1,\\
        -c\rho_F\log\br{1-\rho_F^{-1}z} & \text{if }\alpha=1
      \end{cases}
    \end{equation*}
    and
    \begin{equation*}
      R(z) =
      \begin{cases}
        \br{1-\rho_F^{-1}z}^{-\beta+1} & \text{if }\beta\not=1,\\
        \log\br{1-\rho_F^{-1}z} & \text{if }\beta=1.
      \end{cases}
    \end{equation*}
  \end{enumerate}
\end{lem}


\begin{proof}
  Let $F(z)=P(z)+F_d(z)+S(z)$, where $S(z)=O\br{\br{1-\rho_F^{-1}z}^{-\beta}}$. In order to
  prove~\ref{integrating:diff}, we need to show that $S'(z)=O\br{\br{1-\rho_F^{-1}z}^{-\beta-1}}$.
  Let $\Delta(\rho_F,R,\theta)$ be a $\Delta$-domain to which $F$ is holomorphically extendable and on
  which
  \begin{equation*}
    \limsup_{z\to\rho_F}\frac{|S(z)|}{\left\vert\br{1-\rho_F^{-1}z}^{-\beta}\right\vert} < \infty
  \end{equation*}
  holds. Fix a positive radius $r<R-\rho_F$ and denote the intersection of $\Delta(\rho_F,R,\theta)$
  with the disc of radius $r$ around $\rho_F$ by $U(r,\theta)$. Let $c_r$ be such that
  $|S(z)| \le c_r\left\vert\br{1-\rho_F^{-1}z}^{-\beta}\right\vert$ for
  every $z\in U(r,\theta)$. For some fixed $\theta' \in (\theta,\pi/2)$, consider
  $z\in U(r/2,\theta')$. There is a constant $a\in(0,1)$ such that for every such $z$, the
  closed disc of radius $a|\rho_F-z|$ around $z$ is contained in $U(r,\theta)$. If $M$ denotes
  the maximal value of $|S(z)|$ on the boundary of this disc, then Cauchy's integral formula yields
  \begin{equation*}
    |S'(z)| \le \frac{M}{a|\rho_F-z|} \le \frac{c_r\rho_F^{-1}}{a\min((1-a)^{\beta},(1+a)^{\beta})}\left\vert\br{1-\rho_F^{-1}z}^{-\beta-1}\right\vert
  \end{equation*}
  and thus $S'(z)=O\br{\br{1-\rho_F^{-1}z}^{-\beta-1}}$ on the domain $\Delta(\rho_F,R,\theta')$.

  In order to prove~\ref{integrating:int}, we need to show that $S(z)$ has a primitive $T(z)$
  with $T(z) = O(R(z))$. Let $r$, $c_r$, and $U(r,\theta)$ be as before. For $\varepsilon\in(0,r)$, let $T_{\varepsilon}(z)$
  be the primitive of $S(z)$ with $T_{\varepsilon}(\rho_F-\varepsilon)=0$.
	Then for every $z\in U(r,\theta)$ we have
  \begin{equation*}
    T_{\varepsilon}(z)=\int_{\gamma}S(z)\dd z
  \end{equation*}
  for any curve $\gamma$ from $\rho_F-\varepsilon$ to $z$ in $U(r,\theta)$.
  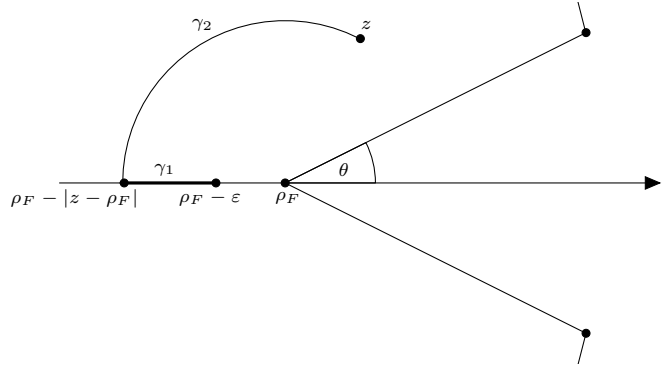
\begin{figure}
\begin{tikzpicture}[line cap=round,line join=round,>=triangle 45,x=2.0cm,y=2.0cm]
\begin{scriptsize}
\draw[->,color=black] (4.5,0.) -- (8.5,0.);
\draw[color=black] (4.6,-0.1) node {$\rho_F-|z-\rho_F|$};
\clip(4.5,-1.2) rectangle (8.5,1.2);
\draw [shift={(6.,0.)},color=black] (0,0) -- (0:0.6) arc (0:26.56:0.6) -- cycle; 
\draw [shift={(6.,0.)}] plot[domain=1.1:3.141592,variable=\t]({1.*1.08*cos(\t r)+0.*1.08*sin(\t r)},{0.*1.08*cos(\t r)+1.*1.08*sin(\t r)});
\draw [shift={(0.,0.)}] plot[domain=0.123:6.1588,variable=\t]({1.*8.062*cos(\t r)},{1.*8.062*sin(\t r)});
\draw (6.,0.)-- (8.,1.);
\draw (6.,0.)-- (8.,-1.);
\draw [line width=1.2pt] (5.54,0.)-- (4.91,0.);

\draw [fill=black] (6.,0.) circle (1.5pt);
\draw[color=black] (6.02,-0.1) node {$\rho_F$};
\draw [fill=black] (5.54,0.) circle (1.5pt);
\draw[color=black] (5.5,-0.1) node {$\rho_F-\varepsilon$};
\draw [fill=black] (6.5,0.96) circle (1.5pt);
\draw[color=black] (6.54,1.05) node {$z$};
\draw[color=black] (5.45,1.05) node {$\gamma_2$};
\draw [fill=black] (4.93,0.) circle (1.5pt);
\draw [fill=black] (0.,0.) circle (1.5pt);
\draw [fill=black] (8.,1.) circle (1.5pt);
\draw [fill=black] (8.,-1.) circle (1.5pt);
\draw[color=black] (5.2,0) node[anchor=south] {$\gamma_1$};
\draw[color=black] (6,0) + (13.28:0.4) node {$\theta$};
\end{scriptsize}   
\end{tikzpicture}
   \caption{The integration curve $\gamma_1+\gamma_2$. The direction of $\gamma_1$ depends on whether $|z-\rho_F|$ is larger or smaller than $\varepsilon$.}\label{pic:intpath}
  \end{figure}

  We choose $\gamma = \gamma_1 + \gamma_2$, where $\gamma_1$ is a line segment
  from $\rho_F-\varepsilon$ to $\rho_F-|\rho_F-z|$ and $\gamma_2$ is a circular arc
  in $U(r,\theta)$ around $\rho_F$ from $\rho_F-|\rho_F-z|$ to $z$ (see \Cref{pic:intpath}). If $M_2$ denotes
  the maximum of $|S(z)|$ on the trace of $\gamma_2$, then
 	\begin{equation*}
		\left|\int_{\gamma_2}S(z)\dd z\right| \le \pi |\rho_F-z| M_2
    \le \pi c_r\rho_F\left|1-\rho_F^{-1}z\right|^{-\beta+1}.
 	\end{equation*}
 	For the integral along $\gamma_1$, we have to distinguish between $\beta\not=1$
  and $\beta = 1$. For $\beta\not=1$, we have
 	\begin{align*}
		\left|\int_{\gamma_1}S(z)\dd z\right|
    &= \left|\int_0^1(\varepsilon-|\rho_F-z|)S\big(\rho_F-\varepsilon+t(\varepsilon-|\rho_F-z|)\big)\dd t\right|\\
    &\le \left|\big|\varepsilon-|\rho_F-z|\big|\int_0^1\left|S\big(\rho_F-\varepsilon+t(\varepsilon-|\rho_F-z|)\big)\right|\dd t\right|\\
		&\le \left|c_r\big|\varepsilon-|\rho_F-z|\big|\int_0^1\br{1-\frac{\rho_F-\varepsilon+t(\varepsilon-|\rho_F-z|)}{\rho_F}}^{-\beta}\dd t\right|\\
    &= \left|c_r\big|\varepsilon-|\rho_F-z|\big|\int_0^1\br{\frac{\varepsilon}{\rho_F}+t\frac{|\rho_F-z|-\varepsilon}{\rho_F}}^{-\beta}\dd t\right|\\
		&= \frac{c_r\rho_F}{|\beta-1|}\left|\left|1-\rho_F^{-1}z\right|^{-\beta+1}-\left(\frac{\varepsilon}{\rho_F}\right)^{-\beta+1}\right|\\
    &\le \frac{c_r\rho_F}{|\beta-1|}\left|1-\rho_F^{-1}z\right|^{-\beta+1} + \frac{c_r\rho_F^{\beta}}{|\beta-1|}\varepsilon^{-\beta+1}.
 	\end{align*}
  For $\beta=-1$ we analogously obtain
  \begin{equation*}
    \left|\int_{\gamma_1}S(z)\dd z\right| \le c_r\rho_F\left|\log\br{1-\rho_F^{-1}z}\right| + c_r\rho_F\left|\log\br{\rho_F^{-1}\varepsilon}\right|.
  \end{equation*}
  By combing the two integrals, we obtain for $\beta\not=1$ and any fixed $\varepsilon$
  \begin{equation}\label{eq:integral}
    |T_{\varepsilon}(z)| \le \br{\pi c_r\rho_F+\frac{c_r\rho_F}{|\beta-1|}}|R(z)| + \frac{c_r\rho_F^{\beta}}{|\beta-1|}\varepsilon^{-\beta+1}.
  \end{equation}
  For $\beta>1$, this implies that $T_{\varepsilon}(z) = O(R(z))$, since $\frac{c_r\rho_F^{\beta}}{|\beta-1|}\varepsilon^{-\beta+1}$ is a
  constant, but $|R(z)|$ tends to infinity as $z$ tends to $\rho_F$. For $\beta=1$, the integral along $\gamma_2$ is
  bounded by a constant and we again obtain $T_{\varepsilon}(z) = O(R(z))$.

  Finally, suppose that $\beta<1$. Then~\eqref{eq:integral} implies that, given $\delta>0$,
  we can choose $\varepsilon_0>0$ and $z\in U(r,\theta)$ close enough to $\rho_F$ so that
  \begin{equation*}
    |T_{\varepsilon}(z)| < \frac{\delta}{2}
  \end{equation*}
  for every $\varepsilon$ with $0<\varepsilon<\varepsilon_0$ and thus
  \begin{equation}\label{eq:convergent}
    |T_{\varepsilon}(z)-T_{\varepsilon'}(z)| < \delta
  \end{equation}
  for every $\varepsilon,\varepsilon'\in(0,\varepsilon_0)$. As both $T_{\varepsilon}$ and $T_{\varepsilon'}$
  are primitives of $S$, they only differ by a constant and thus~\eqref{eq:convergent} holds for every
  $z\in U(r,\theta)$. Therefore the sequence $(T_{1/n})_{n\in\N}$ converges uniformly to a function $T_0$
  which is also a primitive of $S$. Letting $\varepsilon$ tend to zero in~\eqref{eq:integral}, we see
  that
  \begin{equation*}
    |T_0(z)| \le \br{\pi c_r\rho_F+\frac{c_r\rho_F}{|\beta-1|}}|R(z)|
  \end{equation*}
  for every $z\in U(r,\theta)$ and thus $T_0(z) = O(R(z))$.

  We have thus found a primitive $T_0$ of $S$ that satisfies $T_0(z) = O(R(z))$. We conclude the proof
  of~\ref{integrating:int} by choosing $Q$ as the unique primitive of $P$ with $Q(0) = G(0) - G_d(0) - T_0(0)$.
\end{proof}

\Cref{tt} and \Cref{integrating} are very helpful when the generating functions in question are
$\Delta$-analytic. However, for many of our generating functions we will not be able to guarantee
$\Delta$-analyticity. In order to utilise the results of this section also for generating functions that
are not necessarily $\Delta$-analytic, we introduce the following concept and notation. Given a generating
function $F(z)$ and $\Delta$-analytic functions $G(z),H(z)$, we say that $F(z)$ is \emph{congruent} to
$G(z)+O(H(z))$ and write
\begin{equation*}
  F(z) \cong G(z) + O(H(z))
\end{equation*}
if there exist $\Delta$-analytic functions $F^+(z),F^-(z)$ and polynomials $P^+(z),P^-(z)$ such that
\begin{itemize}
\item $F^- \preceq F \preceq F^+$;
\item $F^+(z) = P^+(z) + G(z) + O(H(z))$; and
\item $F^-(z) = P^-(z) + G(z) + O(H(z))$.
\end{itemize}
With this, we are able to apply the transfer theorem even if $F$ itself is not $\Delta$-analytic.
The following lemma is an immediate consequence of \Cref{tt} and the fact that $F^- \preceq F \preceq F^+$.

\begin{lem}\label{congruent:tt}
  If $F(z) \cong G(z) + O(H(z))$ and $G(z) + O(H(z))$ is of one of the types described
  in~\Cref{tt}\ref{tt:polynomial} or~\ref{tt:log}, then
  \begin{equation*}
    [x^n]F(z) = \br{1+O\br{n^{-1/4}}}[z^n]F^+(z) = \br{1+O\br{n^{-1/4}}}[z^n]F^-(z).
  \end{equation*}
\end{lem}

In this paper we will often encounter sums, products, differentials and integrals of generating
functions. The following lemma states that these operations are compatible with the notion of
congruence. We will frequently use this lemma without explicitly mentioning it.

\begin{lem}\label{congruent:meta}
  Let $G_1,G_2,H_1,H_2$ be $\Delta$-analytic functions with only finitely many negative
  coefficients. Let $F_1,F_2$ be generating functions such that
  \begin{align*}
    &F_1(z) \cong G_1(z) + O(H_1(z))&
    &\text{and}&
    &F_2(z) \cong G_2(z) + O(H_2(z))&
  \end{align*}
  and let $F(z)$ be a generating function with the dominant term $F_d(z)$ such that
  \begin{equation*}
    F(z) \cong F_d(z) +
      O\br{\br{1-\rho_F^{-1}z}^{-\beta}},
  \end{equation*}
  where $\br{1-\rho_F^{-1}z}^{-\beta}= o(F_d(z))$. Then the following holds.
  \begin{align*}
    F_1(z) \pm F_2(z) &\cong G_1(z) \pm G_2(z) + O(H_1(z)) + O(H_2(z)),\\
    F_1(z) F_2(z) &\cong G_1(z)G_2(z) + O(G_1(z)H_2(z) + H_1(z)G_2(z) + H_1(z)H_2(z)),\\
    F'(z) &\cong F'_d(z) + O\br{\br{1-\rho_F^{-1}z}^{-\beta-1}}.
  \end{align*}
  Furthermore, if $F_d(z)=c\br{1-\rho_F^{-1}z}^{-\alpha}$, then for any primitive
  $\tilde F(z)$ of $F(z)$ we have
  \begin{equation*}
    \tilde F(z) \cong G_d(z) + O\br{R(z)},
  \end{equation*}
  where $G_d$ and $R$ are as in \Cref{integrating}\ref{integrating:int}.
\end{lem}

\begin{proof}
  The first congruence follows immediately from
  \begin{align*}
    &F_1^- + F_2^- \preceq F_1 + F_2 \preceq F_1^+ + F_2^+&
    &\text{and}&
    &F_1^- - F_2^+ \preceq F_1 - F_2 \preceq F_1^+ - F_2^-.&
  \end{align*}

  For the product $F_1(z) \cdot F_2(z)$, we may assume that $F_1^-$ and $F_2^-$ 
  have nonnegative coefficients, since $G_1$ and $G_2$ have only finitely many negative coefficients.
  Hence
  \begin{equation*}
    F_1^- \cdot F_2^- \preceq F_1 \cdot F_2 \preceq F_1^+ \cdot F_2^+
  \end{equation*}
  and the second congruence follows.

  The last two congruences follow from \Cref{integrating} and the fact that
  \begin{align*}
    &(F^-)' \preceq F' \preceq (F^+)'&
    &\text{and}&
    &\tilde F^- \preceq \tilde F \preceq \tilde F^+,&
  \end{align*}
  where $\tilde F^-,\tilde F^+$ are primitives of $F^-,F^+$, respectively, with
  $\tilde F^-(0) \le \tilde F(0) \le \tilde F^+(0)$.
\end{proof}

\subsection{Maps with large facewidth}

An \emph{essential circle} on $\Sg$ is a circle that is not contractible to a point on $\Sg$. Let $M$ be an embedding 
of a multigraph on $\Sg$. An \emph{essential cycle} of $M$ is a cycle of $M$ that is an essential circle on the surface.
The \emph{facewidth} $\fw(M)$ of $M$ is the minimal number of intersections of $M$ with an essential circle on $\Sg$. 
The \emph{edgewidth} $\ew(M)$ of $M$ is defined as the minimal number of edges of an essential cycle of $M$.
If $g=0$, there are neither essential circles nor essential cycles and we use the convention $\fw(M)=\ew(M)=\infty$.
\Cref{strongemb} implies that if $M$ is connected 
and \emph{not} a $2$-cell embedding, then $\fw(M)=0$. The facewidth $\fw_g(G)$ of a multigraph $G$ that is embeddable on $\Sg$ is defined as the maximal 
facewidth of all its embeddings on $\Sg$. If we count multigraphs with restrictions to their facewidth, we will indicate the restriction 
by a superscript to the corresponding generating function, e.g.\ $G^{\fw_g\ge2}(x)$ for the generating function of all 
multigraphs with facewidth at least two.  If the genus is clear from the context we will omit it and write $\fw(G)$.

Having large facewidth proves to be a very helpful property. It will allow us to derive a constructive decomposition 
along connectivity as well as the existence of a unique embedding for $3$-connected multigraphs. The following lemma was  applied
in a similar way in \cite{Chapuy2011-enumeration-graphs-on-surfaces}.

\begin{lem}\label{robertsonvitray}\cite{robertson1990-representativity}
Let $g>0$ and let $M$ be an embedding of a connected multigraph $G$ on $\Sg$.
\begin{enumerate}
\item\label{rv1} $M$ has facewidth $\fw(M)=k\geq2$ if and only if $M$ 
	has a unique $2$-connected component embedded 
	on $\Sg$ with facewidth $k$ and all other $2$-connected components of $M$ are planar.
\item\label{rv2} If $G$ is $2$-connected, $M$ has facewidth $\fw(M)=k\geq3$ if 
	and only if $M$ has a unique $3$-connected component embedded 
	on $\Sg$ with facewidth $k$ and all other $3$-connected components of $M$ are planar.
\item\label{rv3} Let $M_1$, $M_2$ be embeddings of a $3$-connected multigraph on $\Sg$ and suppose that $\fw(M_1)\ge2g+3$. 
	Then there is a homeomorphism of $\Sg$ that maps $M_1$ to $M_2$.
\end{enumerate}
\end{lem}

\Cref{robertsonvitray}(iii) is a generalisation of Whitney's theorem that all $3$-connected planar multigraphs have 
a unique embedding up to orientation on the sphere. Because we will need \Cref{robertsonvitray} for multigraphs rather than 
for embeddings, we will use the following easy corollary.

\begin{coro}\label{coro:rv}
Let $g>0$ and let $G$ be a non-planar connected multigraph strongly embeddable on $\Sg$.
\begin{enumerate}
\item\label{rvc1} If $\fw_g(G)\ge 2$, 
	then $G$ has a unique $2$-connected non-planar component strongly embeddable on $\Sg$ with facewidth $\fw_g(G)$ 
	and all other $2$-connected components are planar.
\item\label{rvc2} If $G$ is $2$-connected and 
	$\fw_g(G)\ge 3$, then $G$ has a unique $3$-connected non-planar component strongly embeddable on 
	$\Sg$ with facewidth $\fw_g(G)$ and all other $3$-connected components are planar.
\item\label{rvc3} If $G$ is $3$-connected and $\fw_g(G)\ge 2g+3$, then $G$ has a unique $2$-cell embedding on $\Sg$ up to orientation.
\end{enumerate}
\end{coro}

\begin{proof} By \Cref{robertsonvitray}\ref{rv1}, for any fixed embedding of $G$ all but one components are planar. 
As $G$ itself is not planar, that exceptional component has to be non-planar. As the component structure is the same for 
\emph{all} embeddings, the non-planar component $I$ is independent of the embedding. Therefore, $I$ is the unique 
component described in \ref{rvc1}. 
Part \ref{rvc2} is proved analogously and \ref{rvc3} is a direct consequence of \Cref{robertsonvitray}\ref{rv3}.
\end{proof}


	\section{Maps and triangulations}\label{maps}

The goal of this section is to enumerate cubic $3$-connected \emph{maps} on $\Sg$. The dual 
of these maps are triangulations, which are characterised in the following proposition.

\begin{prop} \label{prop:3-conn-dual}
Let $M$ be a $2$-cell embedding of a cubic multigraph on $\Sg$ and let $M^*$ be its dual map. 
Then $M$ is 3-connected if and only if $M^*$ is a triangulation with at least $6$ edges and without 
separating loops, separating double edges, and separating pairs of loops.
\end{prop}
\begin{proof}
For cubic maps with at least four vertices (and thus at least six edges), 3-connectivity and 3-edge-connectivity coincide. This can be seen by a simple case analysis. We thus 
use 3-edge-connectivity hereafter. Since a vertex in $M$ corresponds to a face in $M^*$, deleting edges in the 
primal $M$ has the same effect as cutting the surface along the dual edges of $M^*$, and a set of edges is a separator in 
$M$ if and only if cutting along the dual edges in $M^*$ separates the surface. Thus a bridge in $M$ corresponds 
to a separating loop in $M^*$. A $2$-edge-separator in $M$ corresponds either to a separating double edge or a pair 
of loops in $M^*$ which together separate the surface.
\end{proof}

\subsection{Relations between triangulations}

In order to enumerate the triangulations described in \Cref{prop:3-conn-dual}, we will relate them to simple 
triangulations that have been studied by Gao \cite{Gao1992-2-connected-surface}. To this end we will use 
topological surgeries that eliminate loops and double edges. The most frequently used surgery is the following.
\begin{definition}\label{def:zipping}
	Let $\{e_1,e_2\}$ be a double edge in a triangulation $T$ on $\Sg$. By \emph{zipping $\{e_1,e_2\}$} we mean the
	operation of cutting the surface along $\{e_1,e_2\}$ and eliminating the holes by identifying the two edges on
	the boundaries of the holes. The two edges that result from this identification are called \emph{zipped edges}. 
	The resulting triangulation $\tilde{T}$ is rooted at the corner $\tilde{c}$ corresponding to the root corner $c$ of 
	$T$. If $\{e_1,e_2\}$ separates $\Sg$, one of the two components of $\tilde{T}$ does not 
	contain $\tilde{c}$; we root this component $C$ such that the zipped edge in $C$ is the root edge. Note that there
	are two choices for the root of $C$.
\end{definition}
Note that zipping a double edge leaves the number of edges unchanged.

We will distinguish three kinds of double edges. 
We call a double edge \emph{planar} if zipping it results in two maps, one of which is planar and has no holes. Note that an alternative
way of defining planar double edges would be to say that by cutting along it we obtain two surfaces, one of which is a disc.
We call a double edge \emph{non-planar separating} if zipping it results in two maps none of which is planar
and we call it \emph{non-separating} if zipping it results in only one map. 

Let $\class{M}_g$ be the class of triangulations on $\Sg$ without separating loops, separating double edges, and 
separating pairs of loops and let $M_g(y)$ be its ordinary generating function. Note that these 
triangulations are either the duals of $3$-connected cubic maps on $\Sg$ by \Cref{prop:3-conn-dual} or a triangulation 
with exactly three edges. To obtain the asymptotic number of triangulations in 
$\class{M}_g$, we need some other types of triangulations. Let $\class{N}_g$ be the class of
triangulations on $\Sg$ without loops and separating double edges. Let $\class{R}_g$ be the class of 
triangulations on $\Sg$ without loops and planar double edges. Let $\class{\triangt}_g$ be the class of triangulations on 
$\Sg$ without loops. Let $\class{S}_g$ be class of simple triangulations on $\Sg$ 
(i.e.\ without loops and double edges). Let $\triangn_g(y)$, $\triangr_g(y)$, $\triangt_g(y)$, and $\triangs_g(y)$ be their generating functions, 
respectively. \Cref{table:multi} summarises what kinds of edges are allowed for these classes of triangulations.

\begin{table}[htbp]
\begin{tabular}{c|c|c|c|c}
\hline
\quad & loops & \multicolumn{3}{c}{double edges} \\
\hline
\quad &\quad  & non-separating & planar & non-planar separating \\
\hline
$\class{\triangm}_g$ & Some cases & Yes & No & No \\
\hline
$\class{\triangn}_g$ & No & Yes & No & No \\
\hline
$\class{\triangr}_g$ & No & Yes & No & Yes \\
\hline
$\class{\triangt}_g$ & No & Yes & Yes & Yes \\
\hline
$\class{\triangs}_g$ & No & No & No & No \\
\hline
\end{tabular}
\vspace{1ex}
\caption{Loops and double edges in different classes of triangulations}\label{table:multi}
\label{table:triang}
\end{table}

Note that $\class\triangs_g\subseteq\class\triangn_g\subseteq\class\triangr_g\subseteq\class\triangt_g$ and $\class\triangn_g\subseteq\class\triangm_g$.

We will relate the triangulations in $\class{\triangm}_g$, $\class{\triangn}_g$, $\class{\triangr}_g$, $\class{\triangt}_g$, and $\class{\triangs}_g$
by performing surgeries. The most frequently used surgery will be to zip a separating double edge. In order to relate the maps before and after
the surgery, we need the following lemma.

\begin{lem} \label{lem:surgery-2}
Let $\class{X}_g(m)$ be the class of all rooted triangulations on $\Sg$ with $m$ edges and a marked 
double edge $\{ e_1, e_2 \}$ and let $\class{Y}_g^{(i)}(m)$ be the class of all rooted triangulations on 
$\Sg$ with $m$ edges, $i$ of which are marked.
\begin{enumerate}

\item\label{sep} Let $\class{X}_g^{s}(m)$ be the subclass of $\class{X}_g(m)$ where $\{e_1,e_2\}$ is a 
	non-planar separating 
	double edge. Then zipping $\{e_1,e_2\}$ yields a $2$-to-$2$ 
	correspondence from $\class{X}_g^s(m)$ to 
	\[\bigcup \br{\class{Y}_{g_1}^{(1)}(m_1)\times \class{Y}_{g_2}^{(0)}(m_2)},\]
	where the union is taken over all $m_1,m_2,g_1,g_2\in\N$ such that $m_1+m_2=m$ and $g_1+g_2=g$.
\item\label{planar} Let $\class{X}_g^{pl}(m)$ be the subclass of $\class{X}_g(m)$ where $\{e_1,e_2\}$ is a 
	planar double edge. 
	Then zipping $\{e_1,e_2\}$ yields a $2$-to-$2$ correspondence from $\class{X}_g^{pl}(m)$ to 
	\[\bigcup \br{\class{Y}_{g}^{(1)}(m_1)\times \class{Y}_{0}^{(0)}(m_2)\cup\class{Y}_{0}^{(1)}(m_1)\times \class{Y}_{g}^{(0)}(m_2)},\]
	where the union is taken over all $m_1, m_2\in\N$ such that $m_1+m_2=m$.
\end{enumerate}
\end{lem}

\begin{proof}
%
In Case \ref{sep} the surface is disconnected into $\torus{g_1}$ and $\torus{g_2}$ 
with $g_1+g_2=g$. Assume without 
loss of generality that the root face of the original triangulation is on $\torus{g_1}$. Then we mark the identified 
edge on $\torus{g_1}$. For the root on $\torus{g_2}$ we have two choices.
Going back we have two possibilities of glueing the surface back together, hence obtaining the $2$-to-$2$ 
correspondence.

In Case \ref{planar}, the only difference to Case \ref{sep} is that $g_1=0$ and $g_2=g$ and that the root face might 
be in either of the two maps.
%
\end{proof}

We finish this section by a last topological statement about planar double edges.

\begin{definition}\label{def:relation}
	Let $T$ be a triangulation on $\Sg$ (possibly with holes), $g>0$, and let $d=\{e_1,e_2\}$ be a \emph{planar} double edge.
	If $c\neq d$ is a face, edge, or double edge, we write $c<d$ if the triangulation of the disc
	obtained by cutting along $d$ contains $c$.

\end{definition}

\begin{lem} \label{lem:planar-dbl-inclusion}
Let $T$ be a triangulation on $\Sg$, $g>0$, and let $D$ be the set of \emph{planar} double edges in $T$. 
Then the relation $<$ defined in \Cref{def:relation} is a partial order on $D$. 
\end{lem}
\begin{proof}
Let $d_1,d_2\in D$. Let $P_1$ and $P_2$ be the triangulations of the discs obtained by cutting along $d_1$ and $d_2$, 
respectively. 
Then $d_1<d_2$ if and only if $d_1\subseteq P_2$ by definition and this holds if and only if $P_1\subseteq P_2$, since 
$T$ is a triangulation of $\Sg$ and $g > 0$. As the subset relation on triangulations is a partial ordering, 
so is the relation defined in \Cref{def:relation}.
\end{proof}

\subsection{Counting triangulations}

In this section, we determine the dominant terms of the generating functions of various types of triangulations 
using singularity analysis and topological surgeries, e.g.\ zipping double edges (\Cref{def:zipping}).
The starting point for this will be $\triangt_g(y)$ and $\triangs_g(y)$, which were examined by Gao 
\cite{Gao1992-2-connected-surface,Gao1993-pattern} and (in the planar case) by Tutte \cite{tutte1962-planar-triangulations}. 
However, the results obtained by Gao are not strong enough in order to apply the theory of singularity analysis (developed 
in \Cref{section:singularity}). We obtain more refined versions of their results.

\begin{prop}\label{simplerefined}\label{prop:planar-S}
  The dominant singularity of $\triangs_g(y)$ is given by $\rho_\triangs=\frac{3}{2^{8/3}}$.
  The generating function $\triangs_0(y)$ is $\Delta$-analytic and satisfies
  \begin{align}
  \triangs_0(y)&=\frac{1}{8}-\frac{9}{16}\br{1-\rho_\triangs^{-1}y}+\frac{3}{2^{5/2}}\br{1-\rho_\triangs^{-1}y}^{3/2}
	+O\br{\br{1-\rho_\triangs^{-1}y}^{2}}.
  \end{align}
  For $g\ge1$ we have
  \begin{align}
    \triangs_g(y)&\cong c_g \br{1-\rho_\triangs^{-1} y}^{-5g/2+3/2}\br{1+O\br{\br{1-\rho_\triangs^{-1} y}^{1/4}}},
  \end{align}
  where $c_g$ is a constant depending only on $g$.
  Furthermore, the asymptotic number of simple triangulations on $\Sg$ with $m$ edges is given by
  \begin{align*}
    s_g(m)&=\br{1+O\br{m^{-1/4}}}\frac{c_g}{\Gamma(5(g-1)/2)}m^{5g/2-5/2}\rho_\triangs^{-m},
  \end{align*}
  where $c_0 = \frac{3}{2^{5/2}}$.
\end{prop}

\begin{prop}\label{looplessrefined}\label{prop:planar-T}
  The dominant singularity of $\triangt_g(y)$ is given by $\rho_{\triangt}=\frac{2^{1/3}}{3}$.
  The generating function $\triangt_0(y)$ is $\Delta$-analytic and satisfies
  \begin{align}
    \triangt_0(y)=\frac{1}{8}-\frac{9}{8}\br{1-\rho_{\triangt}^{-1} y}+3\br{1-\rho_{\triangt}^{-1} y}^{3/2}+O\br{\br{1-\rho_{\triangt}^{-1} y}^2}.
  \end{align}
  For $g\ge1$ we have
  \begin{align}
    \triangt_g(y)&\cong \hat{c}_g \br{1-\rho_{\triangt}^{-1} y}^{-5g/2+3/2}\br{1+O\br{\br{1-\rho_{\triangt}^{-1} y}^{1/4}}},
  \end{align}
  where $\hat{c}_g$ is a constant depending only on $g$.
  Furthermore, the asymptotic number of loop-less triangulations on $\Sg$ with $m$ edges is given by
  \begin{align*}
    \hat{s}_g(m)&=\br{1+O\br{m^{-1/4}}}\frac{\hat{c}_g}{\Gamma(5g/2-5/2))}m^{5g/2-5/2}\rho_{\triangt}^{-m},
  \end{align*}
  where $\hat{c}_0 = 3$.
\end{prop}
For the proofs of \Cref{simplerefined,looplessrefined} see \Cref{proof:gao}. The exact values of $c_g$ and $\hat{c}_g$ are 
not necessary for our purpose; the interested reader can find them in \cite{Gao1993-pattern}. The only fact about the 
coefficients we will use in this paper is the relation
\begin{equation}\label{eq:relation}
 \hat{c}_g=2^{5(1-g)/2}c_g.
\end{equation}

Our first step towards obtaining an asymptotic formula for $M_g(y)$ is to relate $\triangr_g(y)$ and 
$\triangt_g(y)$ by recursively zipping planar double edges. The resulting 
relation will show that the dominant singularities of $\triangr_g(y)$ and $\triangs_g(y)$ are the same. 

\begin{prop} \label{prop:R-asymptotic}
The generating functions $\triangr_g(y)$ and $\triangs_g(y)$ have the same dominant singularity $\rho_S$, 
$\triangr_0(y)=\triangs_0(y)$, and 
\[\triangr_g(y)-\triangs_g(y) \cong O\br{\br{1-\rho_S^{-1}y}^{-5g/2+7/4}},\quad g\ge1.\]
\end{prop}
\begin{proof}
For $g=0$, all double edges are planar and hence $\class{R}_0=\class{S}_0$. We may thus assume that $g>0$.
For a triangulation $T\in\class{\triangt}_g$, let $D_{\max}$ be the set of \emph{maximal} planar double edges 
in the partial order defined in \Cref{def:relation}. For any two double edges $d_1,d_2\in D_{\max}$, their planar sides meet
at most at the end vertices of $d_1,d_2$, since otherwise either $d_1>d_2$, $d_2>d_1$, or there would be a double
edge $d$ consisting of one edge from $d_1$ and one edge from $d_2$ such that $d>d_1$ and $d>d_2$.

We can thus zip all double edges in $D_{\max}$ in an arbitrary order. 
We obtain a triangulation $R$ on $\Sg$ without planar double edges (thus in $\class{\triangr}_g$) and $|D_{\max}|$ planar 
triangulations without loops but possibly with double edges (thus in $\class{\triangt}_0$).
To obtain an equation between $\class{\triangt}_g$ and the triangulations resulting from the surgery,
we have to distinguish two cases: the original root face is in one of the planar triangulations or not. 

By the arguments above there can be at most one element in $D_{\max}$ that contains the root face in its planar part. 
By \Cref{def:zipping}, for this double edge $d_0$, the corresponding planar triangulation is rooted at the corner corresponding to the root of $T$ 
and the zipped double edge is marked. All other planar triangulations are rooted at their zipped edge 
and have no additional marked edges. $R$ is rooted at the zipped edge $d_0$ and all other zipped edges are marked.
On the other hand, if none of the double edges in $D_{\max}$ contains the root face in its planar part, $R$ is rooted at the root 
of $T$ and all zipped edges are marked.

In either case we end up with planar triangulations and a triangulation $R\in\class{\triangr}_g$ where any 
subset of the edges is marked. By \Cref{lem:surgery-2}\ref{sep} this construction yields a bijection (or more precisely 
a $k$-to-$k$ correspondence, with $k$ depending on the number of zipped edges) between these collections of triangulations
and $\class{\triangt}_g$.

As any number of edges in $R$ can be marked, we can directly replace any set of edges by planar 
triangulations (more precisely, cut $R$ along an edge and the planar triangulation along its root 
edge and glue them along their boundaries).
The root can either be replaced in the same way as a non-root edge (if the root remains in $R$), 
or by a planar triangulation with one marked edge, counted by $\opa{y}\triangt_0(y)$. Thus we obtain
\[\triangt_g(y)=\left(1 + \triangt_0(y) + \opa{y}\triangt_0(y) \right)\frac{\triangr_g(y(1 + \triangt_0(y)))}
	{1+\triangt_0(y)}\,.\]
Here, $\triangr_g(y(1 + \triangt_0(y)))$ corresponds to the substitution of edges, the denominator $1+\triangt_0(y)$ compensates for 
not allowing the root to be substituted in the same way, and the factor 
$\left(1 + \triangt_0(y) + \opa{y}\triangt_0(y) \right)$ corresponds to the substitution of the root.
This formula can be rearranged as
\begin{equation}\label{eq:rtot}\triangr_g(y(1 + \triangt_0(y)))=\triangt_g(y)\frac{1+\triangt_0(y)}
	{1 + \triangt_0(y) + \opa{y}\triangt_0(y)}\,. 
\end{equation}

By \Cref{looplessrefined}, the dominant singularity of $\triangt_g(y)$ occurs at $\rho_{\triangt}$ for all $g$. 
Again by \Cref{prop:planar-T}, $\triangt_0(y)$ and $\opa{y}\triangt_0(y)$ are finite and nonnegative at 
the singularity. In particular, the fraction in~\eqref{eq:rtot} is $\Delta$-analytic and of the form
\[\frac{1+\triangt_0(y)}{1 + \triangt_0(y) + \opa{y}\triangt_0(y)} = \frac12 + O\br{\br{1-\rho_{\triangt}^{-1}y}^{1/2}}.\]
Summing up we obtain
\begin{equation}\label{eq:rasy}
  \triangr_g(y(1 + \triangt_0(y))) \cong \frac12\hat{c}_g \br{1-\rho_{\triangt}^{-1} y}^{-5g/2+3/2}\br{1+O\br{\br{1-\rho_{\triangt}^{-1} y}^{1/4}}}.
\end{equation}
Therefore, the dominant singularity of $\triangr_g(y)$ satisfies 
\[\rho_{\triangr}\geq \rho_{\triangt} (1 + \triangt_0(\rho_{\triangt})) = (9/8) \rho_{\triangt}=\rho_{\triangs}.\]
On the other hand, $\class{\triangs}_g\subseteq\class{\triangr}_g$, therefore $\rho_\triangr\leq \rho_{\triangs}$
and thus $\rho_\triangr= \rho_{\triangs}$.

It remains to show that
\[\triangr_g(y)\cong c_g\br{1-\rho_S^{-1}y}^{-5g/2+3/2}\br{1+O\br{\br{1-\rho_S^{-1}y}^{1/4}}}.\]
To this end we shall replace factors $\br{1-\rho_{\triangt}^{-1}y}$ in~\eqref{eq:rasy} by $\br{1-\rho_{\triangs}^{-1}y(1+\triangt_0(y))}$
and then replace each occurrence of $y(1+\triangt_0(y))$ by $\hat y$. We have
\begin{align*}
  1-\rho_{\triangs}^{-1}y(1+\triangt_0(y)) &= 1-\rho_{\triangs}^{-1}y\br{\frac98-\frac98\br{1-\rho_{\triangt}^{-1}y} + O\br{\br{1-\rho_{\triangt}^{-1}y}^{3/2}}}\\
  &= 1 - \rho_{\triangt}^{-1}y + \rho_{\triangt}^{-1}y\br{1-\rho_{\triangt}^{-1}y}\br{1+O\br{\br{1-\rho_{\triangt}^{-1}y}^{1/2}}}\\
  &= 2\br{1-\rho_{\triangt}^{-1}y}\br{1+O\br{\br{1-\rho_{\triangt}^{-1}y}^{1/2}}},
\end{align*}
as $\rho_{\triangt}^{-1}y\br{1-\rho_{\triangt}^{-1}y} = \br{1 - \br{1-\rho_{\triangt}^{-1}y}}\br{1-\rho_{\triangt}^{-1}y}$.
Since $a = b\br{1+O\br{b^{1/2}}}$ implies $b = a\br{1+O\br{a^{1/2}}}$ when $a,b\to 0$, we have thus shown
\begin{equation*}
  1-\rho_{\triangt}^{-1}y = \frac12\br{1-\rho_{\triangs}^{-1}y(1+\triangt_0(y))}\br{1+O\br{\br{1-\rho_{\triangs}^{-1}y(1+\triangt_0(y))}^{1/2}}}.
\end{equation*}
As all functions in this equation are $\Delta$-analytic, inserting it into~\eqref{eq:rasy} yields
\begin{align*}
  \triangr_g(\hat y) &\stackrel{\phantom{\eqref{eq:relation}}}{\cong} 2^{5(g-1)/2}\hat{c}_g \br{1-\rho_{\triangs}^{-1}\hat y}^{-5g/2+3/2}\br{1+O\br{\br{1-\rho_{\triangs}^{-1}\hat y}^{1/4}}}\\
  &\stackrel{\eqref{eq:relation}}{=} c_g \br{1-\rho_{\triangs}^{-1}\hat y}^{-5g/2+3/2}\br{1+O\br{\br{1-\rho_{\triangs}^{-1}\hat y}^{1/4}}}
\end{align*}
with $\hat y = y(1+\triangt_0(y))$. By \Cref{prop:planar-S}, this finishes the proof.
\end{proof}

\Cref{prop:R-asymptotic} yields following structural statement.

\begin{coro} \label{coro:R-asymptotic}
For $g\ge1$ the probability that a triangulation $R$ chosen uniformly at random from $\class{\triangr}_g(m)$
is simple is $1-O(m^{-1/4})$.
\end{coro}
\begin{proof}
We have $\triangs_g(y) \preceq \triangr_g(y)$ since $\class{\triangs}_g\subseteq\class{\triangr}_g$. By 
\Cref{prop:R-asymptotic}, both $\triangs_g(y)$ and $\triangr_g(y)$ are of the form $\br{1-\rho_S^{-1}y}^{-5g/2+3/2}(1+O(\br{1-\rho_S^{-1}y}^{1/4}))$, and the result follows by \Cref{tt}.
\end{proof}

As $\class{\triangs}_g\subseteq\class{\triangn}_g\subseteq\class{\triangr}_g$ we immediately obtain analogous results for 
$\class{\triangn}_g$.

\begin{prop} \label{prop:N-asymptotic}
  The generating functions $\triangn_g(y)$ and $\triangs_g(y)$ have the same dominant singularity $\rho_S$,
  $\triangn_0(y)=\triangs_0(y)$, and 
  \begin{equation*}
    \triangn_g(y)-\triangs_g(y)\cong O\br{\br{1-\rho_S^{-1}y}^{-5g/2+7/4}},\quad g\ge1.   
  \end{equation*}
\end{prop}

\begin{coro} \label{coro:N-asymptotic}
For $g\ge1$ the probability that a triangulation $N$ chosen uniformly at random from $\class{\triangn}_g(m)$
is simple is $1-O(m^{-1/4})$.
\end{coro}

In order to analyse the generating function $\triangm_g(y)$, we use a similar strategy as in \Cref{prop:R-asymptotic}.
The only difference between $\class{\triangm}_g$ and $\class{\triangn}_g$ is that for triangulations in
$\class{\triangm}_g$ specific types of loops are allowed. We will use surgeries to eliminate all loops from 
triangulations in $\class{\triangm}_g$. These surgeries will relate $\triangm_g(y)$ to a combination of $\triangn_g(y)$, 
$\triangr_g(y)$ and $\triangs_g(y)$. We will use this relation to prove that $\triangm_g(y)$ and $\triangs_g(y)$ have the 
same dominant term.

\begin{prop} \label{prop:M-asymptotic}
 The generating functions $\triangm_g(y)$ and $\triangs_g(y)$ have the same dominant singularity $\rho_S$,
   $\triangm_0(y)=\triangs_0(y)$, and 
  \[\triangm_g(y)-\triangs_g(y)\cong O\br{\br{1-\rho_S^{-1}y}^{-5g/2+7/4}},\quad g\ge1.\]
\end{prop}

As for \Cref{prop:R-asymptotic,prop:N-asymptotic} we obtain the following corollary.

\begin{coro} \label{coro:M-asymptotic}
For $g\ge1$ the probability that a triangulation $M$ chosen uniformly at random from $\class{\triangm}_g(m)$
is simple is $1-O(m^{-1/4})$.
\end{coro}

\begin{proof}[Proof of \Cref{prop:M-asymptotic}]
A triangulation in $\class{\triangm}_g$ either has no loops or at least one loop. Let $\class{L}_g$ be the subclass 
of $\class{\triangm}_g$ of triangulations with at least one loop. Then we have
\begin{equation*}
	\class{\triangm}_g=\class{\triangn}_g\uplus \class{L}_g,
\end{equation*}
as $\class{\triangn}_g$ is the subclass of $\class{\triangm}_g$ of triangulations with no loops (see \Cref{table:triang}). 
Since $\triangm_g(y)-\triangs_g(y)=\triangm_g(y)-\triangn_g(y)+(\triangn_g(y)-\triangs_g(y))$, \Cref{prop:N-asymptotic} 
implies that it suffices to prove that the 
generating function of ${\class{L}}_g$ satisfies \[L_g(y) \cong O\br{\br{1-\rho_S^{-1}y}^{-5g/2+7/4}}.\]

Let $M\in\class{L}_g$. We wish to eliminate the loops in $M$ by cutting along them and closing the resulting holes 
(in some manner). However, the triangulations resulting from cutting along loops (and closing the holes) do not have 
to be in $\class{\triangm}_g$, as planar double edges and separating pairs of loops can arise.
We therefore cut the surface \emph{recursively} along loops (in any fixed order) until every remaining loop bounds a hole.
Whenever we cut the surface, we root the new surface(s) in the following way. Before cutting along a loop $L$, consider
the root corner $(v,e,e')$ of the component that contains $L$. After cutting along $L$, we define the corner corresponding to
$(v,e,e')$ to be the root corner of its component. If $L$ separates the surface, then there is a unique component $S$ that
does not have a root yet. Let $L_S$ be the copy of $L$ on $S$. The edges incident with the base $u$ of $L_S$ are arranged
around $u$ in a sequence $L_S=e_1,e_2,\dotsc,e_{i-1},e_i=L_S$ in counterclockwise direction. We let $(u,e_1,e_2)$ be the
root corner of $S$.

Let $\overline{M}^{(1)},\ldots,\overline{M}^{(k)}$ be the triangulations of the components of the resulting map 
$\ol{M}$. Without loss of generality let $\ol{M}^{(1)}$ be the component containing (the corner corresponding to) the original root.
By construction, every hole is bounded by a loop, and vice versa every loop is the boundary
of a hole. Let $l$ be the number of loops we cut to obtain $\ol{M}$. As every cut leaves two holes, we have 
$2l$ holes; we denote them by $H_1,\ldots,H_{2l}$ so that for every $j$, $H_{2j-1}$ and $H_{2j}$
originate from cutting the $j$-th loop in our construction.
Let $g_i$ be the genus of the surface that $\overline{M}^{(i)}$ is embedded on. We denote the set of indices $j$ for
which $H_j$ is a hole in $\overline{M}^{(i)}$ by $J_i$ and set $h_i:=|J_i|$.
Let $\overline{\class{M}}_{g_i,J_i}$ be the class of all triangulations of the surface of genus $g_i$ with $h_i$ holes 
$H_j$, $j\in J_i$, that can occur by the previous surgery as a component $\overline{M}^{(i)}$.

In order to obtain an upper bound for $L_{g}(y)$ in terms of 
the generating functions $\overline{M}_{g_i,J_i}(y)$, we will have to incorporate the sum over all possible numbers of 
loops, components, and genera of those components. Our first step is to show that there are only finitely many 
possibilities for these numbers.

\begin{claim}\label{claim:finiteg}
	The tuple $\mathbf{g}:=(l,k;g_1,\ldots,g_k;h_1,\ldots,h_k)$ can attain only finitely many values.
\end{claim}

If $k=1$, then none of the loops was separating and hence cutting along each of the loops decreased the genus of the
surface by one, yielding $l\leq g$ and $g_1=g-l$. Now suppose $k\ge 2$. For each $\ol{M}^{(i)}$ we claim that 
$h_i\ge3$. Indeed, if $\ol{M}^{(i)}$ had only one hole, the loop bounding this hole would have been separating in $M$,
a contradiction to the fact that triangulations in $\class{M}_g$ do not have separating loops. If $\ol{M}^{(i)}$ had
precisely two holes, they either come from cutting the same loop or from cutting two different loops. In the former case,
this would imply $k=1$ as $M$ was connected. In the latter case, cutting along these two loops would disconnect
$\ol{M}^{(i)}$ from the other parts of $M$, violating the definition of $\class{\triangm}_g$. Therefore we have
$h_i\ge3$ for all $i$ and thus $l = (h_1 + \dots + h_k)/2 \geq 3k/2$.
 
Next we show that $k$ and $l$ are bounded from above. When we cut along loops, in each step we either
decrease the genus of the surface by one or we cut it into two parts. Since we ended up with $k$ components, among
$l$ loops that were cut there are exactly $k-1$ loops that created a new component when we cut along them. Thus for $l-k+1$ of the
loops cutting along them decreases the total genus and we have $g \geq l-k+1 \geq k/2+1$, implying
\begin{equation}\label{eq:k}
 k \leq 2(g-1)
\end{equation}
and, by inserting \eqref{eq:k} into $g\geq l-k+1$,
\begin{equation}\label{eq:l}
 l \leq 3(g-1).
\end{equation}
As $k$ and $l$ are bounded and
\begin{equation}\label{eq:g}
  g_1+\dots+g_k = g-(l-k+1)
\end{equation}
is smaller than $g$, there are only finitely many possible values for
$\mathbf{g}$. This finishes the proof of \Cref{claim:finiteg}.

\medskip

Our aim is to obtain upper bounds for $|\overline{\class{M}}_{g_i,J_i}|$ for any fixed $i$.
To do that we shall describe the triangulations that can occur in  
$\overline{\class{M}}_{g_i,J_i}$ in terms of triangulations in $\class\triangs_g$ and 
$\class\triangr_g$.
First we describe two special triangulations that will occur during the proof. Let $\Psi$ be the triangulation 
consisting of one vertex and three loops each bounding a hole. The other special triangulation 
consists of two vertices connected by an edge and a loop based bounding a hole at one of the vertices. We denote this 
triangulation by $\Lambda$ (see \Cref{fig:smallcomponents}). Note that $\Psi$ can occur as a component $\overline{M}^{(i)}$, while $\Lambda$ will
only appear after some surgeries.

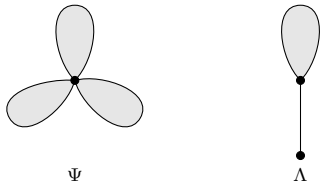
\begin{figure}[htbp]
  \centering
\begin{tikzpicture}[line cap=round,line join=round,>=triangle 45,x=1.0cm,y=1.0cm]
\clip(0.,0.) rectangle (4.5,3.);
\draw (4,0.5)-- (4,1.5);
\draw [fill=gray] (4,1.5) to[out=45,in=0] (4,2.5)to[out=180,in=135] (4,1.5);
\draw [fill=gray] (1,1.5) to[out=45,in=0] (1,2.5)to[out=180,in=135] (1,1.5);
\draw [fill=gray] (1,1.5) to[out=175,in=120] (0.13,1)to[out=300,in=255] (1,1.5);
\draw [fill=gray] (1,1.5) to[out=285,in=240] (1.87,1)to[out=60,in=15] (1,1.5);

\begin{scriptsize}
\draw [fill=black] (4,1.5) circle (1.5pt);
\draw [fill=black] (4,0.5) circle (1.5pt);
\draw [fill=black] (1,1.5) circle (1.5pt);
\draw[color=black] (4,0.25) node {$\Lambda$};
\draw[color=black] (1,0.25) node {$\Psi$};
\end{scriptsize}
\end{tikzpicture}
  \caption{The components of type $\Psi$ and $\Lambda$.}
  \label{fig:smallcomponents}
\end{figure}

Consider a component $\overline{M}^{(i)}$. It has genus $g_i$ and $h_i$ holes $H_j$, $j\in J_i$. 
Recursively for $j\in J_i$, starting at the smallest index, we look at the face $f_j$ incident
with the loop $L_j$ bounding the hole $H_j$. In each step we will perform surgeries for which all maps remain
triangulations. Thus, $f_j$ is bounded by a triangle $t_j$. One of its edges is $L_j$; the other edges are
either two loops, a double edge, or a single edge from the base of $L_j$ to a vertex of degree one. In the
second case, we additionally have to distinguish whether $t_j$ is the root face of its component. We thus have the following
four possibilities.
\begin{enumerate}
\item\label{triangle:Delta}
  $f_j$ is the unique face of a component $\Psi$;
\item\label{triangle:root}
  $f_j$ is the root face of its component and $t_j$ consists of $L_j$ and a double edge;
\item\label{triangle:nonroot}
  $f_j$ is not the root face and $t_j$ consists of $L_j$ and a double edge; or
\item\label{triangle:Lambda}
  $f_j$ is the unique face of a component $\Lambda$.
\end{enumerate}
In Cases~\ref{triangle:Delta}, \ref{triangle:root}, and~\ref{triangle:Lambda}, we do not perform any surgery
and proceed with the next index in $J_i$. In Case~\ref{triangle:nonroot}, we perform the following surgery.

Let $d_j$ be the double edge in $t_j$. We zip $d_j$, obtaining one copy $\Lambda_j$ of $\Lambda$ and a component $S$; denote the
zipped edge in $S$ by $e_j$. Now $S$ might contain planar double edges, but all such double edges are larger than $e_j$
in the sense of \Cref{def:relation}. Recursively pick a smallest planar double edge that does not contain the root face
of $S$ on the same side as $e_j$ and zip it. If such a planar double edge encloses an edge that carries a mark from an
earlier step of the construction, we let the corresponding edge in the planar part inherit the mark.
This iterative zipping results in a sequence $\sigma_j$ of simple planar 
triangulations, each marked at the zipped edge at which the previous triangulation was cut off. These components cannot 
contain (planar) double edges, as such double edges would enclose $e_j$ and thus would contradict the minimality 
of the double edges that have been zipped. Furthermore, the last zipped double edge in the remaining part of $S$ is marked.
If \ref{triangle:root} occurred for an index $j_0$, we first perform the above surgeries for all other $j\in J_i$. After 
these surgeries the face $f_{j_0}$ satisfies \ref{triangle:Lambda} if $g_i=0$ and \ref{triangle:root} otherwise. In the latter case we zip the double edge in $t_{j_0}$ to obtain a copy $\Lambda_{j_0}$ of $\Lambda$.

After performing these surgeries for all loops, every remaining planar double edge has the root face on its planar side; in particular,
the other side of the double edge is not planar. Thus, the remaining planar double edges are partially ordered by
\Cref{lem:planar-dbl-inclusion}. We proceed the same way as before, recursively zipping smallest double edges, thus
obtaining another sequence of simple planar triangulations.

We split $\ol{\class{M}}_{g_i,J_i}$ into the following sub-classes.

\renewcommand{\theenumi}{(\Alph{enumi})}
\begin{enumerate}
  \item Let $\class{A}_{g_i,J_i}$ be the subclass of $\ol{\class{M}}_{g_i,J_i}\setminus\{\Psi\}$ where the root face is
    incident with a loop.
  \item Let $\class{B}_{g_i,J_i}$ be the subclass of $\ol{\class{M}}_{g_i,J_i}$ where the root face is not incident with
    a loop.
\end{enumerate}
\renewcommand{\theenumi}{(\roman{enumi})}
Note that $\ol{M}^{(i)}$ cannot be $\Lambda$, as otherwise the corresponding loop would have been separating. Thus we can write
\[\ol{\class{M}}_{g_i,J_i}=\class{A}_{g_i,J_i}\uplus\class{B}_{g_i,J_i}\uplus\{\Psi\},\]
where $\Psi$ is only present if $g_i=0$ and $h_i=3$.

\begin{claim}\label{claim:nextroot}
 Let $A_{g_i,J_i}(y)$ be the generating function of $\class{A}_{g_i,J_i}$ \emph{where loops are not counted}. Then
 \begin{align*}
  A_{0,J_i}(y)&\preceq h_i 2^{h_i^2}\opa{y}^{h_i-2}\br{\frac{(3y)^{h_i}}{\br{1-\opa{y}\br{\triangs_0(y)}}^{h_i-1}}},
  &\\
  A_{g_i,J_i}(y)&\preceq h_i 2^{h_i^2}\opa{y}^{h_i-1}\br{\br{\frac{3y}{1-\opa{y}\br{\triangs_0(y)}}}^{h_i}\triangr_{g_i}(y)},
  &\text{for }g_i\ge 1.
\end{align*}
\end{claim}

Suppose $g_i=0$.
By the construction above we have one sequence $\sigma_j$ of marked simple planar triangulations for every loop that is not incident 
with the root face. At the start of each of those sequences there is one component $\Lambda_j$. Such a component is counted by $3y$,
since it has one single edge (and we do not count loops) and there are three possibilities for the root. Therefore we obtain a factor of 
\[\frac{(3y)^{h_i-1}}{\br{1-\opa{y}\br{\triangs_0(y)}}^{h_i-1}}.\]
In addition, we have a copy of $\Lambda$ containing the original root face. For each of the $h_i-1$ sequences,
there is one marked edge in one of the other components, thus there are $h_i-1$ additional marked edges. As
at least one of the marks (the last one) is on the copy of $\Lambda$ that contains the root face, we obtain 
\[\opa{y}^{h_i-2}\br{\frac{(3y)^{h_i-1}}{\br{1-\opa{y}\br{\triangs_0(y)}}^{h_i-1}}\opa{y}\br{3y}}.\]
This is an upper bound as not all markings are allowed, e.g. placing the marking of a sequence on the sequence itself. Going in the 
other direction we reattach the triangulations in the sequences $\sigma_j$ and the components $\Lambda_j$. We have
to bound the number of different maps in $\class{A}_{g_i,J_i}$ we can obtain by this construction. 

Marks and sequences are sorted, hence it is uniquely defined which mark belongs to which sequence. We have $h_i$ copies
$\Lambda_j$, $j\in J_i$, of $\Lambda$; deciding for which $j_0$ the component $\Lambda_{j_0}$ contains the original root face gives
us a factor of $h_i$. Let $j$ be the largest index in $J_i\setminus\{j_0\}$. Then the triangulations in the sequence $\sigma_j$ are
recursively attached to $\Lambda_{j_0}$. By \Cref{lem:surgery-2}, each such step is a 2-to-2 correspondence. If for a different
sequence $\sigma_{j'}$, its mark is placed in one of the triangulations of $\sigma_j$, then there are at most two possibilities
which edge carries this mark after reattaching all triangulations in $\sigma_j$. As this can happen at most once for each pair of 
sequences, the number of choices is bounded by $2^{h_i^2}$. Together with the trivial observation $\opa{y}\br{3y} = 3y$ we have
proved the first statement of \Cref{claim:nextroot}.

For $g_i\ge1$ we proceed analogously. Again, we have components $\Lambda_j$, $j\in J_i$, and for each $j\in J_i\setminus\{j_0\}$ 
we have a sequence of marked simple planar triangulations, one component $\Lambda_j$, and a marked edge where the sequence 
was attached to the rest of the map. In contrast to the case $g_i=0$, we also have a sequence $\sigma$ that originated from 
zipping planar double edges
between the root and the non-planar component $N$ in the last step of our construction. In the opposite direction, this
sequence is attached to $N$ at the root edge of $N$, thus there is no additional mark. The triangulation $N$ is of genus $g_i$,
has no holes, and has no planar double edges; thus $N\in\triangr_{g_i}$. In total, we have an upper bound of
\[\opa{y}^{h_i-1}\br{\br{\frac{3y}{1-\opa{y}\br{\triangs_0(y)}}}^{h_i}\triangr_{g_i}(y)}.\]
The factor $h_i 2^{h_i^2}$ results from the same considerations as in the planar case.

\medskip

\begin{claim}\label{claim:awayroot}
 Let $B_{g_i,J_i}(y)$ be the generating function of $\class{B}_{g_i,J_i}$ where loops are not counted. Then
 \begin{align*}
  B_{0,J_i}(y)&\preceq \br{h_i-1} 2^{h_i^2}\opa{y}^{h_i-2}\br{\br{\frac{3y}{1-\opa{y}\br{\triangs_0(y)}}}^{h_i}\br{\opa{y}^2\triangr_0(y)}},
  &\\
  B_{g_i,J_i}(y)&\preceq 2^{(h_i+1)^2}\opa{y}^{h_i}\br{\br{\frac{3y}{1-\opa{y}\br{\triangs_0(y)}}}^{h_i}\frac{\triangr_{g_i}(y)}{1-\opa{y}\br{\triangs_0(y)}}},
  &\text{for }g_i\ge 1.
\end{align*}
\end{claim}

Proving these formulas works analogously to the previous claim. For the planar case the difference is that we now 
have one sequence for every hole and the remaining triangulation $P$ is a planar triangulation without double edges (thus 
in $\class{\triangr}_0$). We have to mark $h_i$ edges, one for each sequence. We claim that at least two of these marks are
on $P$. Indeed, the mark from the last sequence $\sigma_j$ is on $P$. If this mark was the only mark on $P$, then the double
edge we zipped to obtain the last triangulation in $\sigma_j$ would have been separating in $M$, a contradiction. Thus we
have at least two marks on $P$. One of these two marks is the mark for the last sequence $\sigma_j$, 
for the other mark there are $h_i-1$ choices. The factor $2^{h_i^2}$ results from the same reasoning as before.

For the case $g_i\ge1$ we again obtain one sequence for each loop and one additional sequence for the part between the root 
and the non-planar component $N$. As this additional sequence is attached at the root edge of $N$, we do not 
need an additional mark. We thus have $h_i$ marks in total. As we have $h_i+1$ instead of $h_i$ sequences, the factor
$2^{h_i^2}$ changes to a factor $2^{(h_i+1)^2}$.

\medskip

\begin{claim}\label{claim:connecting}
 The generating function $L_g(y)$ of $\class{L}_{g}$ is bounded from above as follows.
 \begin{align}\label{eq:lg}
  L_g(y)\preceq \sum_{\mathbf{g}}\binom{2l}{h_1,h_2,\ldots,h_k}y^l(A_{g_1,J_1}(y)+B_{g_1,J_1}(y)+1)\prod_{i=2}^k(A_{g_i,J_i}(y)+1),
 \end{align}
 where the sum is over all possible values of $\mathbf{g}$.

\end{claim}

After cutting along the $l$ loops, the result is a set of triangulations with holes. The sum realises each possible
distribution of genera and numbers of holes to these components. The multinomial coefficient represents the choice of
how the indices of the holes are distributed among the components, while the factor $y^l$ counts the loops that were cut.
The component $\ol{M}^{(1)}$ containing the root might be in $\class{A}_{g_1,J_1}$, in $\class{B}_{g_1,J_1}$, or a copy of $\Psi$
(which is counted as $1$, since loops have been taken care of separately). All other components are rooted along a loop
around one of the holes. Therefore no component $\ol{M}^{(i)}$ with $i>1$ can be in $\class{B}_{g_1,J_1}$. This finishes
the proof of the claim.

\medskip

With these claims we are able to show that $L_g(y)$ is indeed of the claimed order.

\begin{claim}\label{claim:ab}
 For any $g_i\ge0$ and $h_i$ we have
\begin{align*}
 A_{g_i,J_i}(y)&\cong O\br{\br{1-\rho_S^{-1}y}^{-5g_i/2-h_i+5/2}},\\
 B_{g_i,J_i}(y)&\cong O\br{\br{1-\rho_S^{-1}y}^{-5g_i/2-h_i+3/2}},\\
 L_g(y)&\cong O\br{\br{1-\rho_S^{-1}y}^{-5g/2+2}},
 \end{align*}
 and thus in particular
 \begin{equation*}
 L_g(y) \cong O\br{\br{1-\rho_S^{-1}y}^{-5g/2+7/4}}.
 \end{equation*}
\end{claim}

To determine an upper bound for $A_{g_i,J_i}(y)$ and $B_{g_i,J_i}(y)$ from 
\Cref{claim:nextroot,claim:awayroot} we repeatedly use the fact that the function
$\frac{1}{1-\opa{y}\br{\triangs_0(y)}}$ is $\Delta$-analytic and of the form $c_0+O\br{\br{1-\rho_S^{-1}y}^{1/2}}$,
where $c_0$ is a constant. We obtain

\begin{align*}
 A_{0,J_i}(y)&\preceq c_1 \opa{y}^{h_i-2}\br{\br{c_0+O\br{\br{1-\rho_S^{-1}y}^{1/2}}}^{h_i-1}}\cong 
 O\br{\br{1-\rho_S^{-1}y}^{5/2-h_i}}.
\end{align*}
For $g_i>0$ we have
\begin{align*}
 A_{g_i,J_i}(y)&\preceq c_g \opa{y}^{h_i-1}\br{\br{\frac{3y}{1-\opa{y}\br{\triangs_0(y)}}}^{h_i}\triangr_{g_i}(y)}\\
	&\cong c_g \opa{y}^{h_i-1}\br{\br{c_0+O\br{\br{1-\rho_S^{-1}y}^{1/2}}}^{h_i}O\br{\br{1-\rho_S^{-1}y}^{5g_i/2+3/2}}}\\
	&=O\br{\br{1-\rho_S^{-1}y}^{-5g_i/2-h_i+5/2}}.
\end{align*}
Analogously we obtain 
\begin{align*}
 B_{0,J_i}&\cong O\br{\br{1-\rho_S^{-1}y}^{3/2-h_i}}\text{ and}\\
 B_{g_i,J_i}&\cong O\br{\br{1-\rho_S^{-1}y}^{-5g_i/2-h_i+3/2}}.
\end{align*}
Therefore we conclude the claim for $A_{g_i,J_i}(y)$ and $B_{g_i,J_i}(y)$.
By \Cref{claim:finiteg} the sum in \eqref{eq:lg} has only finitely many terms. Therefore $L_g(y)$ is congruent to 
the dominant term among the summands. For the choice of $k,l,g_1,\dotsc,g_k,h_1,\dotsc,h_k$ for which the summand is
the dominant term, we have by~\eqref{eq:g}
\begin{align*}
 L_g(y)&\cong O\br{\br{1-\rho_S^{-1}y}^{-5g_i/2-h_1+3/2}}\prod_{i=2}^{k}O\br{\br{1-\rho_S^{-1}y}^{-5g_i/2-h_i+5/2}}\\
	&=O\br{\br{1-\rho_S^{-1}y}^{-5\sum g_i/2-\sum h_i+5k/2-1}}\\
	&=O\br{\br{1-\rho_S^{-1}y}^{-5(g-l+k-1)/2-2l+5k/2-1}}\\&=O\br{\br{1-\rho_S^{-1}y}^{-5g/2+l/2+3/2}}.
\end{align*}
As we cut along at least one loop, the claim and thus the proposition follows.
\end{proof}

From \Cref{prop:M-asymptotic} and \Cref{congruent:meta} we immediately obtain the following.
\begin{coro}\label{M-asymptotic}
   The dominant singularity of $\triangm_g$ is given by $\rho_\triangm=\rho_\triangs=\frac{3}{2^{8/3}}$.
  The generating function $\triangm_0(y)$ is $\Delta$-analytic and satisfies
  \begin{align}
  \triangm_0(y)&=\frac{1}{8}-\frac{9}{16}\br{1-\rho_\triangs^{-1}y}+\frac{3}{2^{5/2}}\br{1-\rho_\triangs^{-1}y}^{3/2}
  +O\br{\br{1-\rho_\triangs^{-1}y}^{2}}.
  \end{align}
  For $g\ge1$ we have
  \begin{align}
    \triangm_g(y)&\cong c_g \br{1-\rho_\triangs^{-1} y}^{-5g/2+3/2}\br{1+O\br{\br{1-\rho_\triangs^{-1} y}^{1/4}}},
  \end{align}
  where $c_g$ is a constant depending only on $g$.
\end{coro}

\section{Cubic graphs}\label{graphs}

Unless stated otherwise, graphs are unrooted. Recall that in our generating functions $x$ marks vertices and 
$y$ marks edges. Additionally we will distinguish whether edges are single edges, double edges or loops as they are treated 
differently when obtaining relations between graph classes. We will use the variable $z$ to mark double edges and the
variable $w$ to mark loops. It is easy to see that $3$-connected cubic graphs are simple and that $2$-connected cubic
multigraphs do not contain loops. The generating functions for these classes will only feature the variables of edges
that can occur.

In order to derive asymptotic results we shall deal with univariate generating functions $F(v)$. As cubic (multi)graphs
always have $2n$ vertices and $3n$ edges for some $n\in\N$, we would like the coefficient $(2n)![v^n]F(v)$ to be the number
of graphs (or multigraphs or weighted multigraphs) in the corresponding class with $2n$ vertices and $3n$ edges. Such a
univariate generating function can be obtained by the following substitution.

\begin{definition}\label{def:univariate}
  Let $\class{F}$ be a class of connected cubic multigraphs without triple edges and let
  \begin{equation*}
    F(x,y,z,w) = \sum_{n,m,k,l\ge 0}\frac{f_{n,m,k,l}}{n!}x^ny^mz^kw^l
  \end{equation*}
  be its exponential generating function. 
  We define functions $F(v)$, $F^u(v)$, and $F^{s}(v)$ as follows.
  \begin{align*}
    F(v) &:= F\br{v^{1/4},v^{1/6},\frac{v^{1/3}}{2},\frac{v^{1/6}}{2}},
    &&\\
    F^u(v) &:= F(v^{1/4},v^{1/6},v^{1/3},v^{1/6}),
    &\text{and}&\\
    F^s(v) &:= F(v^{1/4},v^{1/6},0,0).
    &&
  \end{align*}
  If the generating function of $\class{F}$ does not involve $z$ or $w$, we define $F(v)$, $F^u(v)$, and $F^s(v)$
  analogously, only using the substitutions of those variables that occur.
\end{definition}

We claim that $(2n)![v^n]F(v)$ is the number of weighted multigraphs in $\class{F}(2n)$, i.e.\ the sum of $W(G)$ for all
$G\in\class{F}$ with $2n$ vertices (and thus with $3n$ edges). Indeed, if $G\in\class{F}(2n)$ has $k$ double edges,
$l$ loops, and $m$ single edges, then there are $2k+l+m = 3n$ edges in total and the substitution transforms the monomial
$x^{2n}y^mz^kw^l$ into $2^{-(k+l)}v^{n/2+m/6+k/3+l/6} = W(G)v^n$. Similarly, $(2n)![v^n]F^u(v)$ is the number of 
(unweighted) multigraphs in $\class{F}(2n)$. Finally, $(2n)![v^n]F^s(v)$ is the number of simple graphs in $\class{F}
(2n)$, since replacing $z$ and $w$ by $0$ ensures that no graphs with double edges or loops are counted in $F^s(v)$.

When we relate various classes of graphs, we will often mark edges, which corresponds to applying the 
operator $\opa{y}=y\deriv{y}$ to the 
generating function. As we apply singularity analysis to the univariate generating function defined above, we shall 
express applications of $\opa{y}$ by applications of $\opa{v}$. This is formalised by the following lemma.
\begin{lem}\label{substitution}
 Let $F(x,y)$ be the generating function of a class $\class{F}$ of connected cubic multigraphs. Then 
 \[3\opa{v}(F(y))=\left.\opa{y}(F(x,y))\right|_{x=v^{1/4},y=v^{1/6}}.\]
\end{lem}
\begin{proof}
 Since graphs in $\class{F}$ are cubic, we have $3\opa{x}(F(x,y))=2\opa{y}(F(x,y))$ and thus 
 \begin{align*}
	3\opa{v}(F(v))&=3v\derivf{F(v)}{v}=\left.\left(3v\derivf{F(x,y)}{x}\derivf{x}{v}+3v\derivf{F(x,y)}{y}\derivf{y}{v}\right)\right|_{x=v^{1/4},y=v^{1/6}}\\
	&=\left.\left(\frac34v^{1/4}\derivf{F(x,y)}{x}+\frac12v^{1/6}\derivf{F(x,y)}{y}\right)\right|_{x=v^{1/4},y=v^{1/6}}\\&=\left.\opa{y}(F(x,y))\right|_{x=v^{1/4},y=v^{1/6}}.\qedhere
 \end{align*}
\end{proof}

\subsection{From maps to graphs}

Let $\class{\three}_g$ be the class of $3$-connected cubic vertex-labelled graphs \emph{strongly embeddable}
on $\Sg$ and let $\three_g(x,y)$ be its generating function. 
In this section we provide some necessary properties of $\three_g(v)$. We will use the auxiliary classes $\ol{\class{\three}}_g$ of $3$-connected cubic \emph{edge-labelled} graphs 
strongly embeddable on $\Sg$, and $\ol{\class{\triangm}}_g$ of \emph{edge-labelled, unrooted} triangulations where 
the triangulations are in $\class{\triangm}_g$.

\begin{prop}\label{coro:3-conn-graphs}

The dominant singularity of $\vertex{\three}_g(v)$ is $\rho_{\three}=\rho_{\triangs}^3=\frac{27}{256}$ and we have the following congruences.
\begin{align*}
  \vertex{\three}_0(v)&\cong c_0\left( 1 - \rho_\three^{-1} v\right)^{5/2}+O\br{\br{1 - \rho_\three^{-1} v}^3},
  & &\\
  \vertex{\three}_1(v)&\cong c_1\log\left( 1 - \rho_\three^{-1}v\right)+O\br{\br{1 - \rho_\three^{-1} v}^{1/4}},
  &\text{and}&\\
  \vertex{\three}_g(v)&\cong c_g\left( 1 - \rho_\three^{-1} v \right)^{-5g/2+5/2}+O\br{\br{1 - \rho_\three^{-1}, v}^{-5g/2+11/4}},
  &\text{for }g\ge 2,&
\end{align*}
where the $c_g$ are constants.
\end{prop}

Applying \Cref{tt}, we immediately obtain the coefficients of $\three_g(v)$.

\begin{coro}
The coefficients of $\vertex{\three}_g(v)$ satisfy
\[ [v^n]\vertex{\three}_g(v) =\br{1+O\br{n^{-1/4}}} \ol{c}_g n^{5(g-1)/2-1} \rho_{\three}^{-n}, \]
where $\ol{c}_g$ is a constant.
\end{coro}

\begin{proof}[Proof of \Cref{coro:3-conn-graphs}]
 First we compare $\class{\triangm}_g$ and $\ol{\class{\triangm}}_g$. For each triangulation $M\in\class{\triangm}_g$ with 
 $m$ edges there are $m!$ possibilities of labelling its edges. Conversely, for a triangulation $\ol{M}\in\ol{\class{\triangm}}_g$
 there are $2m$ possibilities of rooting it. Therefore the exponential generating function 
 $\ol{\triangm}_g(y)$ of $\ol{\class{\triangm}}_g$ satisfies 
 \[ [y^m]\triangm_g(y) = 2m [y^m]\ol{\triangm}_g(y)\]
 and thus 
 \[\triangm_g(y)=2\opa{y}\ol{\triangm}_g(y).\]
 
 Every graph $G\in \ol{\class{\three}}_g$ has at least two (edge-labelled) $2$-cell embeddings. By \Cref{prop:3-conn-dual}, 
 the maps
 obtained this way are precisely the duals of the triangulations in $\ol{\class{\triangm}}_g$. As $y$ denotes the number of 
 edges in $\ol{\triangm}_g(v^{1/3})$ and $v$ denotes a third of the number of edges in $\ol{\three}_g(v)$, we obtain
 \begin{equation*} 
  2\ol{\three}_g(v) \preceq \ol{\triangm}_g(v^{1/3}).
 \end{equation*}
 We claim that for a cubic map $M$ on $\Sg$, its facewidth $\fw(M)$ is exactly the 
edgewidth $\ew(M^*)$ of the triangulation $M^*$ that is the dual of $M$. Indeed, an essential cycle of $M^*$ witnessing the edgewidth of $M^*$
corresponds to an essential circle on $\Sg$ that meets $M$ in $\ew(M^*)$ edges and no vertices, resulting in $\fw(M)\leq\ew(M^*)$. 
On the other hand, any two faces of $M$ that share a vertex also share an edge, as $M$ is cubic. Thus, 
there is an essential circle witnessing the facewidth of $M$ that meets $M$ only in edges. As this circle corresponds to an essential
cycle of $M^*$, we have $\fw(M)\geq\ew(M^*)$ and thus equality.
Since by \Cref{robertsonvitray}(iii) a $3$-connected graph embeddable on $\Sg$ with facewidth at least $2g+3$ has
exactly two embeddings, we have
\begin{equation*} 
2\ol{\three}^{\fw \geq 2g+3}_g(v) = \ol{\triangm}^{\ew \geq 2g+3}_g(v^{1/3}). 
\end{equation*}
 As obviously $\ol{\three}_g^{\fw \geq 2g+3}(v) \preceq \ol{\three}_g(v)$ we obtain the following chain of inequalities.
 \begin{equation}\label{eq:G-to-M}
  \ol{\triangm}^{\ew \geq 2g+3}_g(v^{1/3})=2\ol{\three}^{\fw \geq 2g+3}_g(v)\preceq 2\ol{\three}_g(v)
	\preceq \ol{\triangm}_g(v^{1/3}).
 \end{equation}

 Since there are no double edges in a $3$-connected cubic graph, we know that the two generating functions 
$\ol\three_g(v)$ and $\three_g(v)$ are closely related. To be precise, 
$(2n)![v^n]\three_g(v)$ is the number of vertex-labelled graphs in $\class{\three}_g(2n)$. Since every such graph has $3n$ 
edges, $(3n)!(2n)![v^n]\three_g(v)$ is the number of $3$-connected cubic graphs with $2n$ vertices embeddable on $\Sg$
with both vertices and edges labelled. As this number is equal to $(2n)!(3n)!\ol{\three}_g(v)$ by an analogous argument, we 
have $[v^n]\three_g(v)=[v^n]\ol{\three}_g(v)$.

Therefore we can replace 
$\ol{\three}_g(v)$ by $\three_g(v)$ in \eqref{eq:G-to-M}. We obtain
\begin{align*}
 \three_g(v)\preceq\frac{1}{2}\ol{\triangm}_g(v^{1/3})=\left.\frac{1}{4}\int t^{-1}\triangm_g(t)\dd t
	\right|_{t=v^{1/3}}
\end{align*}
 and by \Cref{integrating} we obtain an upper bound for $\three_g(v)$ as claimed. To finish the proof we will show 
 the following claim.
 \setcounter{claim}{0}
 \begin{claim}\label{claim:largeew}
  The generating functions $\triangm_g(y)^{\ew \geq 2g+3}$ and $\triangm_g(y)$ have the same dominant singularity and
\[\triangm_g(y)-\triangm_g^{\ew \geq 2g+3}(y)\cong O\br{\br{1-\rho_S^{-1}y}^{-5g/2+7/4}}.\]
 \end{claim}
  Before we proof the claim, let us note that \Cref{coro:3-conn-graphs} follows immediately from \Cref{claim:largeew}, \Cref{congruent:meta}, and \Cref{M-asymptotic}.
 
A more general statement than \Cref{claim:largeew} was proven in \cite{representativity} for a variety of map classes. 
Although we believe that the proof in \cite{representativity} generalises to $\class{\triangm}_g$, which was not considered in 
\cite{representativity}, we give a slightly different proof here.
 
 The generating function of $\class{\triangm}_g \setminus \class{\triangs}_g$
 is congruent to $O\br{\br{1-\rho_S^{-1}y}^{-5g/2+7/4}}$ by \Cref{prop:M-asymptotic}. It thus suffices to show that
 \[\triangs_g^{\ew\le2g+2} \cong O\br{\br{1-\rho_S^{-1}y}^{-5g/2+7/4}}.\]
For $i\ge 3$ let $\class\triangs^{C=i}_g$ be the class of triangulations in $\class\triangs_g$ where one non-contractible 
cycle of length $i$ is marked and denote its generating function by $\triangs^{C=i}_g(y)$. Clearly
$\triangs^{\ew=i}_g(y) \preceq \triangs^{C=i}_g(y)$. Let $M\in\class\triangs_g^{C=i}$ and let $C$ be the marked cycle of 
$M$. We cut along $C$ and close the two resulting holes by inserting discs. For each of the two discs,
we then add a new vertex triangulating the disc. If $C$ was separating, we mark one of the corners at the new vertex in the
component that contains the original root face of $M$. For the other component, we choose one of the corners at the new vertex
to be its root. If $C$ was not separating, then we mark one corner at each of the two new vertices. In total we add $3i$ edges to the map.
This surgery results in
\begin{itemize}
\item two triangulations $M^{(1)}, M^{(2)}$, where $M^{(1)}$ contains the original root face of $M$ and a marked corner or
\item one triangulation $M^*$ with two marked corners.
\end{itemize}
All resulting triangulations are in $\class\triangs_{g'}$ for some $g'$, because the surgery does not create loops or double edges.
Thus, in the first case $M^{(1)} \in \class\triangs_{g_1}$ and $M^{(2)} \in \class\triangs_{g_2}$ (disregarding markings) 
with $g_1 + g_2 = g$ and $g_1, g_2 \geq 1$. In the second case $M^* \in \class\triangs_{g-1}$ (disregarding markings). 

Since a corner $(v_0,e,e')$ is uniquely defined once $v_0$ and $e$ are given, marking a corner is equivalent to marking
an edge and choosing one of its end vertices. In terms of generating functions, this corresponds to applying the operator $\opa{y}=y\deriv{y}$ with 
an additional factor of two. Similarly to previous proofs we will mark recursively, which will result in overcounting.
Since we added $3i$ edges to $M$ by our construction, we have to compensate this by a factor of $y^{3i}$.
Therefore we obtain the relation
\[ y^{3i}\triangs^{C=i}_g(y) \preceq 4\opa{y}^2\br{\triangs_{g-1}(y)} + 
\sum_{\substack{g_1 + g_2 = g \\ g_1,g_2 \geq 1}} 2\opa{y}\br{\triangs_{g_1}(y)} \triangs_{g_2}(y).\]

By \Cref{simplerefined}, we know that 
\[4\opa{y}^2\br{\triangs_{g-1}(y)} + \sum_{\substack{g_1 + g_2 = g \\ g_j \geq 1}} 
 2\opa{y}\br{\triangs_{g_1}(y)} \triangs_{g_2}(y)\cong O\br{\br{1-\rho_S^{-1}y}^{-5g/2+7/4}}.
\]
Because 
\[ \triangs_g^{\ew\le2g+2}(y) = \sum_{i=3}^{2g+2}\triangs_g^{\ew=i}(y) \preceq \sum_{i=3}^{2g+2}\triangs^{C=i}_g(y), \]
this completes the proof of the claim.
\end{proof}

\subsection{From 3-connected graphs to connected multigraphs}
In this section we derive relations between different classes of cubic multigraphs in the form of dominance.
In the end we will relate connected cubic multigraphs via $2$-connected cubic multigraphs to $3$-connected cubic 
graphs enumerated in the previous section. 

Let $\class{\three}_g$, $\class{\two}_g$, and 
$\class{\one}_g$ be the classes of $3$-connected, $2$-connected, and connected vertex-labelled cubic multigraphs 
strongly embeddable on $\Sg$ with generating functions $\three_g(x,y)=\sum \frac{d_{n,m}}{n!}x^ny^m$, 
$\two_g(x,y,z)=\sum \frac{b_{n,m,k}}{n!}x^ny^mz^k$ and $\one_g(x,y,z,w)=\sum \frac{c_{n,m,k,l}}{n!}x^ny^mz^kw^l$, 
respectively. In the generating function $\one_g(x,y,z,w)$ we will not account for the graph consisting of two vertices 
connected by three edges. This graph will be accounted for separately in the end.

First we give a relation between a subclass of $\class{\three}_g$ and a subclass of $\class{\two}_g$. 
To do this we need the class $\class{\network}$ of edge-rooted 2-connected  labelled cubic planar multigraphs, 
called \emph{networks}. In the exponential generating function $N(x,y,z)$ of $\class{N}$ we mark the root always 
with $y$ as a single edge and double edges without the root edge with $z$.

\begin{lem}\label{2conngf}
For $g\ge 1$ the generating functions of $\class\three_g$ and $\class\two_g$ satisfy
\begin{align}\label{eq:threetwo}
\vertex{\three}_g^{\fw\ge3}(x,\oy)&-\vertex{\three}_0(x,\oy)\preceq
\vertex{\two}_g^{\fw\ge3}(x,y,z)\preceq \vertex{\three}_g^{\fw\ge3}(x,\oy),
\end{align}
where $\oy=y(1+N(x,y,z))$.
\end{lem}

\begin{proof}
Let $B$ be a multigraph in $\class{\two}_g^{\fw\ge3}$. We show that it is counted at least once on the 
right-hand side and at most once on the left-hand side of \eqref{eq:threetwo}. 

First, suppose that $B$ is not planar. Then \Cref{coro:rv}\ref{rvc2} states that $B$ has a unique $3$-connected component $T$ 
strongly embeddable on $\Sg$ with the same facewidth. $T$ is in $\class{\three}_g^{\fw\ge3}$ and therefore counted 
once in $\three_g^{\fw\ge3}(x,y)$. 
To get $B$ from $T$ we have to attach $2$-connected components along the edges.
That means, either we leave an edge as it is (obtaining a summand of $y$) or we replace it by two edges (obtaining a 
factor of $y^2$) and one multigraph in $\class{\network}$ without its root edge 
(obtaining a factor of $\frac{1}{y}N$). Thus $B$ is counted exactly once on the right-hand side of \eqref{eq:threetwo}.

If $B$ is planar, then it might be counted more than once on the right-hand side. Indeed, in this case the $2$-connected components 
carrying the facewidth might be different for 
different embeddings.
Therefore $\vertex{\three_g}^{\fw\ge3}(x,y+y\vertex{\network}(x,y,z))$ is an upper bound. To get a lower bound we have to 
subtract all multigraphs we over counted. 
This is achieved by subtracting $\three_0(x,y+y\network(x,y,z))$, as only planar multigraphs are over counted and each such multigraph 
is subtracted once for each of its $3$-connected components.
\end{proof}

In the same spirit we can relate connected and $2$-connected multigraphs using the auxiliary class 
$\class{Q}$ of all edge-rooted connected vertex-labelled cubic planar multigraphs where the root edge is a loop. To simplify 
the formulas later on, the root will be marked by $y$ in the generating function $Q(x,y,z,w)$ and only non-root 
loops are marked by $w$.

\begin{lem}\label{conngf}
For $g\ge1$ the generating functions of $\class\one$ and $\class\two$ satisfy the following relation.
\begin{align}\label{eq:conngf}
\two_g^{\fw\ge2}(x,\oy,\oz)-\two_0(x,\oy,\oz)&\preceq\one_g^{\fw\ge2}(x,y,z,w)\preceq \two_g^{\fw\ge2}(x,\oy,\oz),
\end{align}
where $\oy=\frac{y}{1-Q(x,y,z,w)}$ and $\oz=\frac{1}{2}(\frac{y}{1-Q(x,y,z,w)})^2$.
\end{lem}

\begin{proof}
Let $C\in\class{\one}_g^{\fw\ge2}$. We shall show it is counted at least once on the 
right-hand side and at most once on the left-hand side of \eqref{eq:conngf}. 

First, suppose $C$ is not planar. Then \Cref{coro:rv}\ref{rvc1} states that $C$ has a unique $2$-connected component $B$ 
strongly embeddable on $\Sg$ with the same facewidth, i.e.\ $B\in\class{\two}_g^{\fw\ge2}$. 
To get $C$ from $B$ we have to replace each edge by a sequence of edges and multigraphs in $\class{Q}$, 
that means we replace one edge by a sequence of alternately edges and multigraphs in $\class{Q}$ without the root, 
starting and ending with an edge. Therefore the replacement is $y\mapsto y\frac{1}{1-Q(x,y,z,w)}$.

This results in a $1$-to-$1$ correspondence between the two generating functions $\vertex{\one_g}^{\fw\ge2}(x,y,z,w)$ and $\vertex{\two_g}^{\fw\ge2}
(x,\oy,\oz)$ for non-planar multigraphs. The replacement for double edges results from replacing a set of two edges 
each as above.

As in \Cref{2conngf}, if $C$ is planar, the above argument does not necessarily result in a bijection. Thus we have to 
subtract all corresponding planar multigraphs again in order to avoid overcounting. Therefore we get the claimed result 
analogously to the previous lemma.
\end{proof}

\Cref{2conngf,conngf} give relations only for multigraphs with large facewidth. Again, almost all multigraphs have large facewidth, 
which we show in the following lemma.

\begin{lem}\label{lem:small-fw}
For $g\ge 1$ the following relations hold.
\begin{align}
	\vertex{\two_g}^{\fw=2}(x,&y,z)\frac12\br{y+\frac{z}{y}}^2\br{\frac1y+\frac yz}^2\preceq \opb{y}
			{z}^2\br{\vertex{\two_{g-1}^{\fw\ge2}}(x,y,z)}\nonumber\\
		&+\sum_{g'=1}^{g-1}\opb{y}{z}\br{\two_{g'}^{\fw\ge2}(x,y,z)}\opb{y}{z}\br{\two_{g-g'}^{\fw\ge2}
			(x,y,z)},\label{eq:2-conn-graphs-fw2}\\
	\vertex{\one_g}^{\fw=1}(x,&y,z,w)\br{xyw}^2\br{\frac1y+\frac yz}\preceq \opa{w}^2\br{\vertex{\one_{g-1}}
			(x,y,z,w)}\nonumber\\
		&+\sum_{g'=1}^{g-1}\opa{w}\br{\vertex{\one_{g'}}(x,y,z,w)}\opa{w}\br{\vertex{\one_{g-g'}}(x,y,z,w)}.
			\label{eq:conn-graphs-fw1}
\end{align}
\end{lem}

\begin{proof}
In order to show \eqref{eq:2-conn-graphs-fw2}, let $B$ be a multigraph in $\class{\two}_g^{\fw=2}$. Consider a 
fixed $2$-cell embedding $M$ of $B$ on $\Sg$ with facewidth two
and let $\{e_1=\{v_1,w_1\},e_2=\{v_2,w_2\}\}$ be two edges such that there exists an 
essential circle $C$ on $\Sg$ meeting $M$ only in $e_1$ and $e_2$. Note that $e_1$, $e_2$ do not share vertices, because 
otherwise the facewidth would have been one. Then we delete $e_1$ and $e_2$, cut the surface along $C$ and 
close both holes with a disc (see \Cref{fig:cut}).
\begin{figure}
\begin{tikzpicture}[line cap=round,line join=round,>=triangle 45,x=1.0cm,y=1.0cm,scale=0.85]
\clip(-0.1,-3.) rectangle (13.5,0.5);
\draw [shift={(2.5,-11.5)},color=black]  plot[domain=1.3:1.9,variable=\t]({1.*9.487*cos(\t r)+0.*9.487*sin(\t r)},{0.*9.487*cos(\t r)+1.*9.487*sin(\t r)});
\draw [shift={(2.5,9.)},color=black]  plot[domain=4.35:4.98,variable=\t]({1.*9.487*cos(\t r)+0.*9.487*sin(\t r)},{0.*9.487*cos(\t r)+1.*9.487*sin(\t r)});
\draw [color=black] (0.,-0.5)-- (1.,-0.75);
\draw [color=black] (0.,-1.)-- (1.,-0.75);
\draw [color=black] (1.,-0.75)-- (4.,-0.75);
\draw [color=black] (4.,-0.75)-- (5.,-0.5);
\draw [color=black] (4.,-0.75)-- (5.,-1.);
\draw [color=black] (0.,-1.5)-- (1.,-1.75);
\draw [color=black] (1.,-1.75)-- (0.,-2.);
\draw [color=black] (1.,-1.75)-- (4.,-1.75);
\draw [color=black] (4.,-1.75)-- (5.,-1.5);
\draw [color=black] (4.,-1.75)-- (5.,-2.);
\draw [color=black] (12.,-0.75)-- (13.,-0.5);
\draw [color=black] (12.,-0.75)-- (13.,-1.);
\draw [color=black] (12.,-1.75)-- (13.,-1.5);
\draw [color=black] (12.,-1.75)-- (13.,-2.);
\draw [color=black] (8.,-0.5)-- (9.,-0.75);
\draw [color=black] (9.,-0.75)-- (8.,-1.);
\draw [color=black] (8.,-1.5)-- (9.,-1.75);
\draw [color=black] (9.,-1.75)-- (8.,-2.);
\draw [shift={(8.825,-1.25)},color=black]  plot[domain=-1.395:1.3928,variable=\t]({1.*0.99*cos(\t r)+0.*0.99*sin(\t r)},{0.*0.99*cos(\t r)+1.*0.99*sin(\t r)});
\draw [shift={(12.17,-1.25)},color=black]  plot[domain=1.738:4.540,variable=\t]({1.*0.99*cos(\t r)+0.*0.99*sin(\t r)},{0.*0.99*cos(\t r)+1.*0.99*sin(\t r)});
\draw [shift={(10,-11.5)},color=black]  plot[domain=1.249:1.358,variable=\t]({1.*9.487*cos(\t r)+0.*9.487*sin(\t r)},{0.*9.487*cos(\t r)+1.*9.487*sin(\t r)});
\draw [shift={(11,-11.5)},color=black]  plot[domain=1.783:1.892,variable=\t]({1.*9.487*cos(\t r)+0.*9.487*sin(\t r)},{0.*9.487*cos(\t r)+1.*9.487*sin(\t r)});
\draw [shift={(11,9.)},color=black]  plot[domain=4.390:4.5,variable=\t]({1.*9.487*cos(\t r)+0.*9.487*sin(\t r)},{0.*9.487*cos(\t r)+1.*9.487*sin(\t r)});
\draw [shift={(10,9.)},color=black]  plot[domain=4.925:5.034,variable=\t]({1.*9.487*cos(\t r)+0.*9.487*sin(\t r)},{0.*9.487*cos(\t r)+1.*9.487*sin(\t r)});
\draw [dash pattern=on 1pt off 1pt,color=black] (9.,-1.75)-- (9.,-0.75);
\draw [dash pattern=on 1pt off 1pt,color=black] (12.,-0.75)-- (12.,-1.75);
\draw [color=red] (2.5,-0.5) arc (90:270:0.263cm and 0.763cm);
\draw [color=red,dash pattern=on2pt off 2pt] (2.5,-2.) arc (270:450:0.263cm and 0.763cm);
\draw [color=red](3,-1.3) node {C};

\draw[->](5.5,-1.25) -- (7.5,-1.25);

\begin{scriptsize}
\draw [fill=black] (1.,-0.75) circle (1.5pt);
\draw[color=black] (1.,-1) node {$v_1$};
\draw [fill=black] (1.,-1.75) circle (1.5pt);
\draw[color=black] (1.,-1.48) node {$v_2$};
\draw [fill=black] (4.,-0.75) circle (1.5pt);
\draw[color=black] (3.8,-1) node {$w_1$};
\draw [fill=black] (4.,-1.75) circle (1.5pt);
\draw[color=black] (3.8,-1.48) node {$w_2$};
\draw [fill=black] (0.,-0.5) circle (1.5pt);
\draw [fill=black] (0.,-1.) circle (1.5pt);
\draw [fill=black] (0.,-1.5) circle (1.5pt);
\draw [fill=black] (0.,-2.) circle (1.5pt);
\draw [fill=black] (5.,-0.5) circle (1.5pt);
\draw [fill=black] (5.,-1.) circle (1.5pt);
\draw [fill=black] (5.,-1.5) circle (1.5pt);
\draw [fill=black] (5.,-2.) circle (1.5pt);
\draw[color=black] (1.67,-1) node {$e_1$};
\draw[color=black] (1.67,-1.58) node {$e_2$};
\draw [fill=black] (9.,-0.75) circle (1.5pt);
\draw[color=black] (9.3,-0.72) node {$v_1$};
\draw [fill=black] (9.,-1.75) circle (1.5pt);
\draw[color=black] (9.3,-1.72) node {$v_2$};
\draw [fill=black] (12.,-0.75) circle (1.5pt);
\draw[color=black] (11.75,-0.72) node {$w_1$};
\draw [fill=black] (12.,-1.75) circle (1.5pt);
\draw[color=black] (11.75,-1.72) node {$w_2$};
\draw [fill=black] (13.,-0.5) circle (1.5pt);
\draw [fill=black] (13.,-1.) circle (1.5pt);
\draw [fill=black] (13.,-1.5) circle (1.5pt);
\draw [fill=black] (13.,-2.) circle (1.5pt);
\draw [fill=black] (8.,-0.5) circle (1.5pt);
\draw [fill=black] (8.,-1.) circle (1.5pt);
\draw [fill=black] (8.,-1.5) circle (1.5pt);
\draw [fill=black] (8.,-2.) circle (1.5pt);
\draw[color=black] (9.25,-1.25) node {$e_3$};
\draw[color=black] (11.8,-1.25) node {$e_4$};
\end{scriptsize}
\end{tikzpicture}
\caption{Surgery along an essential circle.}\label{fig:cut}
\end{figure}
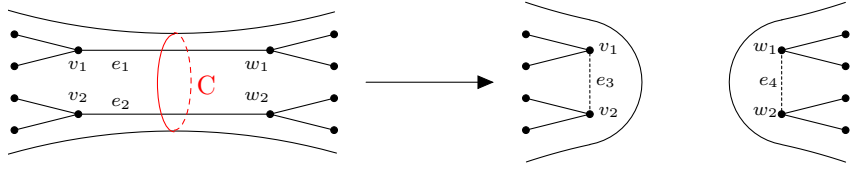
By this surgery we either disconnect the surface or we reduce its genus by one.

\emph{Case 1:} Cutting along $C$ disconnects the surface. As $C$ was an essential loop, both components have genus at least one. 
Therefore we obtain two multigraphs $B_1^*$ and $B_2^*$, strongly embeddable on $\torus{g_1}$ and $\torus{g_2}$, 
respectively, with $g_1,g_2\ge1$ and $g_1+g_2=g$. Without loss of generality we can assume that  
$v_1,v_2\in B_1^*$ and $w_1,w_2\in B_2^*$.
Furthermore $\{e_1,e_2\}$ was a $2$-edge-separator in $B$. Thus, $B_1^*$ and $B_2^*$ are connected as $B$ is $2$-edge-connected.
Let $B_1$ be obtained from $B_1^*$ by adding an edge $e_3=\{v_1,v_2\}$ and marking $e_3$. Note that $B_1$ is also 
strongly embeddable on $\torus{g_1}$. We claim that $B_1$ is $2$-connected. Indeed, any path in $B$ between vertices in 
$B_1$ gives rise to a path in $B_1$ between the same vertices by replacing any sub-path in $B\setminus B_1$ by the edge 
$e_3$. Thus $B_1$ is $2$-connected as $B$ is. So $B_1$ has to be 
$2$-connected as well. Analogously we add the edge $e_4=\{w_1,w_2\}$ to $B_2^*$ to obtain a $2$-connected multigraph $B_2$ 
strongly embeddable on $\torus{g_2}$. 
We also mark $e_4$. Thus we can conclude that in this case a multigraph $B$ can be 
constructed from a $2$-connected multigraph embeddable on $\torus{g'}$ with one marked edge and a $2$-connected 
multigraph embeddable on $\torus{g-g'}$ with one marked edge, resulting in the term
\[\sum_{g'=1}^{g-1}\opb{y}{z}\br{\vertex{\two_{g'}}(x,y,z)}\opb{y}{z}\br{\vertex{\two_{g-g'}}(x,y,z)}.\]
Note that $e_3$ and $e_4$ might be single edges or part of double edges and therefore differentiating with respect
to both results in an upper bound. 
The factor $\br{\frac{1}{y}+\frac{y}{z}}^2$ accounts for the deletion of $e_1$ and $e_2$, each of which might have 
been a single edge or part of a double edge (hence deleting it turns a double edge into a single edge). The factor 
$\br{y+\frac{z}{y}}^2$ represents the insertion of $e_3$ and $e_4$, each of which either adds a single edge or turns a 
single edge into a double edge.
Furthermore we obtain a factor of two for the ways to obtain the original multigraph from $B_1$ and $B_2$.

\emph{Case 2:} Cutting along $C$ does not disconnect the surface. As the embedding after cutting is still a $2$-cell embedding, $B\setminus\{e_1,e_2\}$ is connected.
We can connect $v_1$, $v_2$, $w_1$, $w_2$ in $B\setminus\{e_1,e_2\}$ by two edges (without loss of generality 
$e_3=\{v_1,v_2\}$ 
and $e_4=\{w_1,w_2\}$) so as to obtain a multigraph $B^*$. The graph $\ol{B}=B\cup\{e_3,e_4\}$ has a $2$-cell embedding 
$\ol{M}$ on $\Sg$ such that $e_1\cup e_2\cup e_3\cup e_4$ bounds a face. Indeed, starting from $M$, $e_3$ and $e_4$ 
can be embedded so that they run close to $e_1$, $e_2$, and $C$. Let $M^*$ be the embedding of $B^*$ induced by $\ol{M}$. 
Suppose $B^*$ is not $2$-connected, that is, it has a bridge $e$. Note that $e$ cannot be $e_3$ or 
$e_4$ as $B^*\setminus\{e_3,e_4\}=B\setminus\{e_1,e_2\}$ is connected.

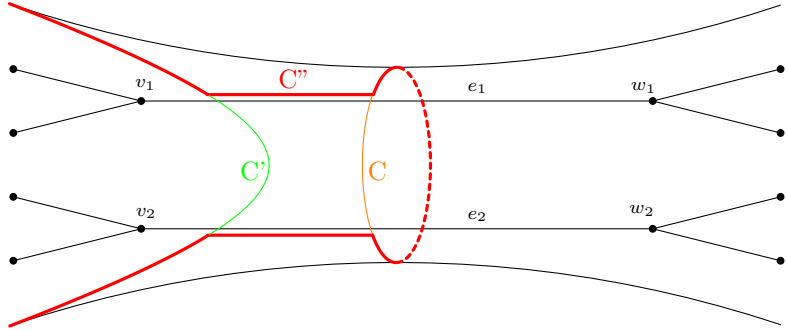
\begin{figure}
\begin{tikzpicture}[line cap=round,line join=round,>=triangle 45,x=2.0cm,y=2.0cm,scale=0.85]
\clip(-0.5,-3.) rectangle (6.5,0.5);
\draw [shift={(3.,-11.5)},color=black]  plot[domain=1.2490457723982544:1.892546881191539,variable=\t]({1.*9.486832980505138*cos(\t r)+0.*9.486832980505138*sin(\t r)},{0.*9.486832980505138*cos(\t r)+1.*9.486832980505138*sin(\t r)});
\draw [shift={(3.,9.)},color=black]  plot[domain=4.3906384259880475:5.034139534781332,variable=\t]({1.*9.486832980505138*cos(\t r)+0.*9.486832980505138*sin(\t r)},{0.*9.486832980505138*cos(\t r)+1.*9.486832980505138*sin(\t r)});
\draw [color=black] (0.,-0.5)-- (1.,-0.75);
\draw [color=black] (0.,-1.)-- (1.,-0.75);
\draw [color=black] (1.,-0.75)-- (5.,-0.75);
\draw [color=black] (5.,-0.75)-- (6.,-0.5);
\draw [color=black] (5.,-0.75)-- (6.,-1.);
\draw [color=black] (0.,-1.5)-- (1.,-1.75);
\draw [color=black] (1.,-1.75)-- (0.,-2.);
\draw [color=black] (1.,-1.75)-- (5.,-1.75);
\draw [color=black] (5.,-1.75)-- (6.,-1.5);
\draw [color=black] (5.,-1.75)-- (6.,-2.);
\draw [color=red, very thick] (1.52,-0.7)-- (2.81,-0.7);
\draw [color=red, very thick] (1.52,-1.8)-- (2.81,-1.8);

\draw[color=orange](2.81,-0.7) arc (134:226:0.5266143949538169cm and 1.5268229493056935cm);
\draw[color=red,very thick](2.9990944015573464,-0.48640228565126853) arc (90:134:0.5266143949538169cm and 1.5268229493056935cm);
\draw[color=red,very thick](2.81,-1.8) arc (226:270:0.5266143949538169cm and 1.5268229493056935cm);
\draw[color=red,very thick,dash pattern=on 2pt off 2pt](2.9990944015573464,-2.013225234956962) arc (270:450:0.5266143949538169cm and 1.5268229493056935cm);
\draw [samples=50,domain=-0.27:0.27,rotate around={179.94666265828377:(5.0469120090352,-1.2551517251060296)},xshift=10.0938240180704cm,yshift=-2.510303450212059cm,color=green] plot ({3.0469063708808277*(1+(\x)^2)/(1-(\x)^2)},{0.9466212588310281*2*(\x)/(1-(\x)^2)});
\draw [samples=50,domain=-0.5:-0.27,rotate around={179.94666265828377:(5.0469120090352,-1.2551517251060296)},xshift=10.0938240180704cm,yshift=-2.510303450212059cm,color=red, very thick] plot ({3.0469063708808277*(1+(\x)^2)/(1-(\x)^2)},{0.9466212588310281*2*(\x)/(1-(\x)^2)});
\draw [samples=50,domain=0.27:0.5,rotate around={179.94666265828377:(5.0469120090352,-1.2551517251060296)},xshift=10.0938240180704cm,yshift=-2.510303450212059cm,color=red, very thick] plot ({3.0469063708808277*(1+(\x)^2)/(1-(\x)^2)},{0.9466212588310281*2*(\x)/(1-(\x)^2)});
\draw [color=orange](2.7,-1.15) node[anchor=north west] {C};
\draw [color=green](1.7,-1.15) node[anchor=north west] {C'};
\draw [color=red](2,-0.43) node[anchor=north west] {C''};
\begin{scriptsize}
\draw [fill=black] (1.,-0.75) circle (1.5pt);
\draw[color=black] (1.0335548589070693,-0.620620382783603) node {$v_1$};
\draw [fill=black] (1.,-1.75) circle (1.5pt);
\draw[color=black] (1.0384686251057282,-1.6279424535086915) node {$v_2$};
\draw [fill=black] (5.,-0.75) circle (1.5pt);
\draw[color=black] (4.9301714544436654,-0.6451892137768979) node {$w_1$};
\draw [fill=black] (5.,-1.75) circle (1.5pt);
\draw[color=black] (4.915430155847688,-1.6082873887140556) node {$w_2$};
\draw [fill=black] (0.,-0.5) circle (1.5pt);
\draw [fill=black] (0.,-1.) circle (1.5pt);
\draw [fill=black] (0.,-1.5) circle (1.5pt);
\draw [fill=black] (0.,-2.) circle (1.5pt);
\draw [fill=black] (6.,-0.5) circle (1.5pt);
\draw [fill=black] (6.,-1.) circle (1.5pt);
\draw [fill=black] (6.,-1.5) circle (1.5pt);
\draw [fill=black] (6.,-2.) circle (1.5pt);
\draw[color=black] (3.6305288734712478,-0.65) node {$e_1$};
\draw[color=black] (3.6305288734712478,-1.65) node {$e_2$};
\end{scriptsize}
\end{tikzpicture}
\caption{Finding an essential circle witnessing small facewidth}\label{fig:newcurve}
\end{figure}

There is a (not necessarily essential) circle $C'$ on $\Sg$ hitting $M^*$ only in $e$. As 
$e$ has not been a bridge in $B$, $C'$ has to meet $e_1$ and $e_2$ as well. If it met neither 
$e_1$ nor $e_2$ it would either contradict $B$ having facewidth two (if $C'$ is essential) or the $2$-connectivity of $B$ 
(if $C'$ is not essential). If it met only one of them, it would 
have to meet one of $e_3$, $e_4$, because $e_1$, $e_2$, $e_3$ and $e_4$ bound a disc in $\ol{M}$. 
This contradicts the fact that $C'$ meets $M^*$ only in $e$. 

We now construct the following circle $C''$. Follow $C'$ from $e$ to $e_1$, but not traversing it. Follow 
$e_1$ until reaching $C$ and then use $C$ to reach $e_2$ without crossing $e_1$ and $e_2$. 
Then return back to $C'$ along $e_2$ and return to $e$ (see \Cref{fig:newcurve}). $C''$ meets $M$ only in $e$. Either $C''$ 
is an essential circuit contradicting the fact that $B$ has facewidth two or it is planar 
contradicting $2$-connectivity.  Again we can conclude that all multigraphs $B$, where the 
surgery does not result in disconnecting the surface can be constructed from a $2$-connected multigraph embeddable on 
$\mathbb{S}_{g-1}$ with two marked edges resulting in the term
\[\opb{y}{z}^2\vertex{\two_{g-1}}(x,y,z).\]
The factor $\frac{1}{2}\br{y+\frac{z}{y}}^2\br{\frac{1}{y}+\frac{y}{z}}^2$ follows as in Case $1$. We thus conclude
\eqref{eq:2-conn-graphs-fw2}.

To prove \eqref{eq:conn-graphs-fw1},
let $G$ be a multigraph in $\class{\one}_g^{\fw=1}$. Fix a $2$-cell embedding $M$ of $G$ on $\Sg$ of facewidth one and let 
$e_1=\{v_1,v_2\}$ be an edge such that there exists an essential circle $C$ on $\Sg$ meeting $G$ only in 
$e_1$. Then we perform the following surgery. We delete $e_1$, cut the surface 
along $C$, close both holes with a disc, and attach an edge, an additional vertex and a loop to both $v_1$ and $v_2$.
Remark that the edge deleted may be a single edge or part of a double edge. Thus we have a factor of 
$(xyw)^2\br{\frac1y+\frac yz}$. The deleted edge cannot be a loop as in cubic maps on orientable surfaces loops are 
always on the boundary of two 
different faces and as such cannot be the only intersection of an embedding of a multigraph with an essential circle.
This can easily be seen as at the base of the loop there is only one other edge. Thus the boundary of the 
face of the loop without this additional edge consists only of traversing the loop once.

By this surgery we either disconnect the surface or we reduce its genus by one.
If we separate the surface, we obtain two connected multigraphs each with one marked loop. These multigraphs are counted by 
$\opa{w}(\one_{g'})$ and by $\opa{w}(\one_{g-g'})$, as the genera of the two parts sum up to $g$ and
the embeddings resulting from the surgery are still $2$-cell embeddings.

If the surface is not separated, the resulting embedding is a $2$-cell embedding and hence the multigraph remains connected. 
Therefore we obtain a multigraph counted by $\opa{w}^2(\one_{g-1})$. The factor in front of the generating function is once again 
obtained by marking two loops. This proves \eqref{eq:conn-graphs-fw1}.
\end{proof}

\Cref{lem:small-fw} will be used to show that the number of multigraphs with small facewidth is negligible. We will thus be able to use 
\Cref{2conngf,conngf} to obtain asymptotic formulas for connected and $2$-connected cubic multigraphs. To this end we will 
need equations for the generating functions of the auxiliary classes $\class N$ and $\class Q$.

\begin{prop}
The generating function $\vertex{\network}(x,y,z)$ of $\class{\network}$ satisfies the system of equations
\begin{align}
\begin{split}
\vertex{\network}(x,y,z)&=\frac{u(1-2u)-x^2y(1+\network(x,y,z))(y^2-2z)}{2},\label{eq:network}\\
x^2y^3(1+\vertex{\network}(x,y,z))^3&=u(1-u)^3
\end{split}
\end{align}
and the generating function $\vertex{Q}(x,y,z,w)$ of $\class{Q}$ satisfies
\begin{align}
\begin{split}
\vertex{Q}&=\frac{x^2y^3}{2}\vertex{A}+\frac{\vertex{Q}^2}{2}+x^2 y^2 w,\label{systemq}\\
\vertex{A}&=\vertex{Q}+\vertex{S}+\vertex{P}+\vertex{H},\\
\vertex{S}&=\frac{\vertex{A}^2}{\vertex{A}+1},\\
\vertex{P}&=\frac{x^2y^3}{2}\vertex{A}^2+x^2y^3\vertex{A}+x^2yz,\\
2\vertex{H}(1+\vertex{A})&=u(1-2u)-u(1-u)^3,\\
x^2y^3\vertex{A}^3&=u(1-u)^3.
\end{split}
\end{align}
\end{prop}

\begin{proof}
We obtain \eqref{systemq} by following the lines of Section $3$ in \cite{Kang2012} or Section 
$3$ in \cite{nrr}. The only difference is that we account for loops and double edges in the 
initial conditions. 
In order to derive the equations for $Q$, one starts with an edge-rooted connected cubic planar graph and recursively 
decomposes it depending on the placement of the root. Then either
\begin{enumerate}
 \item\label{planar:q} the root is a loop;
 \item\label{planar:d} the root is a bridge;
 \item\label{planar:s} the root is part of a minimal separating edge set of size two;
 \item\label{planar:p} the end vertices of the root separate the graph; or
 \item\label{planar:h} the root is part of a $3$-connected component.
\end{enumerate}
In Case \ref{planar:q} we obtain an equation for $Q$; Case \ref{planar:d} results in an equation that can immediately be 
eliminated from the system; Case \ref{planar:s} results in the equation for $S$; Case \ref{planar:p} in the equation for  
$P$; and Case \ref{planar:h} in the parametric equations for $H$ in terms of $u$. It is shown in \cite{Kang2012} 
that these are indeed exhaustive. For each of these cases there is a decomposition of the graph resulting in the 
corresponding equation in the proposition.

To obtain the equations for $\network(x,y,z)$ we start with \eqref{systemq}. Because $\network(x,y,z)$ enumerates edge-rooted $2$-connected 
planar cubic multigraphs,  setting $w=0$ and $Q(x,y,z,w)=0$ results in a system of equations for $\network(x,y,z)=1+A$. The 
given equations follow by eliminating $S$, $P$, and $H$ from these equations. 
\end{proof}

\subsection{Asymptotics}

The goal of this section is to obtain asymptotics for $\one_g(v)$. The analysis for $\one_g^u(v)$ and $\one_g^s(v)$ are analogous; we will point out the differences when they occur.

We start by determining the growth constant for $\network(v)$. We will shortly see that the analogous results for 
$\network^u(v)$ and $\network^s(v)$ are not necessary for our purpose.

\begin{lem}\label{planarnetwork}
The dominant singularity of $\network(v)$ occurs at $\rho_N=2\left(\frac{6}{17}\right)^3$.
Furthermore, $N(v)$ is $\Delta$-analytic and
\[N(v)=\frac{1}{16}-\frac{51}{400}(1-\rho_N^{-1}v)+\frac{17^{5/2}}{2\cdot 3^{1/2}\cdot 5^5}(1-\rho_N^{-1}v)^{3/2}+O\br{(1-\rho_N^{-1}v)^2}.\]
\end{lem}

\begin{proof}
By eliminating $u$ from \eqref{eq:network} and substituting $v$ as in \Cref{def:univariate}, we obtain the following
implicit equation for $\network(v)$.
\begin{align*}
0=&16 N^6 v^2 + N^5v (32 + 96 v) + N^4\cdot (16 + 56v + 240 v^2)\\
&+N^3\cdot (24 - 25v + 320 v^2)+N^2\cdot (12 - 91v + 240v^2) \\
&+ N\cdot (2 - 43v + 96 v^2) - v (1 - 16 v).
\end{align*}
Using standard methods (see for example \cite[VII.7.1]{Flajolet2009-analytic-combinatorics}) to deal with implicitly given functions we determine the dominant singularity to be at 
$\rho_N=2(6/17)^3$ and derive the claimed expression for $N$.
\end{proof}

A second step is to analyse the dominant singularity of $Q(x,y,z,w)$. The numerical values of the results are 
different for $Q(v)$, $Q^u(v)$, and $Q^s(v)$, but analysis works in exactly the same way.
\begin{lem}\label{planarconn}
The dominant singularity of $Q(v)$ is $\rho_Q=\frac{54}{79^{3/2}}$, 
 $Q(v)$ is $\Delta$-analytic and 
\[Q(v)=q_0-q_1(1-\rho_Q^{-1}v)+q_2(1-\rho_Q^{-1}v)^{3/2}+O\br{(1-\rho_Q^{-1}v)^2},\]
where $q_0=1-\frac{17}{2\sqrt{79}}$, $q_1=\frac{189}{298\sqrt{79}}$, and 
$q_2=\frac{79\cdot 2^{3/2} \cdot 3^{5/2}}{199^{5/2}}$.

\end{lem}

\begin{proof}
The simplest way of dealing with \eqref{systemq} after substituting $v$ as in \Cref{def:univariate} is to eliminate $u$, 
$\vertex{D}$, $\vertex{Q}$, $\vertex{S}$, $\vertex{P}$ and $\vertex{H}$ from these equations to get the following implicit 
equation for $\vertex{A}$.
\begin{align*}
0=&4 - 52 A^2 + 240 A^4 - 448 A^6 + 256 A^8 + 
A^3 (336 + 1848 A^2 - 2400 A^4) v\\& + A^6 (3017 + 1024 A^2) v^2+ 
4096 A^9v^3.
\end{align*}

Using standard methods for implicitly given functions (see for example \cite[VII.7.1]{Flajolet2009-analytic-combinatorics}) 
we get the dominant singularity to be at $\rho_Q=\frac{54}{79^{3/2}}$ and we obtain an expression for $A(v)$ in
terms of powers of $(1-\rho_Q^{-1}v)$. 
By substituting this expression into the system of equations, we get the claimed expression for $\vertex{Q}$. 
\end{proof}

Using \Cref{planarnetwork} we can analyse the function $\two_g(v)$. Since the substitution in \Cref{conngf} satisfies 
$\ol{z}=\frac12\ol{y}^2$, we need the function $B_g(v)$ also in the cases of unweighted multigraphs and simple graphs. Accordingly, the networks used in the substitution in \Cref{2conngf} have to satisfy $z=\frac{1}{2}y^2$.

\begin{lem}\label{thm:2conn}
The dominant singularity of $\vertex{\two_g}^{\fw\ge2}(v)$ 
is $\rho_{\two}=\rho_{\network}=\frac{2^43^3}{17^3}$.
The generating function $\vertex{\two}_0^{\fw\ge2}(v)$ is $\Delta$-analytic and satisfies 
\begin{align*}
\vertex{\two}_0^{\fw\ge2}(v)=& a_0+a_1\left( 1 - \rho_{\two}^{-1} v \right)+a_2\left( 1 - \rho_{\two}^{-1} v \right)^2 \\
	&+c_0\left( 1 - \rho_{\two}^{-1} v\right)^{5/2}+O\br{\br{ 1 - \rho_{\two}^{-1} v }^3},
\end{align*}
where $a_0$, $a_1$, $a_2$ are constants. Furthermore,
\begin{equation*}
\vertex{\two}_1^{\fw\ge2}(v)\cong c_1\log\left( 1 - \rho_{\two}^{-1}v\right)+O\br{\br{ 1 - \rho_{\two}^{-1} v}^{1/4}}
\end{equation*}
and
\begin{equation*}
\vertex{\two}_g^{\fw\ge2}(v)\cong c_g\left( 1 - \rho_{\two}^{-1} v \right)^{-5g/2+5/2}+
	O\br{\br{ 1 - \rho_{\two}^{-1} v}^{-5g/2+11/4}}, \quad g\ge 2,
\end{equation*}
where the $c_g$ are constants.
\end{lem}

\begin{proof}
Suppose first that $g=0$. For this case $N(v)$ is the edge-rooted version of 
$\two_0(v)$. Therefore 
\[\two_0(v) =  \int \frac{1}{y}N(v)\dd y=\frac{1}{6}\int\frac{1}{v}N(v)\dd v.\]
The result follows by \Cref{planarnetwork,integrating}.

Let $g\ge 1$.
By the change of variables as in \Cref{def:univariate} we have 
\[\vertex{\three_g}^{\fw\ge3}(x,y+y\vertex{\network}(x,y,z))=\three_g^{\fw\ge3}\br{v(1+N(v))^3}.\] 
The dominant singularity of $\three_g^{\fw\ge3}\br{v(1+N(v))^3}$ is given either by the singularity $\rho_{\network}$ 
of $\vertex{\network}(v)$ or by a solution of $v(1+\vertex{\network}(v))^3=\rho_{\triangs}^3$, where 
$\rho_{\triangs}^3$ is the dominant singularity of $\three_g(v)$. It can be verified that 
$\rho_{\network}(1+\vertex{\network}(\rho_{\network}))^3=\rho_{\triangs}^3=\rho_\three$ by \Cref{simplerefined} 
and \Cref{planarnetwork}. This is the only 
solution of this equation, as $v(1+\vertex{\network}(v))$ is monotone on the interval $[0,\rho_\network)$. So 
$\rho_{\network}$ is the dominant singularity and the 
composition is critical (in the sense of \cite[pp. 411ff]{Flajolet2009-analytic-combinatorics}). 
By \Cref{coro:3-conn-graphs} and \Cref{planarnetwork}, we have
\begin{align*}
\vertex{\three_1}^{\fw\ge3}(v(1+\network(v))^3)\cong c_1\log(1-\rho_{\network}^{-1} v)+O\br{(1-\rho_{\network}^{-1} v)^{1/4}}
\end{align*}
and, for $g\ge 2$,
\begin{align*}
\vertex{\three_g}^{\fw\ge3}(v(1+\network(v))^3)\cong c_g(1-\rho_{\network}^{-1} v)^{-5g/2+5/2}+O\br{(1-\rho_{\network}^{-1} v)^{-5g/2+11/4}}.
\end{align*}
Similarly, we have
\begin{equation*}
  \vertex{\three_0}(v(1+\network(v))^3) \cong c_0\br{1-\rho_{\network}^{-1} v}^{5/2} + O\br{\br{1-\rho_{\network}^{-1} v}^3}.
\end{equation*}
Now \Cref{2conngf} yields that
\begin{align*}
\two_1^{\fw\ge3}(v)\cong c_1\log(1-\rho_{\network}^{-1} v)+O\br{(1-\rho_{\network}^{-1} v)^{1/4}}
\end{align*}
and
\begin{align*}
\two_g^{\fw\ge3}(v)\cong c_g(1-\rho_{\network}^{-1} v)^{-5g/2+5/2}+O\br{(1-\rho_{\network}^{-1} v)^{-5g/2+11/4}}, \quad g\ge 2.
\end{align*}

It remains to show that
\begin{equation}\label{eq:fw2small}
  \two_g^{\fw=2}(v) \cong O\br{\br{1-\rho_{\network}^{-1} v}^{-5g/2+11/4}}.
\end{equation}
By \Cref{lem:small-fw} and the fact that 
\[\opa{y}\br{F\br{x,y,\frac{y^2}{2}}}\succeq\left.\opa{z}(F(x,y,z))\right|_{z=\frac{y^2}{2}}\]
for every generating function $F(x,y,z)$, we have the relation 
\begin{align}
\begin{split}\label{eq:fw2foruse}
	\vertex{\two_g}^{\fw=2}\br{x,y,\frac{y^2}{2}}\frac12&\br{y+\frac{\frac{y^2}{2}}{y}}^2\br{\frac1y+\frac {y}{\frac{y^2}{2}}}^2\preceq 
		\opa{y}^2\br{\vertex{\two_{g-1}^{\fw\ge2}}\br{x,y,\frac{y^2}{2}}}\\&+\sum_{g'=1}^{g-1}\opa{y}
	\br{\vertex{\two_{g'}^{\fw\ge2}}\br{x,y,\frac{y^2}{2}}}\opa{y}\br{\vertex{\two_{g-g'}^{\fw\ge2}}\br{x,y,\frac{y^2}{2}}}.
\end{split}
\end{align}
%
By \Cref{substitution}, \eqref{eq:fw2foruse} implies that
\begin{align}\label{eq:twoupper}
 \frac{81}{8}\two_g^{\fw=2}(v)\preceq 9\opa{v}^2(\two_{g-1}^{\fw\ge2}(v))+9\sum_{g'=1}^{g-1} \opa{v}(\two_{g'}^{\fw\ge2}(v))\opa{v}(\two_{g-g'}^{\fw\ge2}(v)).
\end{align}
By \Cref{integrating} and the fact that all generating functions on the right-hand side of \eqref{eq:twoupper} are for genus
smaller than $g$, we deduce by induction that 
\begin{align*}
 \opa{v}^2(\two_{g-1}^{\fw\ge2}(v))&\cong O\br{\br{ 1 - \rho_{\two}^{-1} v}^{-5g/2+3}}, &&\\
 \opa{v}(\two_{g'}^{\fw\ge2}(v))&\cong O\br{\br{ 1 - \rho_{\two}^{-1} v}^{-5g'/2+3/2}}, &\text{and}&\\
 \opa{v}(\two_{g-g'}^{\fw\ge2}(v))&\cong O\br{\br{ 1 - \rho_{\two}^{-1} v}^{-5(g-g')/2+3/2}}. &&
\end{align*}
Substituting this into \eqref{eq:twoupper} results in 
\begin{equation*}
\two_g^{\fw=2}(v)\preceq O\br{\br{ 1 - \rho_{\two}^{-1} v}^{-5g/2+3}},
\end{equation*}
which immediately implies \eqref{eq:fw2small} and thus concludes the proof.
\end{proof}

As an immediate corollary of \Cref{thm:2conn} we determine the coefficients of $\vertex{\two}_g^{\fw\ge2}(v)$.
\begin{coro}\label{coro:gf-2conn}
The number of $2$-connected cubic vertex-labelled weighted multigraphs that are strongly embeddable on $\Sg$ with facewidth 
at least two is
\begin{align*}
[v^n]\vertex{\two_g}^{\fw\ge2}(v)&=\br{1+O\br{n^{-1/4}}} c_gn^{5g/2-7/2}\rho_{\two}^{-n}.
\end{align*}
Here $c_g$ is a constant depending only on $g$ and $\rho_{\two}=\rho_{\network}=\frac{2^43^3}{17^3}$.
\end{coro}

We can use \Cref{thm:2conn} to determine the dominant term of connected cubic multigraphs embeddable on 
$\Sg$.

\begin{thm}\label{thm:connected}
The dominant singularity of the generating function $\vertex{\one_g}(v)$ of connected cubic 
vertex-labelled weighted multigraphs that are strongly embeddable on $\Sg$ 
is $\rho_{\one}=\rho_{Q}=\frac{54}{79^{3/2}}$.
The generating function $\one_0(v)$ is $\Delta$-analytic and
\begin{align*}
\vertex{\one}_0(v)= & a_0+a_1\left( 1 - \rho_{\one}^{-1} v \right)+a_2\left( 1 - \rho_{\one}^{-1}v \right)^2 \\
	&+c_0\left( 1 - \rho_{\one}^{-1} v\right)^{5/2}+O\br{\left( 1 - \rho_{\one}^{-1} v \right)^3}.
\end{align*}
where $a_0$, $a_1$, $a_2$ are constants. Furthermore,
\begin{equation*}
\vertex{\one}_1(v)\cong c_1\log\left( 1 - \rho_{\one}^{-1} v\right)+O\br{\left( 1 - \rho_{\one}^{-1}v \right)^{1/4}}
\end{equation*}
and
\begin{equation*}
\vertex{\one}_g(v)\cong c_g\left( 1 - \rho_{\one}^{-1} v \right)^{-5g/2+5/2}+
	O\br{\left( 1 - \rho_{\one}^{-1} v \right)^{-5g/2+11/4}}, \quad g\ge 2,
\end{equation*}
where the $c_g$ are constants.
\end{thm}

\begin{proof}

For $g=0$, the class of planar connected cubic multigraphs is given by unrooting the sum of the classes 
used in \Cref{planarconn}, see \cite{bodirsky2007-cubic-graphs} or \cite{Kang2012} for more details.

Suppose $g\ge 1$.
By the substitution as in \Cref{def:univariate} we have 
\[\vertex{\two_g}^{\fw\ge2}\br{x,\frac{y}{1-Q},\frac12\br{\frac{y}{1-Q}}^2}=\vertex{\two_g}^{\fw\ge2}\br{\frac{v}{(1-Q(v))^3}}.\]
The dominant singularity with respect to $v$ is given either by the singularity 
$\rho_Q$ of $\vertex{Q}(v)$ as calculated in \Cref{planarconn} or by a solution of 
$v/(1-Q(v))^3=\rho_{\two}$, where $\rho_\two$ is the dominant singularity of $\two_g$ calculated in 
\Cref{thm:2conn}. It can easily be verified that $\rho_Q/(1-\vertex{Q}(\rho_Q))^3=\rho_{\two}$
by \Cref{planarconn} and \Cref{thm:2conn}. This is the only solution of this equation as $v/(1-\vertex{Q}(v))^3$ 
is monotone (as it is a geometric series of a generating function with nonnegative coefficients). Therefore 
$\rho_{\one}=\rho_Q$ is the dominant singularity and the composition is critical (in the sense of 
\cite[pp. 411ff]{Flajolet2009-analytic-combinatorics}). By \Cref{planarconn,thm:2conn}, we have
\begin{align*}
\vertex{\two_g}^{\fw\ge2}\br{\frac{v}{(1-Q)^3}} \cong c_g\br{1-\rho_{\one}^{-1}v}^{-5g/2+5/2}+O\br{\br{1-\rho_{\one}^{-1}v}^{-5g/2+11/4}}
\end{align*}
for $g\ge 2$ and
\begin{align*}
\vertex{\two_1}^{\fw\ge2}\br{\frac{v}{(1-Q)^3}} \cong c_1\log\br{1-\rho_{\one}^{-1}v}+O\br{\br{1-\rho_{\one}^{-1}v}^{1/4}}.
\end{align*}
Similarly, we have
\begin{equation*}
  \two_0\br{\frac{v}{(1-Q)^3}} \cong c_0\br{1-\rho_{\one}^{-1}v}^{5/2} + O\br{\br{1-\rho_{\one}^{-1}v}^3}.
\end{equation*}
Now \Cref{conngf} yields that
\begin{align*}
\one_1^{\fw\ge2}(v)\cong c_1\log(1-\rho_{\one}^{-1}v)+O\br{(1-\rho_{\one}^{-1}v)^{1/4}}
\end{align*}
and
\begin{align*}
\one_g^{\fw\ge2}(v)\cong c_g(1-\rho_{\one}^{-1}v)^{-5g/2+5/2}+O\br{(1-\rho_{\one}^{-1}v)^{-5g/2+11/4}}, \quad g\ge 2.
\end{align*}

It thus remains to show that
\begin{equation}\label{eq:fw1small}
  \vertex{\one_g}^{\fw=1}\cong O\br{\br{ 1 - \rho_{\one}^{-1} v}^{-5g/2+11/4}}.
\end{equation}
By \Cref{lem:small-fw} 
and the fact that $\opa{y}(F(x,y,\frac{y^2}{2},\frac{y}{2}))\succeq\left.\opa{w}(F(x,y,z,w))\right|_{z=\frac{y^2}{2},w=\frac{y}
{2}}$ for every generating function $F(x,y,z,w)$, we have the relation 
\begin{align}
\begin{split}\label{eq:fw1foruse}
\vertex{\one_g}^{\fw=1}\br{x,y,\frac{y^2}{2},\frac{y}{2}}&\br{xy\frac{y}{2}}^2\br{\frac1y+\frac{y}{\frac{y^2}{2}}}\preceq \opa{y}^2\br{\one_{g-1}\br{x,y,\frac{y^2}{2},\frac{y}{2}}}\\
&+\sum_{g'=1}^{g-1}\opa{y}\br{\one_{g'}\br{x,y,\frac{y^2}{2},\frac{y}{2}}}\opa{y}\br{\one_{g-g'}\br{x,y,\frac{y^2}{2},\frac{y}{2}}}.
\end{split}
\end{align}
By \Cref{substitution}, \eqref{eq:fw1foruse} implies that
\begin{align}\label{eq:oneupper}
 \frac{3}{4}\one_g^{\fw=1}(v)\preceq 9\opa{v}^2(\one_{g-1}(v))+9\sum_{g'=1}^{g-1} \opa{v}(\one_{g'}(v))\opa{v}(\one_{g-g'}(v)).
\end{align}
By the fact that all generating functions on the right-hand side of \eqref{eq:oneupper} are for 
genus smaller than $g$, we can use induction on $g$ like in the proof of \Cref{thm:2conn} to deduce \eqref{eq:fw1small}.
This concludes the proof.
\end{proof}

From \Cref{thm:connected} and \Cref{congruent:tt} we can immediately evaluate the coefficients of $\vertex{\one}_g(v)$.
\begin{coro}\label{coro:gf-conn}
The asymptotic number of connected cubic vertex-labelled multigraphs that are  weighted by their compensation factor 
and are strongly embeddable on $\Sg$ is given by
\begin{align*}
[v^n]\vertex{\one_g}(v)&=\br{1+O\br{n^{-1/4}}} c_gn^{5g/2-7/2}\rho_{\one}^{-n}.
\end{align*}
Here $c_g$ is a constant only depending on $g$ and $\rho_{\one}=\rho_Q=\frac{54}{79\sqrt{79}}$.

\end{coro}

\subsection{Proof of \Cref{main2}}
Using the results from the previous section we can now prove \Cref{main2}. Recall that a cubic multigraph embeddable on 
$\Sg$ is given by a set of connected cubic multigraphs embeddable on $\torus{g_i}$ such that $\sum g_i\leq g$ (see 
\Cref{strongemb}). Therefore we have the relation

\begin{align}
\vertex{\general}_g(v)&\preceq\sum_{k=1}^\infty\sum_{\sum g_i\le g}\frac{1}{k!}\prod_{i=1}^k\br{\vertex{\one}_{g_i}(v)+\frac{v}{6}}.\label{eq:all}
\end{align}
The summand $\frac{v}{6}$ accounts for the fact that a component might also be a triple edge, not accounted for in 
$\vertex{\one}_{g_i}(v)$ (this additional summand will differ when proving \Cref{main1,main3}).
We only get an upper bound, because we overcount if a multigraph is strongly embeddable on surfaces of multiple genera. Later
we will also obtain a lower bound with the same asymptotics to complete the proof.

If $g=0$, \eqref{eq:all} simplifies to $\general_0=\exp(\one_0+\frac{v}{6})$, as there is no overcounting in this case. This coincides with 
Theorem 1 of \cite{Kang2012} and therefore we can conclude our statement. 
(Although it is not directly shown there, the same arguments can be used for unweighted planar cubic multigraphs. 
For simple graphs, see \cite{bodirsky2007-cubic-graphs}.)

Now suppose $g\ge 1$. The first step to obtain asymptotics from \eqref{eq:all} is to rearrange the sum in such a way that all planar components are 
singled out. This results in 

\begin{align}\label{eq:general:upperbound}
\vertex{\general}_g(v)&\preceq\sum_{k=0}^g\sum_{\substack{\sum g_i\le g\\g_i\ge1}}\frac{1}{k!}\prod_{i=1}^k\br{\vertex{\one}_{g_i}(v)+\frac{v}{6}}\sum_{j=0}^\infty 
\frac{k!}{(k+j)!}\br{\vertex{\one}_0(v)+\frac{v}{6}}^j.
\end{align}

By the dominant term and the value of $\vertex{\one}_0(v)$ at the singularity $\rho_\one$ from 
\Cref{thm:connected} we observe that the contribution of 
the last sum to the formula is only a constant factor. Thus it remains to calculate the dominant term  
of $\frac{1}{k!}\prod \br{\vertex{\one}_{g_i}(v)+\frac{v}{6}}$. As the first sum consists only of a constant number 
of summands, the dominant term of the right-hand side of \eqref{eq:general:upperbound} will be the (sum of the) 
dominant terms from $\frac{1}{k!}\prod \br{\vertex{\one}_{g_i}(v)+\frac{v}{6}}$ up to the constant obtained from the planar components.
That is, we shall compute the dominant term of 
\begin{align*}
A(v) := \frac{1}{k!}\prod_{i=1}^k\br{\one_{g_i}(v)+\frac{v}{6}},
\end{align*}
where the $g_i$ are positive and sum up to $g'\le g$. 

Suppose $g=1$. Then either $k=g'=0$ or $k=g'=1$. By \Cref{thm:connected}, we have 
\[A(v) = C_1(v)+\frac{v}{6}\cong c_1\log(1-\rho_{\one}^{-1}v)+\frac{v}{6}+O\br{\br{1-\rho_{\one}^{-1}v}^{1/4}}\]
and thus 
\[A(v) \preceq P_1(v)+c_1\log(1-\rho_{\one}^{-1}v)+O\br{\br{1-\rho_{\one}^{-1}v}^{1/4}}\]
with $P_1(v)$ a polynomial and $c_1$ a constant.
Suppose now $g\ge 2$.
Without loss of generality let $g_1,\ldots,g_l=1$ and $g_{l+1},
\ldots,g_k>1$. Then
\begin{align}
\begin{split}
A(v)\cong&\br{1+O\br{\br{1-\rho_{\one}^{-1}v}^{1/4}}}\frac{1}{k!}c_1^l
		\left(\log\left(1-\rho_{\one}^{-1} v\right)\right)^l\label{eq:asys}
	\prod_{i=l+1}^kc_{g_i}\left( 1 - \rho_{\one}^{-1}v \right)^{5(1-g_i)/2}\\
\cong& \br{c+O\br{\br{ 1 - \rho_{\one}^{-1} v}^{1/4}}}\br{\log\left( 1 - \rho_{\one}^{-1} v\right)}^l
\left( 1 - \rho_{\one}^{-1}v \right)^{-5g'/2+5k/2}.
\end{split}
\end{align}
For $k=1$ and $g'=g$ (and hence $l=0$) we thus have
\begin{align}
 \frac{1}{1!}\prod_{g'=1}^1\br{\one_{g_i}+\frac{v}{6}}\cong c_g(1-\rho_{\one}^{-1}v)^{-5g/2+5/2}+O\br{(1-\rho_{\one}^{-1}v)^{-5g/2+11/4}}.\label{eq:a}
\end{align}
For $k\ge2$ or $g'<g$, \eqref{eq:asys} yields 
\begin{align}
 \frac{1}{k!}\prod_{i=1}^k\one_{g_i}(v)\cong O\br{(1-\rho_{\one}^{-1}v)^{-5g/2+5/2+2}}\label{eq:b}
\end{align}
and thus 
\[G_g(v)\preceq c_g(1-\rho_{\one}^{-1}v)^{-5g/2+5/2}+O\br{(1-\rho_{\one}^{-1}v)^{-5g/2+11/4}}.\]

We derive a lower bound for $g\ge 1$ as follows. Let $\tilde{\class{G}}_g$ be the class of graphs in $\class{G}_g$ with one component of genus $g$ and all other components planar. Then
\begin{align*}
 \sum_{j=0}^{\infty}\frac{C_g(v)C_0^j(v)}{(j+1)!}\succeq \tilde{G}_g(v)\succeq \sum_{j=0}^{\infty}\frac{C_g(v)C_0^j(v)}{(j+1)!}-\sum_{j=0}^{\infty}\frac{(j+1)C_0^{j+1}(v)}{(j+1)!}.
\end{align*}
Indeed, if the component of genus $g$ is also planar, then the graph might be counted up to $j+1$ times (once for each 
component) on the left-hand side. Substituting the corresponding summands thus yields a lower bound of $\general_g(v)$. 
\begin{align}
	\vertex{\general}_g(v)\succeq \tilde{\general}_g(v)\succeq \sum_{j=0}^\infty\frac{1}{(j+1)!}\one_g(v)\one_0^j(v)-\sum_{j=1}^\infty\frac{j}{j!}\one_0^j(v).\label{lowerbd}
\end{align}
Applying \Cref{tt} to the upper and lower bounds completes the proof.

\subsection{Proofs of \Cref{main1,main3,main4}}
\Cref{main1,main3} can be proven analogously to \Cref{main2}. We obtain $\rho_1$ as the smallest positive solution of
\[0=-46656 + 139968 u + 9524176 u^2 - 1763856 u^3 + 121716 u^4 + 8748 u^5 + 729 u^6\]
and $\rho_3$ as the smallest positive solution of 
\[0=-46656 - 139968 u + 6043120 u^2 - 1717200 u^3 - 69228 u^4 - 8748 u^5 + 729 u^6.\]

\Cref{main4} for the case of weighted multigraphs follows immediately from~\eqref{eq:a}, \eqref{eq:b}, and \Cref{tt}.
Indeed, \eqref{eq:a} and \Cref{tt} imply that the number of weighted multigraphs in $\class{\general}_g(n)$ that
have a unique non-planar component that is not embeddable on $\torus{g-1}$ is
\[\br{1+O\br{n^{-1/4}}}e_g n^{5g/2-5/2}\gamma_2^{2n}(2n)!\,.\]
On the other hand, \eqref{eq:b} and \Cref{tt} imply that the number of weighted multigraphs in $\class{\general}_g(n)$
that do not have such a component is
\[O\br{n^{5g/2-5/2-2}\gamma_2^{2n}(2n)!}.\]
Thus, \Cref{main4} follows. Observe that the probability $1-O(n^{-2})$ is not sharp. Indeed, the exponent in
\eqref{eq:b} could be improved to $-5g/2+5-\varepsilon$ for every $\varepsilon>0$, which would yield a probability
$1-O(n^{-5/2+\varepsilon})$. The statements of \Cref{main4} for unweighted multigraphs and simple graphs are proved
analogously.

\bibliographystyle{plain}
\bibliography{biblio}

\begin{thebibliography}{10}

\bibitem{BENDER1986244}
E.~A. Bender and E.~R. Canfield.
\newblock The asymptotic number of rooted maps on a surface.
\newblock {\em Journal of Combinatorial Theory, Series A}, 43(2):244--257,
  1986.

\bibitem{Bender1993-maps}
E.~A. Bender, E.~R. Canfield, and L.~B. Richmond.
\newblock {The asymptotic number of rooted maps on a surface.: II. Enumeration
  by vertices and faces}.
\newblock {\em J. Combin. Theory Ser. A}, 63(2):318--329, 1993.

\bibitem{representativity}
E.~A. Bender, Z.~Gao, and L.~B. Richmond.
\newblock Almost all rooted maps have large representativity.
\newblock {\em Journal of Graph Theory}, 18(6):545--555, 1994.

\bibitem{bender}
E.~A. Bender, Z.~Gao, and N.~C. Wormald.
\newblock The number of labeled 2-connected planar graphs.
\newblock {\em Electron. J. Combin.}, 9(1):Research Paper 43, 13 pp.
  (electronic), 2002.

\bibitem{Bender1988370}
E.~A. Bender and N.~C. Wormald.
\newblock The asymptotic number of rooted nonseparable maps on a surface.
\newblock {\em Journal of Combinatorial Theory, Series A}, 49(2):370--380,
  1988.

\bibitem{MR2387555}
M.~Bodirsky, C.~Gr{\"o}pl, and M.~Kang.
\newblock Generating unlabeled connected cubic planar graphs uniformly at
  random.
\newblock {\em Random Structures Algorithms}, 32(2):157--180, 2008.

\bibitem{bodirsky2007-cubic-graphs}
M.~Bodirsky, M.~Kang, M.~L{\"{o}}ffler, and C.~McDiarmid.
\newblock Random cubic planar graphs.
\newblock {\em Random Struct. Alg.}, 30:78--94, 2007.

\bibitem{triang-disk}
W.~G. Brown.
\newblock Enumeration of triangulations of the disk.
\newblock {\em Proc. London Math. Soc. (3)}, 14:746--768, 1964.

\bibitem{Chapuy2011-enumeration-graphs-on-surfaces}
G.~Chapuy, {\'{E}}.~Fusy, O.~Gim{\'{e}}nez, B.~Mohar, and M.~Noy.
\newblock Asymptotic enumeration and limit laws for graphs of fixed genus.
\newblock {\em J. Combin. Theory Ser. {A}}, 118(3):748--777, 2011.

\bibitem{flajolet1990-transfer}
P.~Flajolet and A.~M. Odlyzko.
\newblock Singularity analysis of generating functions.
\newblock {\em SIAM J. Discrete Math.}, 3(2):216--240, 1990.

\bibitem{Flajolet2009-analytic-combinatorics}
P.~Flajolet and R.~Sedgewick.
\newblock {\em Analytic Combinatorics}.
\newblock Cambridge University Press, New York, NY, USA, 2009.

\bibitem{Gao1991-2-connected-projective}
Z.~Gao.
\newblock The number of rooted {$2$}-connected triangular maps on the
  projective plane.
\newblock {\em J. Combin. Theory Ser. B}, 53(1):130--142, 1991.

\bibitem{Gao1992-2-connected-surface}
Z.~Gao.
\newblock The asymptotic number of rooted {$2$}-connected triangular maps on a
  surface.
\newblock {\em J. Combin. Theory Ser. B}, 54(1):102--112, 1992.

\bibitem{Gao1993-pattern}
Z.~Gao.
\newblock A pattern for the asymptotic number of rooted maps on surfaces.
\newblock {\em J. Combin. Theory Ser. A}, 64(2):246--264, 1993.

\bibitem{MR1980342}
Z.~Gao and N.~C. Wormald.
\newblock Enumeration of rooted cubic planar maps.
\newblock {\em Ann. Comb.}, 6:313--325, 2002.

\bibitem{gimenez2009}
O.~Gim{\'e}nez and M.~Noy.
\newblock Asymptotic enumeration and limit laws of planar graphs.
\newblock {\em J. Amer. Math. Soc.}, 22(2):309--329, 2009.

\bibitem{quadratic-method}
Ian~P. Goulden and David~M. Jackson.
\newblock {\em Combinatorial enumeration}.
\newblock Wiley-Interscience series in discrete mathematics. J. Wiley \& Sons,
  New York, Brisbane, Toronto, 1983.
\newblock A Wiley-Interscience publication.

\bibitem{jansonea}
S.~Janson, D.~E. Knuth, T.~{\L}uczak, and B.~Pittel.
\newblock The birth of the giant component.
\newblock {\em Random Structures Algorithms}, 4(3):231--358, 1993.
\newblock With an introduction by the editors.

\bibitem{rgraphs}
S.~Janson, T.~{\L}uczak, and A.~Rucinski.
\newblock {\em Random graphs}.
\newblock Wiley-Interscience Series in Discrete Mathematics and Optimization.
  Wiley-Interscience, New York, 2000.

\bibitem{Kang2012}
M.~Kang and T.~{\L}uczak.
\newblock Two critical periods in the evolution of random graphs.
\newblock {\em Trans. Am. Math. Soc.}, 364(8):4239--4265, August 2012.

\bibitem{prep}
M.~Kang, M.Mo{\ss}hammer, and P.~Spr{\" u}ssel.
\newblock Sparse graphs on surfaces.
\newblock in preparation.

\bibitem{ks}
M.~Kang and P.~Sprüssel.
\newblock Characterisation of symmetries of unlabelled triangulations and its
  applications.
\newblock {\em ArXiv e-prints}, 2015.

\bibitem{seminaronrg}
M.~Karo{\'n}ski and Z.~Palka, editors.
\newblock {\em Random graphs '85}, volume 144 of {\em North-Holland Mathematics
  Studies}.
\newblock North-Holland Publishing Co., Amsterdam, 1987.
\newblock Papers from the second international seminar on random graphs and
  probabilistic methods in combinatorics held at Adam Mickiewicz University,
  Pozna{\'n}, August 5--9, 1985, Annals of Discrete Mathematics, 33.

\bibitem{mcdiarmid2005}
C.~McDiarmid, A.~Steger, and D.~J.~A. Welsh.
\newblock Random planar graphs.
\newblock {\em J. Comb. Theory Ser. B}, 93(2):187--205, 2005.

\bibitem{nrr}
M.~{Noy}, V.~{Ravelomanana}, and J.~{Ru{\'e}}.
\newblock On the probability of planarity of a random graph near the critical
  point.
\newblock {\em Proc. Amer. Math. Soc.}, 143(3):925--936, 2015.

\bibitem{osthus}
D.~Osthus, H.~J. Prömel, and A.~Taraz.
\newblock On random planar graphs, the number of planar graphs and their
  triangulations.
\newblock {\em Journal of Combinatorial Theory, Series B}, 88(1):119--134,
  2003.

\bibitem{robertson1990-representativity}
N.~Robertson and R.~Vitray.
\newblock Representativity of surface embeddings.
\newblock In B.~Korte, L.~Lovász, H.J. Promel, and A.~Schrijver, editors, {\em
  Algorithms and Combinatorics}, pages 293--328. Springer, 1990.

\bibitem{titchmarsh1939theory}
E.C. Titchmarsh.
\newblock {\em The Theory of Functions}.
\newblock Oxford science publications. Oxford University Press, 1939.

\bibitem{tutte1962-planar-triangulations}
W.~T. Tutte.
\newblock A census of planar triangulations.
\newblock {\em Canad. J. Math.}, 14:21--38, 1962.

\bibitem{tutte1963-planar-maps}
W.~T. Tutte.
\newblock A census of planar maps.
\newblock {\em Canad. J. Math.}, 15:249--271, 1963.

\bibitem{whitney}
H.~Whitney.
\newblock Congruent graphs and the connectivity of graphs.
\newblock {\em Amer. J. Math.}, 54:150--168, 1932.

\end{thebibliography}

\appendix
\section{Proof of \Cref{simplerefined,looplessrefined}} \label{proof:gao}

We begin with notations.
The \emph{valency} of a face $f$ in a map is the number of corners of $f$.
We call a rooted map $M$ a \emph{quasi triangulation} if all faces except the
root face $f_r$ are bounded by triangles. Let $\class{\quasis}_g$ be the class
of \emph{simple} quasi triangulations and $\quasis_g(y,u)$ its generating function, where
$u$ marks the valency of $f_r$. Given an index set $I$ and an injective function
$h\colon I \to F(M)\setminus\{f_r\}$, we call $M$ an \emph{$I$-quasi triangulation
with respect to $h$} if all faces in $F(M) \setminus \br{h(I)\cup\{f_r\}}$ are
bounded by triangles. If in addition $f_r$ is also bounded by a triangle, we say
that $M$ is an \emph{$I$-triangulation} (with respect to $h$). Let $\class{\triangs}_{g,I}$
and $\class{\quasis}_{g,I}$ be the classes of simple $I$-triangulations and simple $I$-quasi
triangulations, respectively; their generating functions are denoted by $\triangs(y,z_I)$
and $\quasis(y,u,z_I)$, respectively. Here $u$ again marks the valency of the root face
$f_r$ and $z_I=(z_i)_{i\in I}$ is a vector where $z_i$ marks the valency of $h(i)$.

Note that $\class{\triangs}_g=\class{\triangs}_{g,\emptyset}$ and $\class{\quasis}_g=\class{\quasis}_{g,\emptyset}$. 
In the case $I=\emptyset$ we will therefore always use the generating functions $\triangs_g(y)$ and $\quasis_g(y,u)$ without 
mentioning variables $z_i$. To simplify notation the one-vertex map is contained in $\class{\quasis}_0$, although
it is not a quasi triangulation (since it does not have any corners and thus cannot be rooted).
We say that a face $f$ is \emph{marked} if $f\in h(I)$ and that we are \emph{marking a face} $f$ if we add a new index 
$i$ to the set $I$ with $h(i)=f$.

To prove \Cref{simplerefined} we first derive a recursive formula relating $\class{\quasis}_{g,I}$ for different genera 
and different sizes of the set $I$. We will then prove \Cref{simplerefined} by applying this formula inductively.
In order to derive the recursive formula, we will delete the root edge of a given quasi triangulation and then
perform surgeries that either separate the given surface or decrease its genus. One part of the reverse operation then
consists of adding a new edge to a map. Let $S$ be a map and let (not necessarily distinct) corners $c_1=(v_1,e_1^-,e_1^+)$
and $c_2=(v_2,e_2^-,e_2^+)$ of the same face $f$ of $S$ be given. If $T$ is a map with $V(T)=V(S)$ and $E(T) =
E(S)\cup\{e_{\text{new}}\}$, we say that $e_{\text{new}}$ is an \emph{edge from $c_1$ to $c_2$} if
\begin{itemize}
\item $e_{\text{new}}$ is contained in $f$ and its end vertices are $v_1$ and $v_2$;
\item in the cyclic order of edges of $T$ at $v_1$, $e_{\text{new}}$ is the predecessor of $e_1^+$; and
\item in the cyclic order of edges of $T$ at $v_2$, $e_{\text{new}}$ is the successor of $e_2^-$.
\end{itemize}
If $c_1=c_2 =: c$, we also say that $e_{\text{new}}$ is a \emph{loop at $c$}.

Before we derive the recursive formula, we study the base case of \emph{planar} quasi triangulations.

\begin{prop}\label{prop:baseplanar}
 The generating function of planar quasi triangulations satisfies
 \begin{align}
  \quasis_0(y,u)&=1+yu^2\quasis_0^2(y,u)+\frac{y(\quasis_0(y,u)-1)}{u}-y^2u\quasis_0(y,u)-
	\triangs_0(y)(\quasis_0(y,u)-1).\label{baseplanar}
 \end{align}
\end{prop}
\begin{proof}
  The first summand in \eqref{baseplanar} corresponds to the one-vertex map.
 Let $S\in\class{\quasis}_0$ be a planar quasi triangulation with at least one edge. 
 We distinguish two cases.
 
  First suppose that the root edge $e_r$ is a bridge; then the only face incident with $e_r$ is
  the root face $f_r$. The union $f_r \cup e_r$ is not a disc and thus contains a
  non-contractible circle $C$. We delete $e_r$, cut along $C$, and close the two resulting
  holes by inserting discs. By this surgery, $S$ is separated into two quasi triangulations $S_1$, $S_2$.
  Let $v_1$ and $v_2$ be the end vertices of $e_r$ in $S_1$ and $S_2$, respectively. One of these two
  vertices is the root vertex of $S$; by renaming $S_1,S_2$ we may assume that $v_1$ is the root vertex of $S$. In the cyclic order
  of the edges of $S$ at $v_1$, let $e_1^-$ and $e_1^+$ be the predecessor and successor of $e_r$,
  respectively. Define $e_2^-$ and $e_2^+$ analogously at $v_2$. We let $(v_1,e_1^-,e_1^+)$ and $(v_2,e_2^-,e_2^+)$
  be the root of $S_1$ and $S_2$, respectively. We thus have $S_1,S_2\in\class\quasis_0$. Furthermore,
  $S_1$ and $S_2$ together have one edge less than $S$ and the sum of the valencies of their root 
 faces is two less than the valency of $f_r$. Thus we obtain $yu^2\quasis_0^2(x,u)$, the second term of the right-hand side of \eqref{baseplanar}.
 
 Now suppose that $e_r$ is not a bridge. Then it lies on the boundary of the root face and of another face;
  this face is bounded by a triangle. In the cyclic order of edges at the root vertex $v_r$, let $e_r^-$ and $e_r^+$
  be the predecessor and successor of $e_r$, respectively. We delete $e_r$ and obtain a quasi
  triangulation $S'$ that we root at $c'_r:=(v_r,e_r^-,e_r^+)$. The valency of the root face of $S'$ is larger by one
  than the valency of $f_r$. This is reflected by $\frac{y}{u}(\quasis_0(y,u)-1)$, the third term of the right-hand side of \eqref{baseplanar}, because $S'$ cannot be 
 the quasi triangulation consisting only of a single vertex. However, with this summand we have overcounted. 
 Indeed, the reverse construction is as follows. Let $f'_r$ be the root face of $S'$. Then the corners of
  $f'_r$ can be ordered by walking along the boundary of $f'_r$ in counterclockwise direction. In this order, starting from
  $c'_r$, let $c' = (v,e,e')$ be the corner after the next; then $S$ is obtained from $S'$
  by inserting an edge from $c_r'$ to $c'$. If $v_r=v$ this results in a loop; if $v_r$ and $v$ are adjacent, we obtain a double edge (see
  \Cref{fig:planarquasi}).

  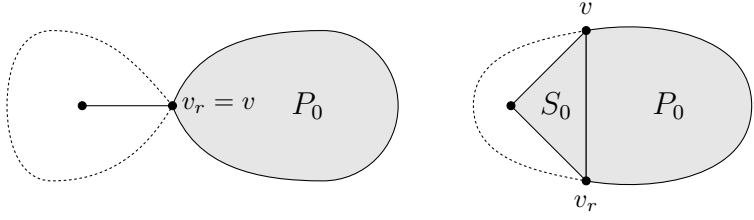
\begin{figure}[htbp]
   \begin{tikzpicture}

\draw [dash pattern=on 1pt off 1pt](3,2) to[out=130,in=0] (1.4,3)to[out=180,in=90] (0.8,2) to[out=270,in=180] (1.4,1) to[out=0,in=240] (3,2);
\draw (1.8,2.)-- (3,2.);
\draw [fill=gray](3,2) to[out=70,in=180] (5,3)to[out=0,in=90] (6,2) to[out=270,in=0] (5,1) to[out=180,in=290] (3,2);
\draw [color=black](4.8,2) node {\Large $P_0$};
\draw (3,2) node[anchor=west] {$v_r=v$};

\draw [fill=gray](8.5,3)-- (8.5,1)-- (7.5,2.)-- (8.5,3);
\draw [color=black](8.1,2) node {\Large $S_0$};
\draw [dash pattern=on 1pt off 1pt](8.5,3) to[out=190,in=90] (7,2) to[out=270,in=170](8.5,1); 
\draw [fill=gray](8.5,3) to[out=10,in=90] (10.7,2) to[out=270,in=350](8.5,1) -- (8.5,3); 
\draw [color=black](9.6,2) node {\Large $P_0$};
\draw [color=black](8.5,3.3) node {$v$};
\draw [color=black](8.5,0.7) node {$v_r$};

\begin{scriptsize}
\draw [fill=black] (1.8,2.) circle (1.5pt);
\draw [fill=black] (3,2.) circle (1.5pt);
\draw [fill=black] (8.5,1) circle (1.5pt);
\draw [fill=black] (8.5,3) circle (1.5pt);
\draw [fill=black] (7.5,2.) circle (1.5pt);
\end{scriptsize}
   \end{tikzpicture}
    \caption{Obtaining a loop or a double edge by inserting an edge.}
    \label{fig:planarquasi}
  \end{figure}

 These cases have to be subtracted again in order to obtain a valid formula. First suppose that $v_r=v$. As 
 we do not have double edges in $S'$, this is only possible if the corner between $c_r'$ and $c'$ is at a vertex of degree one.
 We have to subtract $y^2u \quasis_0(y,u)$, i.e.\ the fourth term of the right-hand side of \eqref{baseplanar}, for this case (we add one vertex and two edges to a quasi triangulation 
 and increase the root face valency by one).
 Now suppose that $v_r$ and $v$ are adjacent, i.e. inserting an edge between them creates a double edge. In this case 
 zipping the double edge separates the quasi triangulation into two quasi triangulations $S_1,S_2$. For one of them, without
  loss of generality for $S_1$, the root face valency is the same as for $S$, while the root face of $S_2$ has valency three.
  Thus, $S_1$ is in $\class{\quasis}_0$ but not the one-vertex map, while $S_2\in\class{\triangs}_0$. Summing up we 
 have to subtract $\triangs_0(y)(\quasis_0(y,u)-1)$, the fifth term of the right-hand side of \eqref{baseplanar}.
 \end{proof}

We can use the quadratic method (see e.g. \cite{quadratic-method}) to obtain the main result for 
planar triangulations from \Cref{prop:baseplanar}. Those were already obtained by Tutte \cite{tutte1962-planar-triangulations} with slightly different 
parameters.
\begin{lem}\label{inductionbase}
 The dominant singularity of $\triangs_0(y)$ is $\rho_\triangs=\frac{3}{2^{8/3}}$, $\triangs_0(y)$ is $\Delta$-analytic 
 and satisfies
 \begin{align}\label{triangulations}
  \triangs_0(y)&=\frac{1}{8}-\frac{9}{16}\br{1-\rho_\triangs^{-1}y}+\frac{3}{2^{5/2}}\br{1-\rho_\triangs^{-1}y}^{3/2}
	+O\br{\br{1-\rho_\triangs^{-1}y}^{2}}.
 \end{align}
 Furthermore, for $u=f(y)$ with
 \begin{equation*}f(y)=\frac{t^{1/3}}{1+t}\quad\text{ and }\quad y=t^{1/3}(1-t),\label{simplesubstitution}\end{equation*}
 the equations
 \begin{align*}
  \quasis_0(y,f(y))&=\frac{5}{4}-\frac{3}{2^{5/2}}\br{1-\rho_\triangs^{-1}y}^{1/2}+O\br{1-\rho_\triangs^{-1}y},\\
  \left.\br{\deriv{u}\quasis_0(y,u)}\right\vert_{u=f(y)}&= \frac{75}{2^{13/3}}-\frac{125\cdot 3^{3/4}}{2^{23/3}}\br{1-\rho_\triangs^{-1}y}^{1/4}
  +O\br{\br{1-\rho_\triangs^{-1}y}^{1/2}}
 \end{align*}
 hold and $\quasis_0(y,f(y))$ is $\Delta$-analytic. Let $n\ge 2$ be an integer. Then
 \begin{align*}
  \left.\br{\derivn{u}{n}\quasis_0(y,u)}\right\vert_{u=f(y)}=c(n)\br{1-\rho_\triangs^{-1}y}^{-n/2+3/4}+
	O\br{\br{1-\rho_\triangs^{-1}y}^{-n/2+1}},
 \end{align*}
 where $c(n)$ is a positive constant depending only on $n$.
\end{lem}

\begin{proof}
 Multiplying \eqref{baseplanar} by $4yu^4$ and rearranging the terms yields 
 \begin{equation}\label{quadrat}
  \br{2yu^3\quasis_0(y,u)+q(y,u)}^2=q(y,u)^2+4y^2u^3-4yu^4-4yu^4\triangs_0(y),
 \end{equation}
 where $q(x,u)=y-y^2u^2-u-u\triangs_0(y)$. Let
  \begin{align*}
    Q(y,u)&=2yu^3\quasis_0(y,u)+q(y,u)
    &\text{and}&\\
    R(x,u)&=q(y,u)^2+4y^2u^3-4yu^4-4yu^4\triangs_0(y).
    &&
  \end{align*}
 Then \eqref{quadrat} reduces to $Q^2(y,u)=R(y,u)$. 
 To obtain the claimed asymptotic 
 behaviour one chooses $u=f(y)$ in such a way that $Q(y,f(y))=0$. This $u$ is a double zero of $Q^2(y,u)$ and therefore 
 both $R(y,u)$ and $\deriv{u}R(y,u)$ are $0$ at $u=f(y)$, giving
 \begin{align*}
	0&=q(y,u)^2+4y^2u^3-4yu^4-4yu^4\triangs_0(y),\\
	0&=2q(y,u)(1+\triangs_0(y)+2y^2u)+16y u^3+16y u^3\triangs_0(y)-12y^2u^2.
 \end{align*}
By eliminating $f(y)$ from this system we obtain the implicit equation 
\begin{equation*}
 \triangs_0(y)^4 + 3 \triangs_0(y)^3+ \triangs_0(y)^2 (3 + 8 y^3) + \triangs_0(y) (1 - 20 y^3) = (1 - 16 y^3) y^3.
\end{equation*}
By standard methods for implicitly given functions (e.g.\ \cite[VII.7.1]{Flajolet2009-analytic-combinatorics}) we obtain the dominant singularity and the singular 
expansion of $\triangs_0(y)$ as stated in \eqref{triangulations}.

Conversely, by eliminating $\triangs_0(y)$ and substituting $y=t^{1/3}(1-t)$ we obtain $f(y)=\frac{t^{1/3}}{1+t} =
\frac{y}{1-t^2}$ and $\triangs_0(y)=t(1-2t)$. Since $\frac1{1-t^2}$ has only nonnegative coefficients in $t$ and
$t=t(y)$ has only nonnegative coefficients by Lagrange Inversion, $f(y)$ has only nonnegative coefficients as well.
From the implicit equation for $f(y)$ we deduce that
\begin{equation}
  f(y)=\frac{2^{4/3}}{5}-\frac{2^{11/6}}{25}\br{1-\rho_\triangs^{-1}y}^{1/2}+O\br{1-\rho_\triangs^{-1}y}.\label{asyy}
\end{equation}
From $2yf(y)^3\quasis_0(y,f(y))+q(y,f(y)) = Q(y,f(y)) = 0$, \eqref{triangulations}, \eqref{asyy},
and $y = \rho_{\triangs}-\rho_{\triangs}(1-\rho_{\triangs}^{-1}y)$ we derive the claimed expression
 \begin{align*}
  \quasis_0(y,f(y))&=\frac{5}{4}-\frac{3}{2^{5/2}}\br{1-\rho_\triangs^{-1}y}^{1/2}+O\br{1-\rho_\triangs^{-1}y}.
 \end{align*}

Given $n\in\N_0$, let us write $R^{(n)}(y,u) = \derivn{u}{n}R(y,u)$.
By the choice of $f(y)$ we know that $R^{(0)}(y,f(y))=R^{(1)}(y,f(y))=0$. As $R(y,u)$ is a polynomial of degree four in $u$,
we have $R^{(n)}(y,u)=0$ for all $n\ge5$. For $n\in\{2,3,4\}$, we obtain the dominant term of $R^{(n)}(y,f(y))$ by 
first differentiating $R(y,u)$ with respect to $u$ and then substituting $u=f(y)$, \eqref{triangulations}, \eqref{asyy},
and $y = \rho_{\triangs}-\rho_{\triangs}(1-\rho_{\triangs}^{-1}y)$. This yields
\begin{align*}
 R^{(2)}(y,f(y))&=\frac{27}{2^{7/2}}\br{1-\rho_\triangs^{-1}y}^{1/2}+O\br{1-\rho_\triangs^{-1}y},\\
 R^{(3)}(y,f(y))&=-\frac{675}{2^{16/3}}+O\br{\br{1-\rho_\triangs^{-1}y}^{1/2}},\\
 R^{(4)}(y,f(y))&=-\frac{10125}{2^{23/3}}+O\br{\br{1-\rho_\triangs^{-1}y}^{1/2}}.
\end{align*}

We define $Q^{(n)}(y,u)$ and $P_0^{(n)}(y,u)$ analogously to $R^{(n)}(y,u)$. From the facts that $Q(y,f(y))=0$ and 
$\derivn{u}{n}\br{Q^2(y,u)}= R^{(n)}(y,u)$ we deduce that
\begin{align}\label{eq:qr}
  \begin{split}
  2nQ^{(1)}(y,f(y))Q^{(n-1)}(y,f(y)) =& \,R^{(n)}(y,f(y))\\
  &- \sum_{k=2}^{n-2}\binom{n}{k}Q^{(k)}(y,f(y))Q^{(n-k)}(y,f(y))
  \end{split}
\end{align}
for every $n\in\N$. For $n=2$, this implies that
\begin{equation*}
  Q^{(1)}(y,f(y)) = \ol{c}(1-\rho_{\triangs}^{-1}y)^{1/4} + O\br{(1-\rho_{\triangs}^{-1}y)^{3/4}},
\end{equation*}
where $\ol{c} = \pm\frac{3^{3/2}}{2^{11/4}}$. By differentiating $Q(y,u)=2yu^3\quasis_0(y,u)+q(y,u)$ with
respect to $u$, we deduce that
\begin{equation*}
  \quasis_0^{(1)}(y,f(y)) = \frac{75}{2^{13/3}}+\ol{c}\frac{125}{2^{59/12}3^{3/4}}\br{1-\rho_\triangs^{-1}y}^{1/4}
  +O\br{\br{1-\rho_\triangs^{-1}y}^{1/2}}.
\end{equation*}
Since $\quasis_0^{(1)}(y,u)$ is a generating function of a combinatorial class, its coefficients $[y^ku^l]\quasis_0^{(1)}(y,u)$ are nonnegative. As $f(y)$ has only
nonnegative coefficients as well, all coefficients of $\quasis_0^{(1)}(y,u)\vert_{u=f(y)}$ are nonnegative,
implying that $\ol{c} = -\frac{3^{3/2}}{2^{11/4}}$.

For $n=3$, we deduce from \eqref{eq:qr} that
\begin{equation*}
  Q^{(2)}(y,f(y)) = -\frac{675}{6\ol{c}2^{16/3}}(1-\rho_{\triangs}^{-1}y)^{-1/4} + O\br{(1-\rho_{\triangs}^{-1}y)^{1/4}}.
\end{equation*}
For $n\ge 4$, the term $R^{(n)}(y,f(y))$ is constant, while the sum on the right-hand side is nonempty. Since the
sum only involves terms $Q^{(j)}(y,f(y))$ with $2\le j\le n-2$, we deduce by induction that
\begin{equation}
 Q^{(n)}(y,f(y))=\ol{c}(n)\br{1-\rho_\triangs^{-1}y}^{-n/2+3/4}+O\br{\br{1-\rho_\triangs^{-1}y}^{-n/2+5/4}},\label{an}
\end{equation}
where $\ol{c}(n)$ is a constant depending only on $n$ and $\ol{c}(n)>0$ for $n\ge 2$.
 
 The claimed expressions of $\quasis_0^{(n)}(y,f(y))$ are now obtained by differentiating 
 \[Q(y,u)=2y^2u^3\quasis_0(y,u)+q(y,u)\] 
 $n$ times and by \eqref{triangulations}, \eqref{asyy}, \eqref{an}, and induction.
 
 As all generating functions in this proof are given by a system of algebraic equations, they are $\Delta$-analytic.
\end{proof}

Our next aim is to derive a recursion formula for $\quasis_g(y,u,z_I)$. Using the planar case in \Cref{inductionbase}
as the base case, inductively applying the recursion formula allows us to derive similar statements as \Cref{inductionbase}
for all $g$ and $I$. In order to derive the recursion formula, we will perform different surgeries on the surface
depending on the placement of the root. We distinguish four cases.

\renewcommand{\theenumi}{(\Alph{enumi})}
\begin{enumerate}
 \item\label{rootedge:bridge}
  The root edge $e_r$ is only incident with the root face $f_r$ and is a bridge;
 \item\label{rootedge:nobridge}
  $e_r$ is only incident with $f_r$ and is not a bridge;
 \item\label{rootedge:marked}
  $e_r$ is incident with $f_r$ and one marked face; and
 \item\label{rootedge:unmarked}
  $e_r$ is incident with $f_r$ and one unmarked face.
\end{enumerate}
\renewcommand{\theenumi}{(\roman{enumi})}
The recursion formula is then of the form
\begin{equation}
  \quasis_g(y,u,z_I)=A_g(y,u,z_I)+B_g(y,u,z_I)+C_g(y,u,z_I)+D_g(y,u,z_I),\label{recursion-general}
\end{equation}
where $A_g(y,u,z_I)$, $B_g(y,u,z_I)$, $C_g(y,u,z_I)$, and $D_g(y,u,z_I)$ are the generating functions of the sub-classes 
$\class{A}_{g,I}$, $\class{B}_{g,I}$, $\class{C}_{g,I}$, and $\class{D}_{g,I}$ of 
$\class{\quasis}_{g,I}$ corresponding to the four cases \ref{rootedge:bridge}, \ref{rootedge:nobridge},
\ref{rootedge:marked}, and \ref{rootedge:unmarked}, respectively.
Each of the generating functions can be further decomposed as
\begin{equation}\label{recursion:decompositions}
  \begin{split}
    A_g(y,u,z_I) &= a(y,u)\quasis_g(y,u,z_I)+\main_A(g;y,u,z_I)-\error_A(g;y,u,z_I),\\
    B_g(y,u,z_I) &= b(y,u)\quasis_g(y,u,z_I)+\main_B(g;y,u,z_I)-\error_B(g;y,u,z_I),\\
    C_g(y,u,z_I) &= c(y,u)\quasis_g(y,u,z_I)+\main_C(g;y,u,z_I)-\error_C(g;y,u,z_I),\\
    D_g(y,u,z_I) &= d(y,u)\quasis_g(y,u,z_I)+\main_D(g;y,u,z_I)-\error_D(g;y,u,z_I),
  \end{split}
\end{equation}
where $a(y,u)$, $b(y,u)$, $c(y,u)$, and $d(y,u)$ will be functions only involving the generating functions $P_0$ and $S_0$ of the planar case,
while the other functions will involve terms of the type $\quasis_{g'}(y,u,z_{I'})$ for $g'<g$ or $I'\subsetneq I$. This will enable
us to use \eqref{recursion-general} to recursively determine the dominant terms of $\quasis_g(y,u,z_I)$. In this recursion,
the functions $\main_A$, $\main_B$, $\main_C$, and $\main_D$ will contribute to the dominant term; the functions
$\error_A$, $\error_B$, $\error_C$, and $\error_D$ will turn out to be of smaller order.

We start by determining the functions for Case \ref{rootedge:bridge}. In this case, after deleting the root edge
we can split the map into two maps whose genera add up to $g$.

\begin{lem}\label{casea}
 The functions $a(y,u,z_I)$, $\main_A(g;y,u,z_I)$, and $\error_A(g;y,u,z_I)$ in \eqref{recursion:decompositions}
 are given by
 \begin{align*}
  a(y,u,z_I)&=2yu^2\quasis_0(y,u),\\
  \main_A(g;y,u,z_I)&=yu^2\sum_{t,J}\quasis_t\br{y,u,z_J}\quasis_{g-t}\br{y,u,z_{I\setminus J}},\\
  \error_A(g;y,u,z_I)&=0.
 \end{align*}
 The sum is over $t = 0,\dotsc,g$ and $J\subseteq I$ such that $(t,J)\neq(0,\emptyset)$ and 
 $(t,J)\neq(g,I)$.
\end{lem}
\begin{proof}
  Let $S$ be an $I$-quasi triangulation in $\class{A}_{g,I}$, with respect to $h\colon I\to F(S)$, say.
  By \ref{rootedge:bridge}, the union $f_r\cup e_r$ is not a disc and thus contains a non-contractible circle $C$. We
  delete $e_r$, cut along $C$, and close the two resulting holes by inserting discs. Since $e_r$ was a bridge, this
  surgery results in two components $S_1$ and $S_2$. We define the roots of $S_1$ and $S_2$ like in
  \Cref{prop:baseplanar}: let $v_1$ and $v_2$ be the end vertices of $e_r$ in $S_1$ and $S_2$, respectively.
  Without loss of generality we may assume that $v_1$ is the root vertex of $S$. In the
  cyclic order of the edges of $S$ at $v_1$, let $e_1^-$ and $e_1^+$ be the predecessor and successor of $e_r$,
  respectively. Define $e_2^-$ and $e_2^+$ analogously at $v_2$. We let $(v_1,e_1^-,e_1^+)$ and $(v_2,e_2^-,e_2^+)$
  be the root of $S_1$ and $S_2$, respectively. Denote the root faces by $f_1$ and $f_2$, respectively. These are
  the faces of $S_1$ and $S_2$ into which the discs were inserted.

  Since every face in $F(S)\setminus\{f_r\}$ corresponds to a face in $F(S_1)\setminus\{f_1\}$ or in $F(S_2)\setminus\{f_2\}$,
  $h$ induces a function $\tilde h \colon I \to (F(S_1)\cup F(S_2))\setminus\{f_1,f_2\}$. If we write $J =
  \tilde h^{-1}(F(S_1))$, then $S_1$ is a $J$-quasi triangulation on a surface of genus $t\le g$; consequently, $S_2$
  is an $(I\setminus J)$-quasi triangulation on a surface of genus $g-t$. By deleting $e_r$, we decreased the number of corners
  of $f_r$ by two; the surgery then distributed the remaining corners of $f_r$ to $f_1$ and $f_2$. Therefore, the sum of valencies
  of $f_1$ and $f_2$ is smaller by two than the valency of $f_r$. On the other hand, we clearly have $|E(S_1)| + |E(S_2)| = |E(S)| - 1$.
  The reverse operation of the surgery is to delete an open disc from each of $f_1,f_2$, glue the surfaces along the boundaries
  of these discs, and add an edge from the root corner of $S_1$ to the root corner of $S_2$. As this operation is uniquely
  defined, we deduce that
 \[  A_g(y,u,z_I)=yu^2\sum_{t=0}^g\sum_{J\subseteq I}\quasis_t\br{y,u,z_J}\quasis_{g-t}\br{y,u,z_{i\setminus J}}.\]
  Extracting the terms for $(t,J)=(0,\emptyset)$ and $(t,J)=(g,I)$ finishes the proof.
\end{proof}

For Case \ref{rootedge:nobridge}, we will cut the surface along a circle contained in $f_r\cup e_r$ and close
the holes by inserting discs. However, because the surface will not be separated by this surgery, one needs to keep track 
of where to cut and glue (to reverse the surgery). To this end we have to mark faces. Therefore, the index set $I$ will increase.

\begin{lem}\label{caseb}
 The functions $b(y,u,z_I)$, $\main_B(g;y,u,z_I)$, and $\error_B(g;y,u,z_I)$ in \eqref{recursion:decompositions}
 are given by
 \begin{align*}
  b(y,u,z_I)&=0,\\
  \main_B(g;y,u,z_I)&=yu^2\opa{{z_{i_0}}}\br{\quasis_{g-1}(y,u,z_{I\cup\{i_0\}})}\big\vert_{z_{i_0}=u},\\
  \error_B(g;y,u,z_I)&\preceq\br{1+yu^2}\opa{u}\br{\quasis_{g-1}(y,u,z_I)}.
 \end{align*}
\end{lem}

\begin{proof}
  Let $S$ be an $I$-quasi triangulation in $\class{B}_{g,I}$, with respect to $h\colon I\to F(S)$, say.
  We use the analogous surgery as in \Cref{casea}, with the difference that $S$ is not separated by cutting along the
  circle $C$. Therefore we only obtain one map $T$. One of the end vertices of $e_r$ is the root vertex $v_r$ of
  $S$. Let $e_r^-$ and $e_r^+$ be the predecessor and successor of $e_r$
  in the cyclic order of edges of $S$ at $v_r$, respectively. Then we define the root of $T$ to be $(v_r,e_r^-,e_r^+)$.
  Denote the root face of $T$ by $f'_r$; this is one of the two faces into which we inserted discs to close the holes
  during our surgery. Denote the other such face by $f_2$. We mark $f_2$ by adding a new index $i_0$ to the index
  set $I$ and extend the function $h$ to $I\cup\{i_0\}$ by setting $h(i_0) := f_2$. Then $T$ is an
  $(I\cup\{i_0\})$-quasi triangulation on $\torus{g-1}$.

  To reverse the surgery, we delete an open disc from each of $f'_r$, $f_2$, glue the surface along the boundaries
  of these discs, add a new edge $e_{\text{new}}$ from the root corner of $T$ to a corner $c_2$ of $f_2$, and let
  $(v_r,e_{\text{new}},e_r^+)$ be the new root corner.
  We thus have to mark a corner of $f_2$, which corresponds to applying the operator $\opa{{z_{i_0}}}$ to the
  generating function. After glueing, the corners of $f_2$ become corners of the new root face; we thus have to
  remove $i_0$ from the index set and replace $z_{i_0}$ by $u$ in the generating function. Like in the previous
  cases, adding $e_{\text{new}}$ increases the total number of edges by one and the valency of the root face by
  two, as $e_{\text{new}}$ lies only on the boundary of the new root face. This results in the term
  $yu^2\opa{{z_{i_0}}}\br{\quasis_{g-1}(y,u,z_{I\cup\{i_0\}})}\mid_{z_{i_0}=u}$.

  However, by this construction we have overcounted. If the vertex $v$ of the corner $c_2$ is adjacent to
  $v_r$, then $e_{\text{new}}$ will be part of a double edge; if $v=v_r$, $e_{\text{new}}$ will be a loop. We want
  to subtract all resulting maps $\tilde S$ for which $e_{\text{new}}$ is a loop or part of a double edge. Suppose first that
  $e_{\text{new}}$ is part of a double edge. Since $e_{\text{new}}$ lies only on the boundary of the root face of
  $\tilde S$, the double edge is not separating. Thus, zipping it results in an $I$-quasi triangulation $\tilde T$
  on $\torus{g-1}$. One of the two zipped edges is the root edge $e'_r$ of $\tilde T$, denote the other zipped edge
  by $e'$. Both $e'_r$ and $e'$ lie on the boundary of the root face (see \Cref{fig:caseb}). Let $v'_r$ be the
  root vertex of $\tilde T$; then $v'_r$ is one of the two copies of $v_r$. If we denote the other copy by $v'$, then
  $v'$ is an end vertex of $e'$ and thus there is a corner $c'=(v',e'',e')$ of the root face of $\tilde T$. We can
  reconstruct $\tilde S$ from $\tilde T$ in the following way: cut along $e'_r$ and $e'$ and glue the surface along
  the boundaries of the resulting holes in the unique way that identifies $v'_r$ and $v'$. Identifying the corner
  $c'$ is bounded by marking an \emph{arbitrary} corner of the root face of $\tilde T$. This corresponds to applying
  the operator $\opa{u}$ to the generating function $\quasis_{g-1}(y,u,z_I)$. As zipping a double edge does neither change
  the number of edges nor the valencies of faces, $\opa{u}\br{\quasis_{g-1}(y,u,z_I)}$ is an upper bound in this case.

  Suppose now that $e_{\text{new}}$ is a loop and recall that $(v_r,e_{\text{new}},e_r^+)$ is the root corner of
  $\tilde S$. Since $e_{\text{new}}$ lies only on the boundary of the root face of $\tilde S$, there is a unique
  edge $e_2\not=e_r^+$ such that $(v_r,e_{\text{new}},e_2)$ is a corner of the root face. We cut along $e_{\text{new}}$,
  close the two resulting holes by inserting discs, and delete the two copies of $e_{\text{new}}$.
  Again, cutting does not separate the surface. Thus, we obtain a map $\tilde T$ on $\torus{g-1}$
  that does not have loops or double edges. Let $v'_r$ be the copy of $v_r$ in $\tilde T$ that is incident with
  $e_r^+$ and let $v'_2$ be the other copy. Then the root of $\tilde T$ is $(v'_r,e',e_r^+)$ for some edge $e'$.
  Furthermore, the root face of $\tilde T$ has a corner $(v'_2,e'_2,e_2)$. Now $\tilde S$ can be reconstructed from
  $\tilde T$ in the following way (see \Cref{fig:caseb}).
  \begin{enumerate}
  \item Add a loop at each of $(v'_r,e',e_r^+)$ and $(v'_2,e'_2,e_2)$;
  \item delete the resulting two faces of valency one;
  \item identify the two loops.
  \end{enumerate}
 
  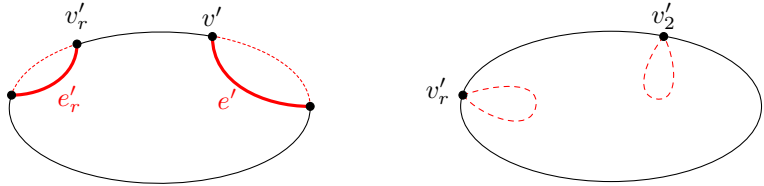
\begin{figure}[htbp]
   \begin{tikzpicture}[line cap=round,line join=round,>=triangle 45,x=1.0cm,y=1.0cm]

\draw [color=black] (3.7,3.43) arc (70:123.5:2cm and 1cm);
\draw [color=black] (1.03,2.65) arc (170:360:2cm and 1cm);

\draw [color=red, dash pattern=on 1pt off 1pt] (1.9,3.33) arc (123.5:170:2cm and 1cm);
\draw [color=red, very thick] (1.9,3.33) to[out=270,in=0] (1.03,2.65);
\draw [color=red, dash pattern=on 1pt off 1pt] (5,2.5) arc (0:70:2cm and 1cm);
\draw [color=red, very thick] (5,2.5) to[out=180,in=270] (3.7,3.43);

\draw (1.9,3.7) node {$v_r'$};
\draw (3.7,3.7) node {$v'$};
\draw [color=red](1.8,2.6) node {$e_r'$};
\draw [color=red](3.9,2.6) node {$e'$};

\draw [color=black] (11,2.5) arc (0:360:2cm and 1cm);
\draw [color=red,dash pattern= on 2pt off 2pt](7.03,2.65) to[out=310,in=260](8,2.5) to[out=80,in=20] (7.03,2.65);
\draw [color=red,dash pattern= on 2pt off 2pt](9.7,3.43) to[out=230,in=170](9.6,2.6) to[out=350,in=300] (9.7,3.43);
\draw (6.7,2.7) node {$v_r'$};
\draw (9.7,3.7) node {$v_2'$};

\begin{scriptsize}
\draw [fill=black] (5,2.5) circle (1.5pt);
\draw [fill=black] (3.7,3.43) circle (1.5pt);
\draw [fill=black] (1.9,3.33) circle (1.5pt);
\draw [fill=black] (1.03,2.65) circle (1.5pt);
\draw [fill=black] (7.03,2.65) circle (1.5pt);
\draw [fill=black] (9.7,3.43) circle (1.5pt);
\end{scriptsize}
   \end{tikzpicture}
    \caption{Deriving an upper bound for $\error_B$.}
    \label{fig:caseb}
  \end{figure}

  In order to identify the corner $(v'_2,e'_2,e_2)$, we mark an arbitrary corner of the root face, which is again
  overcounting. Since we have to add one edge to $\tilde T$ and increase the valency of the root face by two to
  reconstruct $\tilde S$, we have an additional factor of $yu^2$, resulting in the claimed upper bound for
  $\error_B$.
\end{proof}

In Case \ref{rootedge:marked}, the root edge is not a bridge. Therefore, we will not be able to find
a circle $C$ like in the previous two cases. On the other hand, deleting the root edge does not produce
any faces that are not discs. Our construction in this case will thus start without cutting the surface.

\begin{lem}\label{casec}
 The functions $c(y,u,z_I)$, $\main_C(g;y,u,z_I)$, and $\error_C(g;y,u,z_I)$ in \eqref{recursion:decompositions}
 are given by
 \begin{align*}
  c(y,u,z_I)&=0,\\
  \main_C(g;y,u,z_I)&=y\sum_{i\in I}\sum_{T\in\class{\quasis}_g(I\setminus\{i\})}y^{|E(T)|}
	\prod_{j\neq i}z_j^{\beta_j(T)}\sum_{k=1}^{\beta(T)+1}u^kz_i^{\beta(T)+2-k},\\
  \error_C(g;y,u,z_I)\preceq& \,\sum_{i\in I}\bigg(\br{1+yu z_i}\opa{{z_i}}\br{\quasis_{g-1}(y,u,z_I)}\nonumber\\
	&+\br{1+yu z_i}\sum_{t=0}^{g}\sum_{J\subseteq I\setminus\{i\}} \quasis_t(y,u,z_J)\quasis_{g-t}(y,z_i,z_{I\setminus (J\cup\{i\})})\bigg),
 \end{align*}
  where $\beta(T)$ and $\beta_j(T)$ denote the valencies of the root face of $T$ and of the face with index $j$
  in $T$, respectively.
\end{lem}
Note that the sum in $\main_C$ is over all $i\in I$ and all $I\setminus\{i\}$-quasi triangulations. As such,
$\main_C$ can be written as
\[
 M_C = y \sum_{i \in I} \frac{u^2z_i P_g(y,u,z_{I\setminus \{ i \} }) - u z_i^2  P_g(y,z_i,z_{I\setminus \{ i \} })}{u-z_i}.
\]
However, similarly to \Cref{inductionbase}, we will replace $u$ and $z_i$ by $f(y)$ in order to derive the desired
asymptotic formulas, which would result in a division by $0$. For that reason we will use $\main_C$ as stated in
\Cref{casec}. We will derive a more convenient formulation in \Cref{maincderiv}.

\begin{proof}[Proof of \Cref{casec}]
  Let $S$ be an $I$-quasi triangulation in $\class{C}_{g,I}$, with respect to $h\colon I\to F(S)$, say.
  We delete the root edge $e_r$, thus obtaining a map $T$ on $\torus{g}$. The root of $T$ is defined as follows. Let
  $e_r^-$ and $e_r^+$ be the predecessor and successor of $e_r$ at $v_r$, respectively; then
  $(v_r,e_r^-,e_r^+)$ is the root of $T$. By \ref{rootedge:marked}, $e_r$ was incident with a marked face $h(i)$.
  The root face of $T$ is $f'_r := f_r \cup e_r \cup h(i)$ and $T$ is an $(I\setminus\{i\})$-quasi triangulation with
  respect to $h\vert_{I\setminus\{i\}}$.

  Let $c$ be a corner of $f'_r$ and let $\tilde S$ be obtained from $T$ by adding an edge
  $e_{\text{new}}$ from $(v_r,e_r^-,e_r^+)$ to $c$ and let the root of $\tilde S$ be
  \begin{itemize}
  \item $(v_r,e_{\text{new}},e_r^+)$ if $c\not=(v_r,e_r^-,e_r^+)$ and
  \item either $(v_r,e_{\text{new}},e_r^+)$ or $(v_r,e_{\text{new}},e_{\text{new}})$ otherwise,
  \end{itemize}
  see \Cref{fig:casec}.
  Adding $e_{\text{new}}$ divides $f'_r$ into two faces. One of these faces is the root face of $\tilde S$;
  we mark the other face with the index $i$ and denote the corresponding function $I\to F(\tilde S)$ by $\tilde h$.
  Clearly, there is a unique choice of $c\not=(v_r,e_r^-,e_r^+)$
  such that $\tilde S = S$. If $c$ is a corner at $v_r$ (in particular if $c=(v_r,e_r^-,e_r^+)$), then
  $e_{\text{new}}$ will be a loop. If $c$ is a corner at a vertex adjacent to $v_r$, then $e_{\text{new}}$
  will be part of a double edge. In either case, $\tilde S$ will not be simple and thus not an $I$-quasi
  triangulation. Although the case $c=(v_r,e_r^-,e_r^+)$ is clearly one of the cases when $\tilde S$ is
  not simple, it is slightly easier to derive the formulas including this case.

 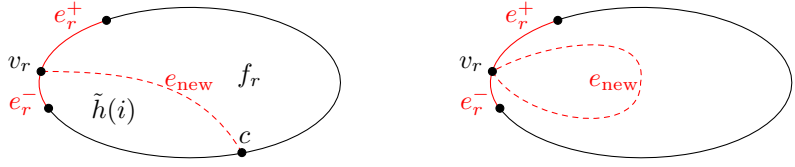
\begin{figure}[htbp]
   \begin{tikzpicture}[line cap=round,line join=round,>=triangle 45,x=1.0cm,y=1.0cm]

\draw [color=black] (1.13,2.16) arc (200:483.5:2cm and 1cm);

\draw [color=red] (1.9,3.33) arc (123.5:200:2cm and 1cm);
 \draw [color=red, dash pattern= on 2pt off 2pt] (3.7,1.57) to[out=120,in=0] (1.03,2.65);
\draw (0.75,2.75) node {$v_r$};
\draw [color=red](0.8,2.25) node {$e_r^-$};
\draw [color=red](1.4,3.4) node {$e_r^+$};
\draw [color=red](3,2.5) node {$e_{\text{new}}$};
\draw (2,2.15) node {$\tilde h(i)$};
\draw (3.8,2.6) node {$f_r$};
\draw (3.7,1.57) node[anchor=260] {$c$};

\draw [color=black] (7.13,2.16) arc (200:483.5:2cm and 1cm);

\draw [color=red] (7.9,3.33) arc (123.5:200:2cm and 1cm);
\draw [color=red,dash pattern= on 2pt off 2pt](7.03,2.65) to[out=300,in=270](9,2.5) to[out=80,in=30] (7.03,2.65);

\draw (6.75,2.75) node {$v_r$};
\draw [color=red](6.8,2.25) node {$e_r^-$};
\draw [color=red](7.4,3.4) node {$e_r^+$};
\draw [color=red](9.1,2.5) node[anchor=east] {$e_{\text{new}}$};

\begin{scriptsize}
\draw [fill=black] (1.13,2.16) circle (1.5pt);
\draw [fill=black] (3.7,1.57) circle (1.5pt);
\draw [fill=black] (1.03,2.65) circle (1.5pt);
\draw [fill=black] (1.9,3.33) circle (1.5pt);
\draw [fill=black] (7.9,3.33) circle (1.5pt);
\draw [fill=black] (7.13,2.16) circle (1.5pt);
\draw [fill=black] (7.03,2.65) circle (1.5pt);
\end{scriptsize}
   \end{tikzpicture}
  \caption{Adding the edge $e_{\text{new}}$ from $(v_r,e_r^-,e_r^+)$ to $c$ to obtain $\tilde S$.
    If $c=(v_r,e_r^-,e_r^+)$, then each of the two faces can either be the root face or $\tilde h(i)$.}
  \label{fig:casec}
\end{figure}

  As $f'_r$ has valency $\beta(T)$, there are $\beta(T)+1$ choices for $\tilde S$. The valency of the root face
  of $\tilde S$ is one if $c=(v_r,e_r^-,e_r^+)$ and the root face is $(v_r,e_{\text{new}},e_{\text{new}})$.
  If $c=(v_r,e_r^-,e_r^+)$ and the root face is $(v_r,e_{\text{new}},e_r^+)$, the valency is $\beta(T)+1$. Depending
  on which corner is chosen as $c$, the valency of the root face can take any value $k$ between one and $\beta(T)+1$;
  the face $\tilde h(i)$ then has valency $\beta(T)+2-k$. The generating function of maps that can occur as $\tilde S$
  from this particular $(I\setminus\{i\})$-quasi triangulation $T$ is thus given by
 \[y^{|E(T)+1|}\prod_{j\in I\setminus\{i\}}z_j^{\beta_j(T)}\sum_{k=1}^{\beta(T)+1}u^kz_i^{\beta(T)+2-k}.\]
 This holds as the number of edges is increased by one and the valencies of all other marked faces do not change. 
 After summing over all possible marked faces and all possible $T$, we obtain $\main_C$.

  As already mentioned, we overcount whenever the chosen corner $c$ is at $v_r$ or at a vertex adjacent
  to $v_r$, making $e_{\text{new}}$ a loop or part of a double edge, respectively. Suppose first that $e_{\text{new}}$
  is part of a double edge. We zip the double edge. If it does not separate the surface, we have an upper bound
  $\opa{{z_i}}(\quasis_{g-1}(y,u,z_I))$ analogous to \Cref{caseb}. Indeed, the only difference to the corresponding
  case in \Cref{caseb} is that we mark a corner of (the face corresponding to) $\tilde h(i)$ instead of a corner of
  the root face, because $e_{\text{new}}$ was incident with both the root face and $\tilde h(i)$. If the double edge
  separates the surface, we obtain two maps $T_1$ on $\torus{t}$ for $0\le t\le g$ and $T_2$ on $\torus{g-t}$. One of
  the two maps, without loss of generality $T_2$, contains (the face corresponding to) $\tilde h(i)$. As $T_2$ is
  rooted at a corner of that face and the root face is never marked, the number of marks decreases by one. Thus, $T_1$ is
  a $J$-quasi triangulation on $\torus{t}$ and $T_2$ is a $(I\setminus(J\cup\{i\}))$-quasi triangulation on $\torus{g-t}$,
  where $J\subseteq I\setminus\{i\}$. Going back, all corners of the root face of $T_2$ become corners of the face with
  index $i$, meaning that we have to replace $u$ by $z_i$ in $\quasis_{g-t}(x,z_i,z_{I\setminus (J\cup\{i\})})$. This gives
  us an upper bound of
  \begin{equation*}
    \sum_{t=0}^{g}\sum_{J\subseteq I\setminus\{i\}} \quasis_t(x,u,J)\quasis_{g-t}(x,z_i,z_{I\setminus (J\cup\{i\})}).
  \end{equation*}

  If $e_{\text{new}}$ is a loop, then we proceed the same way as in \Cref{caseb}: we cut along $e_{\text{new}}$,
  close the two resulting holes by inserting discs, and delete the two copies of $e_{\text{new}}$. Like in \Cref{caseb},
  the reverse construction yields the same bounds as in the case of $e_{\text{new}}$ being part of a double edge; the
  additional factor $yuz_i$ is due to the fact that we add one edge and increase the valencies of the root face and of
  $\tilde h(i)$ by one.
\end{proof}

The construction in Case \ref{rootedge:unmarked} is similar to Case \ref{rootedge:marked}. The fact that the second face incident with
$e_r$ is not marked makes the analysis easier.

\begin{lem}\label{cased}
 The functions $d(y,u,z_I)$, $\main_D(g;y,u,z_I)$, and $\error_D(g;y,u,z_I)$ in \eqref{recursion:decompositions}
 are given by
 \begin{align*}
  d(y,u,z_I)&=yu^{-1}-y^2u-\triangs_0(y),\\
  \main_D(g;y,u,z_I)&=-\triangs_g(y,z_I)\quasis_0(y,u),\\
  \error_D(g;y,u,z_I)\preceq&\,\sum_{t,J}\triangs_t(y,J)\quasis_{g-t}(y,u,z_{I\setminus J})\nonumber\\
    &+3\quasis_{g-1}(y,u,z_{I\cup\{i_0\}})\vert_{z_{i_0}=u}.
 \end{align*}
 The sum is over $t = 0,\dotsc,g$ and $J\subseteq I$ such that $(t,J)\neq(0,\emptyset)$ and 
 $(t,J)\neq(g,I)$.
\end{lem}
\begin{proof}
%
  Let $S$ be an $I$-quasi triangulation in $\class{D}_{g,I}$. 
  We delete $e_r$ and choose the root of the resulting map $T$ to be $c'_r:=(v_r,e_r^-,e_r^+)$ like in \Cref{casec}.
  As the second face $f$ incident with $e_r$ is not marked and $S$ is an $I$-quasi triangulation, $f$ is
  bounded by a triangle. Thus, $T$ is also an $I$-quasi triangulation and the valency of its root face $f'_r$ is
  larger by one than the valency of $f_r$. For the reverse construction, consider the ordering of the corners
  of $f'_r$ in clockwise direction along its boundary and let $c$ be the corner after the next, starting from
  $c'_r$. We add an edge $e_{\text{new}}$ from $c'_r$ to $c$ and let $(v_r,e_{\text{new}},e_r^+)$ be the root
  of the resulting $I$-quasi triangulation $\tilde S$. If $T$ was obtained from $S$ by deleting $e_r$, then
  $\tilde S=S$. However, if $T$ is an \emph{arbitrary} $I$-quasi triangulation on $\torus{g}$, then $e_{\text{new}}$
  might be a loop or part of a double edge. Thus,
  \begin{equation}\label{eq:nomark:upper}
    yu^{-1}\quasis_g(y,u,z_I)
  \end{equation}
  is only an upper bound for $D_g(y,u,z_I)$. Again, we have to subtract the cases when $\tilde S$ is not simple.

  The case when $e_{\text{new}}$ is a loop yields a term of
  \begin{equation}\label{eq:nomark:loop}
    -y^2u\quasis_g(x,u,z_I)
  \end{equation}
  analogously to \Cref{prop:baseplanar}. When $e_{\text{new}}$ is part of a double edge, we need to distinguish
  whether this double edge separates the surface. If it does separate, we obtain
  \begin{equation}\label{eq:nomark:separate}
    -\sum_{t=0}^{g}\sum_{J\subseteq I}\triangs_t(y,J)\quasis_{g-t}(y,u,z_{I\setminus J})
  \end{equation}
  by zipping the double edge, similar to \Cref{casea}. The only differences are that the number of edges and the valencies of the faces do
  not change and that one of the two components is a $J$-triangulation, since its root face is $f$ and thus has valency three.
  Finally, if the double edge does not separate, then after zipping it we have to mark $f$ with a new index $i_0$ like
  in \Cref{caseb}. However, since the valency of $f$ is three, we only have three possibilities how to reverse the
  construction. As the number of edges and all valencies remain unchanged, we have a summand
  \begin{equation}\label{eq:nomark:nonseparate}
    -3\quasis_{g-1}(y,u,z_{I\cup\{i_0\}})\vert_{z_{i_0}=u}\,.
  \end{equation}
  Note that \eqref{eq:nomark:nonseparate} is overcounting as the reverse construction can lead to additional loops or
  double edges.

  Combining \eqref{eq:nomark:upper}, \eqref{eq:nomark:loop}, and the term from \eqref{eq:nomark:separate} with $(t,J) = (0,\emptyset)$,
  we deduce the claimed expression for $d(y,u,z_I)$. The term from \eqref{eq:nomark:separate} with $(t,J) = (g,I)$ yields
  $\main_D(g;y,u,z_I)$; the remaining terms form the upper bound for $\error_D(g;y,u,z_I)$.
\end{proof}

The only term having a different structure than the others is $\main_C$ which cannot be easily expressed in terms of  
$\quasis_{g'}(y,u,z_{I'})$ and $\triangs_{g'}(y,I')$ with some genus $g'$ and set $I'$. From $\main_B$ and $\error_B$ we observe that we need to calculate derivatives with respect to 
$u$ and $z_{i_0}$ and that we want to set $z_{i_0}=u$ in the end. We will be interested in the dominant term of $\quasis_g(y,u,z_I)$
when we set $u=f(y)$ and $z_i=f(y)$ for all $i\in I$; we will abbreviate this by $u=z_I=f(y)$. Observe that setting $u=z_I=f(y)$
does not have any influence on the functions $\triangs_g(y)$, as they only depend on the variable $y$.

The following proposition enables us to express arbitrary derivatives of $\main_C$ at $u=z_I=f(y)$ in terms of derivatives 
of $\quasis_g$.

\begin{prop}\label{maincderiv}
 Let $|y|<\rho_{\triangs}$, $n\in\N_0$, and $\alpha_i\in\N_0$ for all $i\in I$. Write $|\alpha_I|$ for $\sum \alpha_i$. Then
 \begin{align}\label{eq:mainc}
  &\frac{\partial^{n+|\alpha_I|}}{\partial u^n\prod_{i\in I}\partial z_i^{\alpha_i}} \main_C(g;y,u,z_I)\Bigg\vert_{u=z_I=f(y)}\\
	&=y\br{\sum_{i\in I}\frac{n!\alpha_i!}{(n+\alpha_i+1)!}\frac{\partial^{n+1+|\alpha_I|}}{\partial u^{n+1+\alpha_i}
		\prod_{j\in I\setminus\{i\}}\partial z_j^{\alpha_j}}\br{u^3\quasis_g(y,u,z_{I\setminus\{i\}})}}\Bigg\vert_{u=z_I=f(y)} \,.\nonumber
 \end{align}
\end{prop}
\begin{proof}
 The generating function $yu^3\quasis_g(y,u,z_{I\setminus\{i\}})$ is given by
 \begin{equation*}
  yu^3\quasis_g(y,u,z_{I\setminus\{i\}})=y\sum_{T\in\class{\quasis}_g(I\setminus\{i\})} y^{|E(T)|}u^{\beta(T)+3}\prod_{j\in 
	I\setminus\{i\}}z_j^{\beta_j(T)}.
 \end{equation*}
 By comparing this term with the summand in
 \begin{equation*}
    \main_C(g;y,u,z_I)=y\sum_{i\in I}\sum_{T\in\class{\quasis}_g(I\setminus\{i\})}y^{|E(T)|}
	\prod_{j\neq i}z_j^{\beta_j(T)}\sum_{k=1}^{\beta(T)+1}u^kz_i^{\beta(T)+2-k}
 \end{equation*}
  for a fixed index $i\in I$,
 one sees that the difference between them is that the factor $u^{\beta(T)+3}$ is replaced 
 by $\sum_{k=1}^{\beta(T)+1}u^kz_i^{\beta(T)+2-k}$. Taking the derivatives with respect to $u$ and $z_i$ the given number of 
 times and comparing the coefficients yields factors
  \begin{align*}
    \frac{(\beta(T)+3)!}{(\beta(T)+2-n-\alpha_i)!}&u^{\beta(T)+2-n-\alpha_i}
    &\text{and}&\\
    \sum_{k=n}^{\beta(T)+2-\alpha_i}\frac{k!(\beta(T)+2-k)!}{(k-n)!(\beta(T)+2-k-\alpha_i)!}&u^{\beta(T)+2-n-\alpha_i},
    &&
  \end{align*}
  respectively, when $n+\alpha_i+1 \le \beta(T)+3$ and factors $0$ otherwise. The quotient of these two coefficients
  equals $\frac{n!\alpha_i!}{(n+\alpha_i+1)!}$ by a binomial identity. Summing over $i\in I$ finishes the proof.
\end{proof}

The only other term where differentiating is not straight forward is $\main_B$.
By using the chain rule $n$ times we obtain 
 \begin{align}
   \frac{\partial^{n+|\alpha_I|}}{\partial u^n\prod_{i\in I}\partial z_i^{\alpha_i}}&\main_B(g;y,u,z_I)\label{mainb}\\
  =&\frac{\partial^{n+|\alpha_I|}}{\partial u^n\prod_{i\in I}\partial z_i^{\alpha_i}}\br{yu^2\br{z_{i_0}\deriv{{z_{i_0}}}\quasis_{g-1}\br{y,u,z_{I\cup\{i_0\}}}}\bigg\vert_{z_{i_0}=u}}\nonumber\\
	=&y\sum_{k=0}^n\binom{n}{k}\br{\frac{\partial^{n-k+|\alpha_I|}}{\partial u^{n-k}\prod_{i\in I}\partial z_i^{\alpha_i}}\derivn{z_{i_0}}{k+1}\br{u^3\quasis_{g-1}\br{y,u,z_{I\cup\{i_0\}}}}}\bigg\vert_{z_{i_0}=u}.\nonumber
\end{align}

Using \eqref{eq:mainc} and \eqref{mainb} we can now determine the dominant terms of the derivatives
of $S_g$ and $P_g$.

\begin{thm}\label{finalind}
  Let $\alpha_i\in\N_0$, $i\in I$, and $|\alpha_I|:=\sum\alpha_i$.
 If $(g,I)\neq (0,\emptyset)$, then
 \begin{equation}\label{rec-triang}
  \frac{\partial^{|\alpha_I}|}{\prod\partial z_i^{\alpha_i}}\triangs_g(y,z_{I})\bigg\vert_{z_I=f(y)}\cong a_0 + c_g\br{1-\rho_{\triangs}^{-1}y}^{e_1}+O\br{\br{1-\rho_{\triangs}^{-1}y}^{e_1+1/4}},
 \end{equation}
 where $a_0$ and $c_g=c_g(\alpha_i,i\in I)$ are positive constants and
 \[e_1=-\frac{5g}{2}-\frac{5|I|}{4}-\frac{|\alpha_I|}{2}+\frac{3}{2}.\]
  $\left.\derivn{u}{n}P_0(y,u)\right\vert_{u=f(y)}$ is given as in \Cref{inductionbase}.
 If $(g,I,n)\neq (0,\emptyset,0)$, then
 \begin{equation}\label{rec-quasi}
  \frac{\partial^{n+|\alpha_I|}}{\partial u^n\prod\partial z_i^{\alpha_i}} \quasis_g(y,u,z_{I})\bigg\vert_{u=z_I=f(y)}\cong c \br{1-\rho_{\triangs}^{-1}y}^{e_2}+O\br{\br{1-\rho_{\triangs}^{-1}y}^{e_2+1/4}},
 \end{equation}
 where $c=c(g,|I|,n,|\alpha_I|)$ is a positive constant and 
 \[e_2=e_1-\frac{n}{2}-\frac{3}{4}.\]
 \end{thm}

\begin{proof}
 We show this by induction on $(g,|I|,n)$ in lexicographic order. \Cref{inductionbase} shows that \eqref{rec-quasi} is
  true for $(g,I)=(0,\emptyset)$ and $n>0$. Note that $|\alpha_I|=0$ for $I=\emptyset$.
 
 Suppose now that \eqref{rec-quasi} is true for all $(g,|I|,n)< (g_0,|I_0|,0)$ and \eqref{rec-triang} is true for all 
 $(g,|I|)<(g_0,|I_0|)$ with $(g,I)\not=(0,\emptyset)$. We first prove that \eqref{rec-triang} then holds for $(g_0,|I_0|)$. By multiplying
  \eqref{recursion-general} by $u$ and applying \Cref{casea,caseb,casec,cased} we obtain
 \begin{equation*}
  u(1-a-d)\quasis_{g_0}(y,u,z_{I_0})=u(\main_A+\main_B+\main_C)-u\triangs_{g_0}(y,z_{I_0})\quasis_0(y,u)-u\error,
 \end{equation*}
 where $\error=\error_B+\error_C+\error_D$.
 The term
  \begin{equation*}
    u(1-a-d)=u-2yu^3\quasis_0(y,u)-u\triangs_0(y)-y+y^2u^2
  \end{equation*}
 is equal to $-Q(y,u)$ in \eqref{quadrat} and thus
 \begin{equation}\label{recursionexplicit}
  -Q(y,u)\quasis_{g_0}(y,u,z_{I_0})=u(\main_A+\main_B+\main_C)-u\triangs_{g_0}(y,z_{I_0})\quasis_0(y,u)-u\error.
 \end{equation}
 Therefore, the left-hand side is zero when replacing $u$ by $f(y)$. As this factor is independent 
 of $z_{I_0}$, this does also hold when differentiating the equation $\alpha_i$ times with respect to $z_i$. Thus we obtain
 \begin{align*}
  u\quasis_0(y,u)\frac{\partial^{|\alpha_{I_0}|}\triangs_{g_0}(y,z_{I_0})}{\prod\partial z_i^{\alpha_i}}\bigg\vert_{u=z_{I_0}=f(y)}&=
	u\frac{\partial^{|\alpha_{I_0}|}(\main_A+\main_B+\main_C-\error)}{\prod\partial z_i^{\alpha_i}}\bigg\vert_{u=z_{I_0}=f(y)}.
 \end{align*}
 By inspecting the formulas for $\main_A$ to $\error_D$ in \Cref{casea,caseb,casec,cased} one sees that all 
 occurring terms are lexicographically smaller
 than $(g_0,|I_0|,0)$ and therefore the induction hypothesis can be used. Inspection of the exponents of $(1-\rho_{\triangs}^{-1}y)$ in all those 
 terms shows the following.
 
 \begin{itemize}
  \item[$\main_A$:] The summands have the form 
	\begin{equation*}
    \frac{\partial^{|\alpha_{I_0}|}}{\prod\partial z_i^{\alpha_i}}\br{\quasis_t\br{y,u,z_J}\quasis_{g_0-t}\br{y,u,z_{I_0\setminus J}}}
  \end{equation*}
  for $(t,J)\not=(0,\emptyset)$ and $(t,J)\not=(g_0,I_0)$. Thus, for $g_0+|I_0|\le 1$, we
  have $\main_A=0$. For all other values of $(g_0,|I_0|)$, each summand is of the form $c_A(1-\rho_{\triangs}^{-1}y)^{m_A}+O\br{(1-\rho_{\triangs}^{-1}y)^{m_A+1/4}}$ with
	\begin{align*}
		m_A&=-\frac{5t}{2}-\frac{5|J|}{4}-\frac{|\alpha_J|}{2}+\frac{3}{4}-\frac{5(g_0-t)}{2}-\frac{5|I_0\setminus J|}{4}-\frac{|\alpha_{I_0\setminus J}|}{2}+\frac{3}{4}=e_1
	\end{align*}
	by induction. Furthermore, all coefficients are positive by induction.

  \item[$\main_B$:] We have $\main_B = yu^2\opa{{z_{i_0}}}\br{\quasis_{{g_0-1}}(y,u,z_{I\cup\{i_0\}})}\mid_{z_{i_0}=u}$. Thus, for $g_0=0$
  we have $\main_B=0$ and otherwise $\main_B = c_B(1-\rho_{\triangs}^{-1}y)^{m_B}+O\br{(1-\rho_{\triangs}^{-1}y)^{m_B+1/4}}$ with
	\begin{align*}
	 m_B&=-\frac{5(g_0-1)}{2}-\frac{5|I_0\cup\{i_0\}|}{4}-\frac{|\alpha_{I_0}|+1}{2}+\frac{3}{4}=e_1
	\end{align*}
  by induction. Again, the coefficient is positive by induction.

  \item[$\main_C$:] We determine the expression for $\main_C$ by \Cref{maincderiv}. For $g_0=0$ and $|I_0|=1$
  we have $\main_C \cong c_{C,1}(1-\rho_{\triangs}^{-1}y)^{1/4}+O\br{(1-\rho_{\triangs}^{-1}y)^{1/2}}$, which is of the
  desired order, since $e_1=1/4$ in this case. For all other $(g_0,I_0)$, induction yields that
  $\main_C = c_{C,2}(1-\rho_{\triangs}^{-1}y)^{m_C}+O\br{(1-\rho_{\triangs}^{-1}y)^{m_C+1/4}}$ with
	\begin{align*}
	 m_C&=-\frac{5g_0}{2}-\frac{5(|I_0|-1)}{4}-\frac{|\alpha_{I_0}|}{2}-\frac{1}{2}+\frac{3}{4}=e_1\,.
	\end{align*}
	Like in the previous cases, the coefficient is positive by induction.

  \item[$\error_B$:] The function $\error_B$ is bounded from above by 
	$\br{yu^2+1}\opa{u}\br{\quasis_{g_0-1}(y,u,z_{I_0})}$. For $g_0=0$, we thus have $\error_B=0$ and
  otherwise $\error_B = O\br{(1-\rho_{\triangs}^{-1}y)^{e_B}}$ with
	\begin{align*}
	 e_B&=-\frac{5(g_0-1)}{2}-\frac{5|I_0|}{4}-\frac{|\alpha_{I_0}|}{2}-\frac{1}{2}+\frac{3}{4}\ge e_1+\frac14\,.
	\end{align*}

  \item[$\error_C$:] The first summand in the expression of $\error_C$ from \Cref{casec} is $0$ if $g_0=0$ and
  otherwise $O((1-\rho_{\triangs}^{-1}y)^{e_{C,1}})$ with
	\begin{align*}
	 e_{C,1}&=-\frac{5(g_0-1)}{2}-\frac{5|I_0|}{4}-\frac{|\alpha_{I_0}|+1}{2}+\frac{3}{4}\ge e_1+\frac14\,.
	\end{align*}
	The second summand is $a_{\error}+O\br{(1-\rho_{\triangs}^{-1}y)^{1/2}}$ if $g_0=0$ and $|I_0|=1$. Suppose $(g_0,|I_0|)\not=(0,1)$.
  Then every term \[\br{1+yu z_i}\quasis_t(y,u,z_J)\quasis_{g-t}(y,z_i,z_{I\setminus (J\cup\{i\})})\] with
  $(t,J)\not=(0,\emptyset)$ and $(t,J)\not=(g_0,I_0\setminus\{i\})$ is 
  $O((1-\rho_{\triangs}^{-1}y)^{e_{C,2}})$ with
	\begin{align*}
		e_{C,2}&=-\frac{5t}{2}-\frac{5|J|}{4}-\frac{|\alpha_J|}{2}+\frac{3}{4}-\frac{5(g_0-t)}{2}-\frac{5(|I_0\setminus J|-1)}{4}-\frac{|\alpha_{I_0\setminus J}|-\alpha_i}{2}+\frac{3}{4}\\
		&\ge e_1+\frac14
	\end{align*}
	by induction. The corresponding terms for $(t,J)=(0,\emptyset)$ and $(t,J)=(g_0,I_0\setminus\{i\})$ are
  $O\br{(1-\rho_{\triangs}^{-1}y)^{e_{C,3}}}$ with
  \begin{align*}
    e_{C,3}&=-\frac{5g_0}{2}-\frac{5(|I_0|-1)}{4}-\frac{|\alpha_{I_0}|-\alpha_i}{2}+\frac{3}{4}\ge e_1+\frac14\,.
  \end{align*}
  In total, we have $\error_C = O\br{(1-\rho_{\triangs}^{-1}y)^{e_1+1/4}}$.

  \item[$\error_D$:] The first summand in the expression of $\error_D$ from \Cref{cased} is $0$ if $g_0+|I_0|\le 1$ and
  otherwise each of its summands is $O\br{(1-\rho_{\triangs}^{-1}y)^{e_{D,1}}}$ with
	\begin{align*}
		e_{D,1}&=-\frac{5t}{2}-\frac{5|J|}{4}-\frac{|\alpha_J|}{2}+\frac{3}{2}-\frac{5(g_0-t)}{2}-\frac{5|I_0\setminus J|}{4}-\frac{|\alpha_{I_0\setminus J}|}{2}+\frac{3}{4}\\
		&\ge e_1+\frac14
	\end{align*}
  by induction. The second term is $0$ for $g_0=0$ and $O\br{(1-\rho_{\triangs}^{-1}y)^{e_{D,2}}}$ otherwise with
	\begin{align*}
	 e_{D,2}&=-\frac{5(g_0-1)}{2}-\frac{5(|I_0|+1)}{4}-\frac{|\alpha_{I_0}|}{2}+\frac{3}{4}\ge e_1+\frac14\,.
	\end{align*}
\end{itemize}
Combining these results, we have proved \eqref{rec-triang}, where $a_0=a_{\main}-a_{\error}$ for $g_0=0$
and $|I_0|=1$. As this constant is the value of the generating function $\triangs_0(y,z_{I_0})\vert_{z_{I_0}=f(y)}$
at its singularity $\rho_{\triangs}$, $a_0$ is positive. For $|\alpha_{I_0}|>0$ or $(g_0,|I_0|)\not=(0,1)$, the exponent
$e_1$ is negative and thus \eqref{rec-triang} is true with the same value for $a_0$. Finally, $c_g$ is positive, since
it is the sum of positive numbers.

To prove \eqref{rec-quasi}, recall that we assume that \eqref{rec-quasi} is true for $(g,|I|,n)<(g_0,|I_0|,0)$ and
we have already shown that \eqref{rec-triang} is true for $(g,|I|)\le(g_0,|I_0|)$. Let $n_0\in\N_0$ and assume that
\eqref{rec-quasi} is also true for $(g_0,|I_0|,n)$ with $n<n_0$. Consider the derivative $\derivn{u}{n+1}$ of
\eqref{recursionexplicit} and set $u=f(y)$; as $Q(y,f(y))=0$, this yields
\begin{equation*}
  -\left.\sum_{k=0}^{n}\binom{n+1}{k}\derivn{u}{k}P_{g_0}(y,u,z_{I_0})\derivn{u}{n+1-k}Q(y,u)\right\vert_{u=z_{I_0}=f(y)}
\end{equation*}
for the left-hand side of \eqref{recursionexplicit}. For the derivatives of $\main_A$, $\main_B$, and $\main_C$,
we obtain
\begin{align*}
  \main_A &= c_A(1-\rho_{\triangs}^{-1}y)^{m_A} + O\br{(1-\rho_{\triangs}^{-1}y)^{m_A+1/4}},\\
  \main_B &= c_B(1-\rho_{\triangs}^{-1}y)^{m_B} + O\br{(1-\rho_{\triangs}^{-1}y)^{m_B+1/4}},\\
  \main_C &= c_C(1-\rho_{\triangs}^{-1}y)^{m_C} + O\br{(1-\rho_{\triangs}^{-1}y)^{m_C+1/4}}
\end{align*}
with positive constants $c_A,c_B,c_C$ and $m_A=m_B=m_C= e_1-\frac{n+1}{2} = e_2+\frac14$. For the derivatives of
$\error_B$, $\error_C$, and $\error_D$ the exponents $e_B,e_{C,1},\dotsc$ in the considerations above become smaller
by $\frac{n+1}{2}$ as well. By \eqref{an} and the induction hypothesis, each term
\begin{equation*}
\left.\derivn{u}{k}P_{g_0}(y,u,z_{I_0})\derivn{u}{n+1-k}Q(y,u)\right\vert_{u=z_{I_0}=f(y)}
\end{equation*}
for $k<n$ is of the form $\ol{c}(k)(1-\rho_{\triangs}^{-1}y)^{e_1-(n+1)/2}+O\br{(1-\rho_{\triangs}^{-1}y)^{e_1-(n+1)/2+1/4}}$
with $\ol{c}(k)>0$. Since $\deriv{u}Q(y,u)\vert_{u=f(y)} = \ol{c}(1-\rho_{\triangs}^{-1}y)^{1/4}+O\br{(1-\rho_{\triangs}^{-1}y)^{1/2}}$
with $\ol{c}<0$, \eqref{rec-quasi} follows.
\end{proof}

By \Cref{inductionbase} and setting $I=\emptyset$ in \eqref{rec-triang} we obtain \Cref{simplerefined}.

An analogous result can be obtained for $I$-quasi triangulations where double edges are allowed. We will denote the 
analogous generating functions with hats. To show \Cref{looplessrefined} we will use the same method as for 
\Cref{simplerefined} and will just point out the differences. The analogous equation to \Cref{baseplanar} is
\begin{equation}
   \hat{\quasis}_0(y,u)=1+yu^2\hat{\quasis}_0^2(y,u)+\frac{y(\hat{\quasis}_0(y,u)-1)}{u}-y^2u\hat{\quasis}_0(y,u)-
	y^2u\triangt_0(y)\hat{\quasis}_0(y,u).\label{basedouble}
\end{equation}
This is due to the fact that double edges are allowed and thus the final case in the proof of \Cref{baseplanar} 
does not have to be considered. Conversely, if the root is a loop, the exceptional set is now larger, because the 
single edge is not the only quasi triangulation with root face valency two any more.

By doing the same calculations as in \Cref{inductionbase} for \eqref{basedouble} we obtain the dominant singularity at 
$\frac{2^{1/3}}{3}$ and
\begin{equation*}
    \triangt_0(y)=\frac{1}{8}-\frac{9}{8}\br{1-\rho_{\triangt}^{-1} y}+3\br{1-\rho_{\triangt}^{-1} y}^{3/2}+O\br{\br{1-\rho_{\triangt}^{-1} y}^2}.
\end{equation*}

In the general case we distinguish the same four cases. Case A (\Cref{casea}) works in the same way. In the other three cases (\Cref{caseb,casec,cased}) the error term concerning double edges does not appear and the term dealing with 
loops will be modified similar to the planar case. Thus we obtain an analogous recursion formula and obtain the same results
as in \Cref{finalind} with $(1-\rho_{\triangt}^{-1} y)$ instead of $(1-\rho_{\triangs}^{-1} y)$ but with the same exponent. Thus \Cref{looplessrefined} follows.

\end{document}